\documentclass[leqno]{siamltex704}
\usepackage{amsmath}
\usepackage{graphicx}
\usepackage{mathrsfs}
\usepackage{bm}
\usepackage{float}
\usepackage{amsfonts,amssymb}
\usepackage{dsfont}
\usepackage{pifont}
\usepackage{tikz}
\usetikzlibrary{3d}
\usepackage{wrapfig} 
\usepackage{hyperref}
\usepackage{multirow}
\usepackage{mathtools}
\usepackage{subfigure}
\usepackage{color}

\usepackage{colortbl}

\definecolor{mygray}{gray}{.9}
\definecolor{mypink}{rgb}{.99,.91,.95}
\definecolor{mycyan}{cmyk}{.3,0,0,0}

\newtheorem{remark}{Remark}[section]
\def\S{{\mathfrak s}}
\def\T{{\mathcal T}}
\def\E{{\mathcal E}}

\def\bn{{\bf n}}

\def\bv{{\bf v}}
\def\bu{{\bf u}}
\def\pT{{\partial T}}
\def\3bar{{|\!|\!|}}

\newtheorem{FD-algorithm}{5-Point Finite Difference Algorithm}[section]

\def\bmsigma{\bm{\sigma}}
\def\mbbQ{\mathbb{Q}}

\newcommand{\vertiii}[1]{{\left\vert\kern-0.25ex\left\vert\kern-0.25ex\left\vert #1
    \right\vert\kern-0.25ex\right\vert\kern-0.25ex\right\vert}}

\title{A Locking-Free $P_0$ Finite Element Method for Linear Elasticity Equations on Polytopal Partitions}

\author{Yujie Liu\thanks{Center for Quantum Computing, Peng Cheng Laboratory, Shenzhen 518005, China; Center for Mathematical Sciences, Huazhong University of Science \& Technology, Wuhan, China; (liuyj02@pcl.ac.cn). The research of Liu was partially supported by Guangdong Provincial Natural Science Foundation (No. 2017A030310285), Shandong Provincial natural Science Foundation (No. ZR2016AB15) and Youthful Teacher Foster Plan Of Sun Yat-Sen University (No. 171gpy118),} \and Junping Wang
\thanks{Division of Mathematical Sciences, National Science Foundation, Alexandria, VA 22314 (jwang@nsf.gov). The research of Wang was supported by the NSF IR/D program, while working at National Science Foundation. However, any opinion, finding, and conclusions or recommendations expressed in this material are those of the author and do not necessarily reflect the views of the National Science Foundation.}}
\begin{document}

\maketitle
\begin{abstract}
This article presents a $P_0$ finite element method for boundary value problems for linear elasticity equations. The new method makes use of piecewise constant approximating functions on the boundary of each polytopal element, and is devised by simplifying and modifying the weak Galerkin finite element method based on $P_1/P_0$ approximations for the displacement. This new scheme includes a tangential stability term on top of the simplified weak Galerkin to ensure the necessary stability due to the rigid motion. The new method involves a small number of unknowns on each element; it is user-friendly in computer implementation; and the element stiffness matrix can be easily computed for general polytopal elements. The numerical method is of second order accurate, locking-free in the nearly incompressible limit, ease polytopal partitions in practical computation. Error estimates in $H^1$, $L^2$, and some negative norms are established for the corresponding numerical displacement.  Numerical results are reported for several 2D and 3D test problems, including the classical benchmark Cook's membrane problem in two dimensions as well as some three dimensional problems involving shear loaded phenomenon. The numerical results show clearly the simplicity, stability, accuracy, and the efficiency of the new method.
\end{abstract}



\pagestyle{myheadings}

%

\section{Introduction}
In this paper, we are concerned with the development of new and robust finite element methods for boundary value problems of linear elasticity equations with general polytopal finite element partitions. The kinematic model of linear elasticity seeks a displacement vector field $\bm{u}$ satisfying
\begin{eqnarray}
-\nabla\cdot \bmsigma (\bm{u}) &=& \bm{f}\quad {\rm in}\  \Omega,  \label{elasbdy}\\
\bm{u}&=&\bm{g}\quad {\rm on}\ \Gamma_D, \label{elasbcd}\\
\bm{\sigma} \bm{n}&=&\bm{\varrho}\quad {\rm on}\ \Gamma_N, \label{elasbcn}
\end{eqnarray}
where $\Omega$ is a bounded polytopal domain in $\mathbb{R}^d \;(d\ge 2)$ with boundary $\Gamma=\partial\Omega=\Gamma_D \bigcup \Gamma_N$. Assume the Dirichlet boundary $\Gamma_D$ has positive measure in $\mathbb{R}^{d-1}$ and $\Gamma_N=\Gamma/\Gamma_D$. $\bm{\sigma}(\bm{u})$ is the symmetric Cauchy stress tensor. For linear, homogenous, and isotropic materials, the Cauchy stress tensor is given by
\begin{eqnarray}
\bm{\sigma} (\bm{u})=2\mu \varepsilon(\bm{u}) +\lambda(\nabla \cdot \bm{u}) \text{\uppercase\expandafter{\romannumeral1}},
\end{eqnarray}
where $\varepsilon(\bm{u})=\frac{1}{2}(\nabla \bm{u} + \nabla \bm{u}^T)$ is the linear strain tensor, $\mu$ and $\lambda$ are the Lam\'e constants. For linear plane strain, the Lam\'e constants are given by
\begin{eqnarray}
\lambda =\frac{E \nu}{(1+\nu)(1-2\nu)},\quad \mu= \frac{E }{2(1+\nu)},
\end{eqnarray}
where $E$ is Young's modulus and $\nu$ is Poisson ratio.

For the linear elasticity problem, numerical methods may experience ``locking" in that when the Poisson ratio $\nu$ is close to $\frac{1}{2}$ (i.e., when the material is nearly incompressible), poor convergence or no convergence in the displacements may be observed for the corresponding numerical solutions. Locking occurs because for the limit case of $\nu=\frac12$, the exact displacement of the elasticity problem must
satisfy the equation $\nabla\cdot\bm{u}=0$. A common approach of designing locking-free numerical schemes involves the {\it inf-sup} condition of Babu\u{s}ka \cite{babuska} and Brezzi \cite{brezzi} for the approximating functions for a reformulated elasticity problem which seeks $\bm{u}\in[H^1(\Omega)]^d$ and $p\in L^2(\Omega)$ such that $\bm{u}=\bm{g}$ on $\Gamma_D$ and $(2\mu\varepsilon(\bm{u})+p I) \bm{n}=\bm{\varrho}$ on $\Gamma_N$, and satisfying
\begin{eqnarray}
-\nabla\cdot(2\mu\varepsilon(\bm{u})) - \nabla p &=&\bm{f}, \quad \text{ in } \Omega, \label{eq:general_stokes.1}\\
\nabla\cdot\bm{u}-\lambda^{-1}p &=&0,\quad \text{ in } \Omega. \label{eq:general_stokes.2}
\end{eqnarray}
The system \eqref{eq:general_stokes.1}-\eqref{eq:general_stokes.2} is a generalized Stokes problem (see
\cite{Brezzi-Fortin, Girault, abd1984} and the references cited
therein) in which $p=\lambda \nabla \cdot \bm{u}$ acts as a pseudo-pressure. It is clear that any finite elements that are stable for the Stokes problem would provide a locking-free scheme for the linear elasticity problem, but at the cost of solving a saddle-point problem with an additional variable $p$.

``Locking" is a subtle issue when conforming finite elements are employed to approximate the linear elasticity equation. In 1983, no locking was shown to be possible for the $p$-version of the finite element method on smooth domains, see Vogelius \cite{vogelius}. Scott and Vogelius \cite{scott-vogelius} showed that locking is absence for polynomials of degree $k\ge 4$ on triangular partitions. On the other hand, it was shown in \cite{babuska-suri} by Babu\v{s}ka and Suri that, for conforming finite element methods, locking cannot be avoided on quadrilateral meshes for any polynomial of degree $k\ge 1$.

In discontinuous finite element methods, it was shown in \cite{larson} by Hansbo and Larson that the numerical solutions from a
discontinuous Galerkin is locking-free for any values of
$k\ge 1$, see also \cite{Wihler} for the case of $k=1$. 
Di Pietro and Nicaise \cite{Daniele.01} devised a locking-free discontinuous Galerkin scheme for composite materials featuring quasi-incompressible and compressible sections. In \cite{Ferdinando_2005}, a Mixed-Enhanced Strain finite element method was developed within the context of the mixed formulation (in terms of $\bm{u}$ and $p$) for nearly incompressible linear elasticity problems. In \cite{Gain_2014}, a virtual element method was developed for three-dimensional linear elasticity problems on arbitrary polyhedral meshes. In \cite{FuCockburn_2015,Soon_Cockburn_2009}, linear elasticity equations were studied by using HDG methods. In \cite{WangWang_2016}, the authors developed a locking-free weak Galerkin finite element method for the linear elasticity equation by using $P_k/P_\ell$ elements (i.e., $P_k$ on each element and $P_\ell$ on the element boundary) where $\ell=k-1$ or $k$ for $k>1$ and $\ell=k$ for $k=1$. In \cite{Sevilla_linearelastostatic_2019}, a locking-free face-centred finite volume method was proposed for linear elastostatics based on a mixed hybrid formulation as a system of first-order equations. In the realm of discontinuous finite element methods, ``locking" is no longer a difficult issue to resolve for linear elasticity equations.

The main objective of this paper is to push the boundaries of the
weak Galerkin finite element method \cite{WangWang_2016} when the $P_1/P_0$ element is employed for the displacement variable. This is a case that was not included in the stability and convergence theory developed in \cite{WangWang_2016}, as one of the key assumptions targeted for the rigid motion is not satisfied by the $P_1/P_0$ element. But our extensive numerical experiments suggest that the $P_1/P_0$ element has a superb convergence performance on {\em true} polygonal or polyhedral elements (i.e., non-triangular elements in 2D or non-tetrahedral elements in 3D). The goal of this study is to establish a mathematical theory for the $P_1/P_0$ weak Galerkin finite element method for the linear elasticity problem \eqref{elasbdy}-\eqref{elasbcn}.

The main steps behind this study are as follows. First of all, we shall derive a simplified weak Galerkin method by eliminating the unknowns associated with the interior of each element, which shall be accomplished by following the framework developed in \cite{LiDanWW, LiuWang_SWG} for the second order elliptic and the Stokes equations. The resulting numerical scheme involves only piecewise constant functions on the element boundaries. This simplified WG method is further stabilized by using a `tangential derivative' term (see \eqref{eq:stabilizer} for details) so that the modified scheme becomes to be stable and accurate on arbitrary polytopal partitions.

Our main contribution of this work is on the development of a stable and locking-free finite element method by using very minimal number of unknowns for the displacement variable. The number of unknowns is given by the number of edges or faces in the corresponding polygonal (2D) or polyhedral (3D) finite element partitions. The numerical displacement is shown to be convergent at the order of $\mathcal{O}(h^2)$ in $L^2$ and $\mathcal{O}(h)$ in $H^1$. This mathematical convergence theory is further validated by some numerical experiments. We would like to emphasize that for true polytopal partitions, the tangential derivative stabilizer can be absent from the numerical scheme, yielding a robust and accurate numerical method for the linear elasticity equation with $P_0$ approximating functions.

The remaining of this paper is organised as follows. In Section \ref{Section:2}, we shall present the $P_0$ finite element method for the
linear elasticity problem in the mixed formulation. In Section \ref{Section:3}, we give an explicit form for the element stiffness matrices of our method. In Section \ref{Section:4}, we shall describe an equivalent formulation for the $P_0$ finite element method in the primitive variable $\bm{u}$ only. Section \ref{Section:Stability} is devoted to a discussion of the stability and solvability of the new method. In Section \ref{Section:Preparation}, we shall prepare ourselves for error estimates by deriving some identities. Section \ref{Section:H1} is devoted to the establishment of an error estimate in a discrete $H^1$-norm. In Section \ref{Section:L2}, we shall use the duality argument to derive an error estimate in the $L^2$ and some negative norms for the displacement variable. In Section \ref{Section:TI}, we shall derive some inequalities useful for error analysis. Finally, in Section \ref{Section:NE}, numerical results are reported to demonstrate the accuracy and locking-free nature of the new method.

Throughout the paper, we follow the usual notation for Sobolev spaces and norms. For any open bounded domain $D\subset \mathbb{R}^d$ ($d$-dimensional Euclidean space) with Lipschitz continuous boundary, we use $\|\cdot\|_{s,D}$ and $|\cdot|_{s,D}$ to denote the norm and seminorm in the Sobolev space $H^s(D)$ for any $s\ge 0$, respectively. The inner product in $H^s(D)$ is denoted by $(\cdot,\cdot)_{s,D}$. The space $H^0(D)$ coincides with $L^2(D)$, for which the norm and the inner product are denoted by $\|\cdot \|_{D}$ and $(\cdot,\cdot)_{D}$,
respectively. When $D=\Omega$, we shall drop the subscript $D$ in the norm and inner product notation.

\section{Numerical algorithm based on the mixed formulation}\label{Section:2}
The variational formulation of the mixed form for the linear elasticity problem \eqref{eq:general_stokes.1}-\eqref{eq:general_stokes.2} seeks $\bm{u} \in [H^{1}(\Omega)]^d $ and $p \in L^2(\Omega)$ such that $\bm{u}|_{\Gamma_D}=\bm{g}$ and satisfying
\begin{eqnarray}
&2(\mu \varepsilon(\bm{u}), \varepsilon(\bm{v})) + (\nabla\cdot \bm{v},p)=(\bm{f},\bm{v}) + \langle \bm{\varrho},\bm{v} \rangle_{\Gamma_N}, \quad &\forall\; \bm{v} \in [H^{1}_{0,\Gamma_D}(\Omega)]^d,\label{elasticity_mixed1}\\
&(\nabla \cdot \bm{u}, q )-(\lambda^{-1}p,q)=0, \quad &\forall\; q \in L^2(\Omega),\label{elasticity_mixed2}
\end{eqnarray}
where $H^{1}_{0,\Gamma_D}(\Omega)$ is the closed subspace of $H^{1}(\Omega)$ with vanishing boundary value on $\Gamma_D$.

\begin{figure}[!h]
\begin{center}
\subfigure[2D Non-convex polygonal element]{\label{Fig.sub1.1.1}
            \includegraphics [width=0.475\textwidth]{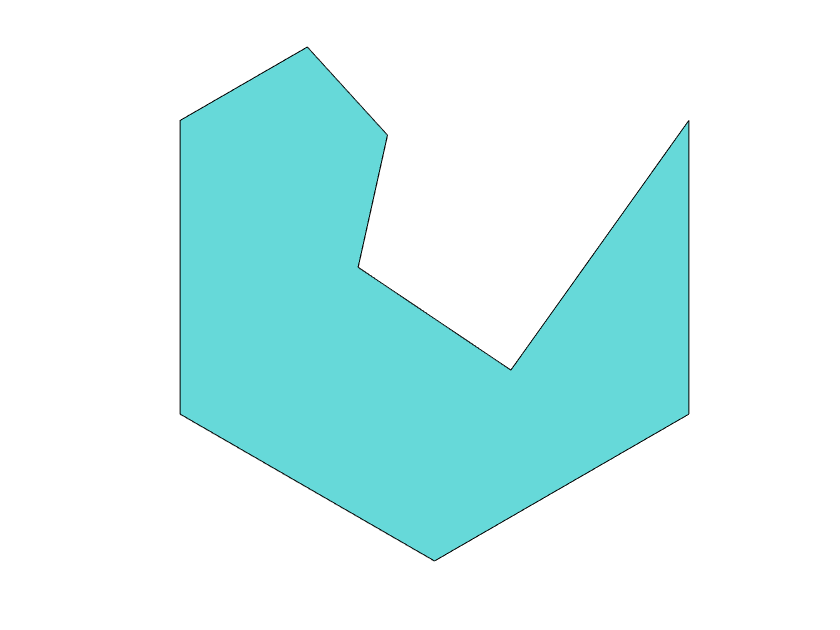}
}
\subfigure[3D polyhedral element]{\label{Fig.sub1.1.2}
            \includegraphics [width=0.475\textwidth]{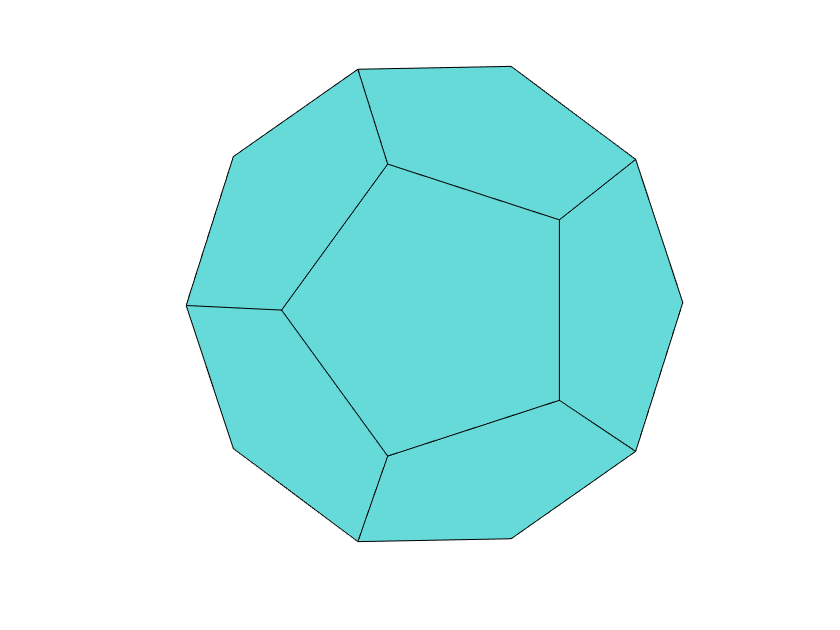}
}
\end{center}
\caption{Illustrative polygonal/polyhedral elements.}
\label{fig.polyheron}
\end{figure}

For simplicity, we consider the case of $2D$ domains. Assume that the domain $\Omega$ is of polygonal type and is partitioned into non-overlapping polygons $\T_h=\{T\}$ that are shape regular (e.g. Figure \ref{Fig.sub1.1.1}). For each $T\in \T_h$, denote by $h_T$ its diameter and by $N$ the number of edges. For each edge $e_i, \ i=1,\ldots, N$, denote by $M_i$ the midpoint and $\bn_i$ the outward normal direction of $e_i$ (see Fig. \ref{fig.polyheron} for an illustration). The meshsize of $\T_h$ is defined as $h=\max_{T\in\T_h} h_T$. Denote by $W(T)$ the space of piecewise constant functions on $\pT$. The global finite element space $W_h(\T_h)$ is constructed by patching together all the local elements $W(T)$ through a single value on each interior edge. The subspace of $W_h(\T_h)$ consisting of functions with vanishing boundary value on $\Gamma_D$ is denoted as $W_h^0(\T_h)$. Moreover, we denote by $RM(T)$ the space of rigid motions on each element $T\in \T_h$,  given by
\begin{eqnarray}
RM(T) =\{a + \eta x: a \in \mathbb{R}^d,\eta\in so(d)\},
\end{eqnarray}
where $x$ is the position vector on $T$ and $so(d)$ is the space of skew-symmetric $d\times d$ matrices.

Let $v_b\in W(T)$ be a piecewise constant function defined on the boundary of $T$, i.e.,
\[
v_b|_{e_i} =v_{b,i},
\]
with $v_{b,i}$ being a constant. The weak gradient of $v_b$ on $T$ is defined as a vector on $T$ such that
\begin{equation}\label{DefWGpoly}
\nabla_w v_b:=\displaystyle\frac{1}{|T|}\sum_{i=1}^N v_{b,i}|e_i|\bf{n_i},
\end{equation}
where $|e_i|$ is the length of the edge $e_i$ and $|T|$ is the area of the element $T$. It is not hard to see that the weak gradient $\nabla_w v_b$ satisfies the following equation:
\begin{equation}\label{DefWGpoly-new}
(\nabla_w v_b, \bm{\phi})_T=\langle v_b, \bm{\phi}\cdot\bn\rangle_\pT
\end{equation}
for all constant vector $\bm{\phi}$. Here and in what follows of the paper, $\langle\cdot,\cdot\rangle_\pT$ stands for the usual inner product in $L^2(\pT)$.

We use the conventional notation of ${P}_j(T)$ for the space of polynomials of degree $j\ge 0$ on $T$. For each $v_b\in W(T)$, we associate it with a linear extension in $T$, denoted as $\S(v_b)\in {P}_1 (T)$, satisfying
\begin{equation}\label{Def.extension}
\sum_{i=1}^{N}(\S(v_b)(M_i) -v_{b,i})\phi(M_i)|e_i|=0,\quad \forall\; \phi\in {P}_1(T).
\end{equation}
It is easy to see that $\S(v_b)$ is well defined by \eqref{Def.extension}, and its computation is local and straightforward. In fact, $\S(v_b)$ can be viewed as an extension of $v_b$ from $\partial T$ to $T$ through a least-squares fitting with respect to a discrete $L^2(T)$ norm.

On each element $T\in \T_h$, we introduce the following bilinear form:
\begin{equation}\label{Def-ST-poly}
\begin{split}
S_T(u_b,v_b):=&\ h^{-1}\sum_{i=1}^N (\S(u_b)(M_i)-u_{b,i})(\S(v_b)(M_i)-v_{b,i})|e_i|\\
            = & \ h^{-1}\langle Q_b^{(0)} \S(u_b)-u_{b},Q_b^{(0)} \S(v_b)- v_{b}\rangle_{\partial T},
            \end{split}
\end{equation}
for $u_b, v_b \in W(T)$, where $Q_b^{(0)}$ is the $L^2$ projection operator onto $W(T)$. The weak gradient and the local stabilizer defined in \eqref{DefWGpoly} and \eqref{Def-ST-poly} can be extended naturally to vector-valued functions. For example, in the two dimensional space, such an extension would imply
\begin{eqnarray}
&&\nabla_w \bm{u}_b :=\begin{bmatrix} \nabla_w u_b\\ \nabla_w v_b \end{bmatrix},\\
&& S_T(\bm{u}_b,\bm{w}_b) := S_T(u_b,w_b) + S_T(v_b,z_b)
\end{eqnarray}
for $\bm{u}_b = \begin{bmatrix} u_b\\ v_b \end{bmatrix} \in [W(T)]^2$ and $\bm{w}_b = \begin{bmatrix} w_b\\ z_b \end{bmatrix} \in [W(T)]^2$.
Using the weak gradient definition, we may define the weak strain tensor as follows:
\begin{eqnarray}
\varepsilon_w(\bm{u}_b)=\frac{1}{2}(\nabla_w \bm{u}_b + \nabla_w \bm{u}_b^{T}).
\end{eqnarray}
Analogously, the weak stress tensor is given by
\begin{eqnarray}
\bm{\sigma}_w(\bm{u}_b)=2\mu \varepsilon_w(\bm{u}_b) +\lambda(\nabla_w \cdot \bm{u} )\text{\uppercase\expandafter{\romannumeral1}}.
\end{eqnarray}

Let us introduce the following bilinear forms:
\begin{eqnarray*}
a(\bm{u}_b,\bm{v}_b)&=&2(\mu \varepsilon_w(\bm{u}_b),\varepsilon_w(\bm{v}_b)) +  S(\bm{u}_b,\bm{v}_b),\\
b(\bm{v}_b, q)&=&(\nabla_w \cdot \bm{v_b},q),\\
d(p,q) &=& \lambda^{-1}(p,q),
\end{eqnarray*}
where $\bm{u}_b, \; \bm{v}_b \in [W_h(\T_h)]^2$ and $p,\; q \in {P}_0 (\T_h)$,  and
\begin{eqnarray}
&&(\mu \varepsilon_w(\bm{u}_b),\varepsilon_w(\bm{v}_b))=\sum_{T\in \T_h}(\mu \varepsilon_w(\bm{u}_b),\varepsilon_w(\bm{v}_b))_T,\\
&& S(\bm{u}_b,\bm{v}_b)=\sum_{T\in \T_h}S_T(\bm{u}_b,\bm{v}_b)_T + \kappa \sum_{e\in \E_h}h_e^\gamma\langle [\nabla_{w,\bm{\tau}} \bm{u}_b ]_e,[\nabla_{w,\bm{\tau}} \bm{v}_b ]_e\rangle_e ,\label{eq:stabilizer}\\
&&(\nabla_w \cdot \bm{v}_b,q)=\sum_{T\in \T_h}(\nabla_w \cdot \bm{v}_b,q)_T.
\end{eqnarray}
Here and in the rest of the paper, $\kappa \geq 0$ is a prescribed stabilization parameter, $\gamma=1$ on interior edge $e$ and $\gamma=2$ on boundary edge $e$, $\nabla_{w,\bm{\tau}} \bm{u}_b = \nabla_{w}\bm{u}_b \cdot \bm{\tau} $ is the tangential component of the weak gradient $\nabla_{w}\bm{u}_b$ with $\bm{\tau}$ being the tangential direction of the edge $e$, and $[\cdot]_e$ represents the jump on edge $e$. When $e$ is on the domain boundary, the jump term
$[\nabla_{w,\bm{\tau}} \bm{u}_b ]_e$ is simply given by $(\nabla_{w,\bm{\tau}} \bm{u}_b)|_e$.

The $P_0$ finite element method for the elasticity equation \eqref{elasticity_mixed1}-\eqref{elasticity_mixed2} seeks $\bm{u}_b \in [W_h(\T_h)]^2$, $p \in {P}_0 (\T_h)$ satisfying $\bm{u}_b|_{\Gamma_d} = Q_b^{(0)} \bm{g}$ and the following equations
\begin{equation}\label{equ.elasticity-Mix-SWG}
\left\{
\renewcommand\arraystretch{2}
\begin{array}{lcl}
a(\bm{u}_b,\bm{v}_b) +b(\bm{v}_b, p) &=& (\bm{f},\S(\bm{v}_b)) + \langle \bm{\varrho}, \bm{v}_b\rangle_{\Gamma_N}\\
b(\bm{u}_b, q) - d(p,q)&=&0
\end{array}
\right.
\end{equation}
for all $\bm{v}_b\in [W_h^0(\T_h)]^2$ and $q \in {P}_0 (\T_h)$.

\begin{remark}
The numerical scheme \eqref{equ.elasticity-Mix-SWG} can be extended to 3D domains with polyhedral partitions. The finite element space will involve piecewise constants on the face of each element; a sample polyhedral element is illustrated in Fig. \ref{Fig.sub1.1.2}. Details are left to interested readers as an exercise.
\end{remark}

\section{Computation of element stiffness matrices}\label{Section:3}
The finite elmenet scheme \eqref{equ.elasticity-Mix-SWG} is user-friendly in computer implementation. In this section, we shall present a formula for the computation of the element stiffness matrix and the element load vector on general 2D polygonal elements. We note that the 3D extension of the element stiffness matrices is straightforward.

Observe that, for $\kappa=0$, the scheme \eqref{equ.elasticity-Mix-SWG} involves no jump term for the weak tangential derivative of the approximating functions so that the corresponding element stiffness matrices are easier to compute than the general case of $\kappa>0$.
For $\kappa>0$, the scheme involves the term $\sum_{e\in \E_h}h_e^\gamma\langle [\nabla_{w,\bm{\tau}}\bm{u}_b],[\nabla_{w,\bm{\tau}} \bm{v}_b]\rangle_e$ which consists of the jump of the tangential derivative on each edge and hence involves two elements in the local stiffness matrix.
Thus, it is effectively a `two-element stiffness matrix' which shall be called `edge stiffness matrix'. In practical computation, the above two types of stiffness matrices must be each computed and assembled in the implementation of the numerical scheme.

\begin{figure}[!h]
\begin{center}
\includegraphics [width=0.475\textwidth]{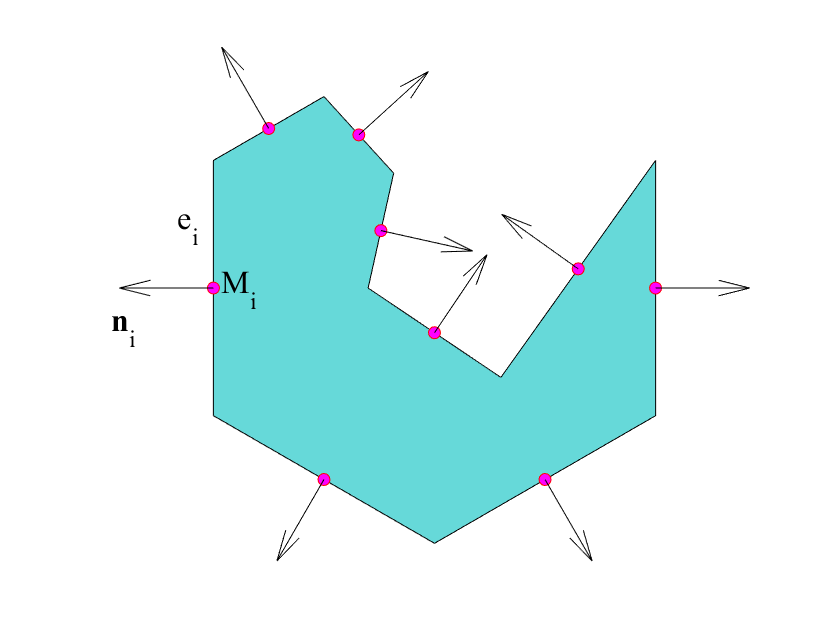}
\end{center}
\caption{A 2D non-convex polygonal element}
\label{fig.polyheron_SM}
\end{figure}

Denote by
$X_{{u}_b}$, $X_{{v}_b}$  and $P$ the vector representation of $u_b$, $v_b$ and $p_h$ given by
\begin{equation}
X_{{u}_b}
=
\begin{bmatrix}\label{Elem_vector}
u_{b,1} \\
u_{b,2} \\
\vdots  \\
u_{b,N} \\
\end{bmatrix},\;
X_{{v}_b}
=
\begin{bmatrix}
v_{b,1} \\
v_{b,2} \\
\vdots  \\
v_{b,N} \\
\end{bmatrix},\;\text{ and } P=[p_h].
\end{equation}
Here $u_{b,i},\, v_{b,i,\;1\leq i\leq N}$ are the values of $u_b,\; v_b$ at the midpoint $M_i$ of edge $e_i$, and $P$ is the value of the numerical pseudo-pressure $p_h$ at the center of $T$. Subsection \ref{SubSection:ESM-1} is devoted to the computation of the element stiffness matrix for scheme \eqref{equ.elasticity-Mix-SWG} with $\kappa=0$, and Subsection \ref{SubSection:ESM-2} shall detail the computation of the `edge stiffness matrix' for the term $\sum_{e\in \E_h}h_e^\gamma\langle [\nabla_{w,\bm{\tau}}\bm{u}_b],[\nabla_{w,\bm{\tau}} \bm{v}_b]\rangle_e$.

\subsection{Element stiffness matrix and element load vector with $\kappa=0$}\label{SubSection:ESM-1}
We have the following theorem for the element stiffness matrix and the load vector for scheme \eqref{equ.elasticity-Mix-SWG} with $\kappa=0$:
\begin{theorem}\label{THM:ESM}
The element stiffness matrix and the element load vector for the SWG scheme \eqref{equ.elasticity-Mix-SWG} when $\kappa =0$ are given in a block matrix form as follows:
\begin{equation}\label{EQ:ESM:01}
\begin{bmatrix*}[l]
- \displaystyle\frac{\mu}{|T|} Q_2 \otimes Q_2  + {A\_B} & ~~\displaystyle\frac{\mu}{|T|} Q_2 \otimes Q_1           & \; Q_1 \\
~~\displaystyle\frac{\mu}{|T|} Q_1 \otimes Q_2          & - \displaystyle\frac{\mu}{|T|} Q_1 \otimes Q_1 + {A\_B}   & \; Q_2\\
~~Q_1^t & ~~Q_2^t  & \; \displaystyle\frac{-|T|}{\lambda} \\
\end{bmatrix*}
\begin{bmatrix}
X_{{u}_b} \\
X_{{v}_b} \\
P \\
\end{bmatrix}
\cong \begin{bmatrix}
F_1 \\
F_2 \\
0 \\
\end{bmatrix}
+
\begin{bmatrix}
G_1 \\
G_2 \\
0 \\
\end{bmatrix},
\end{equation}
where the block components in \eqref{EQ:ESM:01} can be computed explicitly as follows:
\begin{eqnarray*}
&&{A\_B}=A+B,\\
&&A=\{a_{i,j}\}_{N\times N} = h^{-1} (E- EM(M^tEM)^{-1}M^tE), \\
&&B=\{b_{i,j}\}_{N\times N}, \; b_{i,j} = 2\mu\bm{n}_i\cdot\bm{n}_j\displaystyle\frac{|e_i||e_j|}{|T|},\\
&&D=\{d_{i,j}\}_{3\times N}=(M^tEM)^{-1}M^tE,\\
&& Q_1=\big\{n_{x,i} |e_i|\big\}_{N\times 1},\;\ Q_2=\big\{n_{y,i} |e_i|\big\}_{N\times 1},\\
&&F_1=\{f^{(1)}_{i}\}_{N\times 1}, \; \displaystyle f^{(1)}_{i} = \int_{T}f^{(1)}(x,y) (d_{1,i} + d_{2,i}(x-x_T) + d_{3,i}(y-y_T))dT,\\
&&F_2=\{f^{(2)}_{i}\}_{N\times 1}, \; \displaystyle f^{(2)}_{i} = \int_{T}f^{(2)}(x,y) (d_{1,i} + d_{2,i}(x-x_T) + d_{3,i}(y-y_T)) dT,\\
&&G_1=\{g^{(1)}_{i}\}_{N\times 1}, \; \displaystyle g^{(1)}_{i} = \int_{e_i}g^{(1)}(x,y) \delta_{\Gamma_N}dS,\;
\\
&&G_2=\{g^{(2)}_{i}\}_{N\times 1}, \; \displaystyle g^{(2)}_{i} = \int_{e_i}g^{(2)}(x,y) \delta_{\Gamma_N}dS,\;
\end{eqnarray*}
Here $Q_1^t$ and $Q_2^t$ stands for the transpose of $Q_1$ and $Q_2$, respectively. The matrices $M$ and $E$ are given by
\begin{equation*}
M=
\begin{bmatrix}
1      & x_{1} - x_T & y_{1} - y_T\\
1      & x_{2} - x_T & y_{2} - y_T\\
\vdots & \vdots        & \vdots       \\
1      & x_{N} - x_T & y_{N} - y_T\\
\end{bmatrix}_{N\times3},\;
E=
\begin{bmatrix}
|e_1| &       &           &       \\  
      & |e_2| &           &       \\
      &       & \ddots    &       \\
      &       &           & |e_N| \\
\end{bmatrix}_{N\times N},
\end{equation*}
where $M_T=(x_T,y_T)$ is any point on the plane (e.g., the center of $T$ as a specific case), $(x_i, y_i)$ is the midpoint of $e_i$, $|e_i|$ is the length of edge $e_i$, $\bm{n}_i$ is the outward normal vector on $e_i$, and $|T|$ is the area of the element $T$.
\end{theorem}
\begin{proof}
The element stiffness matrix on $T\in\T_h$ consists of four matrices corresponding to the following forms:
$$
a(\bm{u}_b,\bm{v}_b)_T,\; b(\bm{v}_b, p)_T,\; b(\bm{u}_b, q)_T, \text{ and } d(p,q)_T.
$$
Since the stiffness matrices for $b(\bm{v}_b, p)_T,\; b(\bm{u}_b, q)_T, \; d(p,q)_T$ and the load vector $(\bm{f},\S(\bm{v}_b))$ have been fully discussed in \cite{LiuWang_SWG_Stokes_2018}, we shall focus only on the computation of the stiffness matrix corresponding to $a(\bm{u}_b,\bm{v}_b)_T$ in this paper.
From the definition of the bilinear form $a(\bm{u}_b,\bm{v}_b)_T$ for $\kappa =0$, we have:
\begin{equation}\label{EQ:October22:100}
a(\bm{u}_b,\bm{v}_b)_T=  S_T(\bm{u}_b,\bm{v}_b) + 2(\mu \varepsilon_w(\bm{u}_b),\varepsilon_w(\bm{v}_b))_T.
\end{equation}
On each element $T\in\T_h$, it is not hard to check the following identity:
\begin{equation}
\begin{split}
&(2\mu \varepsilon_w(\bu_b), \varepsilon_w(\bv_b))_T\\
=& (2\mu \varepsilon_w(\bu_b), \nabla_w \bv_b)_T \\
=& (\mu(\nabla_w \bu_b + \nabla_w \bu_b^t), \nabla_w \bv_b)_T\\
=& (2\mu\nabla_w\bu_b, \nabla_w\bv_b)_T - (\mu( \nabla_w\bu_b - \nabla_w\bu^t_b), \nabla_w\bv_b)_T \\
=& (2\mu\nabla_w\bu_b, \nabla_w\bv_b)_T - (\mu \nabla_w\times \bu_b, \nabla_w\times\bv_b)_T,
\end{split}
\label{eq.bilinear.stiff.1}
\end{equation}
where the weak curl of $\bm{u}_b$ on $T$ is defined as
\begin{eqnarray}\label{Defcurlpoly}
\nabla_w \times \bm{u}_b
&=&-\displaystyle\frac{1}{|T|}\sum_{i=1}^N \bm{u}_{b,i}\times \bm{n}_i |e_i|
=\displaystyle\frac{1}{|T|}\sum_{i=1}^N (u_{b,i}n_{y,i} - v_{b,i} n_{x,i}) |e_i|.
\end{eqnarray}
Note that the weak curl $\nabla_w v_b$ satisfies the following equation:
\begin{equation}\label{DefWG-curl-new}
(\nabla_w \times \bm{v}_b, \bm{\phi})_T=-\langle \bm{v}_b\times\bn, \bm{\phi}\rangle_\pT
\end{equation}
for all constant vector $\bm{\phi}$.

From $S_T(\bm{u}_b,\bm{w}_b) := S_T(u_b,w_b) + S_T(v_b,z_b)$, the element stiffness matrix corresponding to the bilinear form $S_T(u_b,w_b)$ is given by
\begin{equation}\label{EQ:October22:101}
A=\{a_{i,j}\}_{i,j=1}^N: = \kappa h^{-1}(E- EM(M^tEM)^{-1}M^tE),
\end{equation}
and so is the one for the bilinear form $S_T(v_b,z_b)$.
Moreover, the element stiffness matrix corresponding to $(2\mu\nabla_w u, \nabla_w w)$ is given by
\begin{equation}\label{EQ:October22:102}
B:=\{b_{i,j}\}_{i,j=1}^N,\; b_{i,j} = 2\mu\bm{n}_i\cdot\bm{n}_j\displaystyle\frac{|e_i||e_j|}{|T|}.
\end{equation}
Readers are referred to \cite{LiuWang_SWG_Stokes_2018} for a detailed derivation of \eqref{EQ:October22:101}-\eqref{EQ:October22:102}.

Next, from \eqref{Defcurlpoly} and the equality \eqref{DefWG-curl-new}, we have
\begin{eqnarray}
&&(\mu\nabla_w \times \bu_b,\nabla_w  \times \bv_b )_T\\
&=&(\displaystyle\mu\frac{1}{|T|}\sum_{i=1}^N \bu_{b,i}\times \bm{n}_i|e_i|,\displaystyle\frac{1}{|T|}\sum_{j=1}^N \bv_{b,j}\times \bm{n}_j|e_j|)_{T} \nonumber\\
&=&(\displaystyle\mu\frac{1}{|T|}\sum_{i=1}^N ( u_{b,i}n_{y,i} - v_{b,i} n_{x,i})|e_i|,\displaystyle\frac{1}{|T|}\sum_{j=1}^N (w_{b,j}n_{y,j} - z_{b,j} n_{x,j})|e_j|)_{T} \nonumber\\
&=&\begin{bmatrix} z_{b,j} ~~ w_{b,j} \end{bmatrix}
\begin{bmatrix} a_{i,j}   & b_{i,j} \\ c_{i,j}   & d_{i,j} \end{bmatrix}
\begin{bmatrix} v_{b,i} \\ u_{b,i} \end{bmatrix}
=\begin{bmatrix} w_{b,j} ~~ z_{b,j} \end{bmatrix}
\begin{bmatrix} d_{i,j}   & c_{i,j} \\ b_{i,j}   & a_{i,j} \end{bmatrix}
\begin{bmatrix} u_{b,i} \\ v_{b,i} \end{bmatrix},\nonumber
\end{eqnarray}
with
\begin{eqnarray*}\label{EQ:element-stiffness-matrix-CD}
&& a_{i,j} = ~~(\mu n_{x,i},n_{x,j})_T\displaystyle\frac{|e_i||e_j|}{|T|^2},~~ b_{i,j} = -(\mu n_{y,i},n_{x,j})_T\displaystyle\frac{|e_i||e_j|}{|T|^2},\\
&& c_{i,j} = -(\mu n_{x,i},n_{y,j})_T\displaystyle\frac{|e_i||e_j|}{|T|^2},~~ d_{i,j} = ~~(\mu n_{y,i},n_{y,j})_T\displaystyle\frac{|e_i||e_j|}{|T|^2},~ 1\leq i,j\leq N.
\end{eqnarray*}
Thus, the corresponding element stiffness matrix corresponding to $(\mu \nabla_w\times \bu, \nabla_w\times\bv)$ is given by
$C:= \begin{bmatrix} d_{i,j}  & c_{i,j}\\ b_{i,j}  & a_{i,j} \end{bmatrix}$.
Note that the matrix $\{a_{i,j}\}$ is in fact given by $Q_1 \otimes Q_1$, and similarly for other terms in $C$. Thus, the element stiffness matrix corresponding to $a(\bm{u}_b,\bm{v}_b)_T$ takes the following form:
\begin{equation}\label{SM_sigma}
P_1=
\begin{bmatrix*}[l]
- \displaystyle\frac{\mu}{|T|} Q_2 \otimes Q_2  + A + B & ~~\displaystyle\frac{\mu}{|T|} Q_2 \otimes Q_1          \\
~~\displaystyle\frac{\mu}{|T|} Q_1 \otimes Q_2          & - \displaystyle\frac{\mu}{|T|} Q_1 \otimes Q_1 + A + B
\end{bmatrix*},
\end{equation}
with $Q_1=\big\{n_{x,i} |e_i|\big\}_{N\times 1}$ and $Q_2=\big\{n_{y,i} |e_i|\big\}_{N\times 1}$.
\end{proof}

\begin{figure}[!h]
\begin{center}
\includegraphics [width=0.6\textwidth]{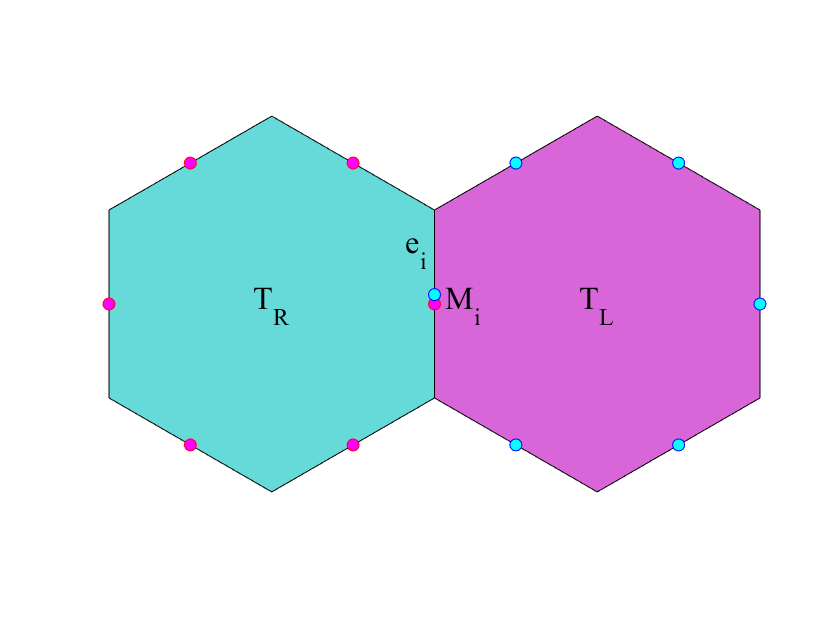}
\end{center}
\caption{Two adjacent polygonal elements with a common edge}
\label{fig.Twopolyheron}
\end{figure}

\subsection{Edge stiffness matrix for $h_e^\gamma\langle [\nabla_{w,\bm{\tau}} \bm{u}_b],[\nabla_{w,\bm{\tau}} \bm{v}_b]\rangle_e$}\label{SubSection:ESM-2}
For any interior edge $e$ of the partition $\T_h$, there exist two adjacent elements that share $e$ in common (see Figure \ref{fig.Twopolyheron}). Observe that the stabilization term $h_e^\gamma\langle [\nabla_{w,\bm{\tau}} \bm{u}_b],[\nabla_{w,\bm{\tau}} \bm{v}_b]\rangle_e$ involves the jump of the weak tangential derivative between the two adjacent elements denoted by $T_R$ and $T_L$. Thus, the corresponding stiffness matrix, called the {\em edge stiffness matrix}, shall include the degrees of freedom from both elements.
Denote by
$X^R_{{u}_b} $, $X^R_{{v}_b}$  and $P^R$  (respectively $X^L_{{u}_b}$, $X^L_{{v}_b}$, $P^L$) the vector representation of $u_b$, $v_b$ and $p_h$ on the element $T_R$ (respectively $T_L$) defined as in the equation \eqref{Elem_vector}.

\begin{theorem}\label{THM:EdgeSM}
The edge stiffness matrix for the term  $h_e^\gamma\langle [\nabla_{w,\bm{\tau}} \bm{u}_b],[\nabla_{w,\bm{\tau}} \bm{v}_b]\rangle_e$ in the SWG scheme \eqref{equ.elasticity-Mix-SWG} has the following representation as a block matrix:
\begin{equation}\label{EQ:EdgeSM:01}
\begin{matrix}
\cellcolor{red!15}h_e^\gamma|e|
\end{matrix}
\left[
\begin{bmatrix*}[l]
\cellcolor{red!15}~~\mathbb{Q} &\cellcolor{red!15} & & \\
\cellcolor{red!15}~~&\cellcolor{red!15}&& \\
\cellcolor{red!00}~~&&\cellcolor{red!15}~~\mathbb{Q}&\cellcolor{red!15} \\
\cellcolor{red!00}~~&&\cellcolor{red!15}~~&\cellcolor{red!15} \\
\end{bmatrix*}
\right.
\left[
\begin{bmatrix}
\cellcolor{blue!15}X^R_{{u}_b} \\
\cellcolor{blue!15}X^L_{{u}_b} \\
\cellcolor{blue!15}X^R_{{v}_b} \\
\cellcolor{blue!15}X^L_{{v}_b}
\end{bmatrix}
\right.,
\end{equation}
where the block matrix $\mathbb{Q}$ is given by
\begin{equation}\label{EQ:EdgeSM:01:200}
\mathbb{Q}_=
[Q^R_1 \;  -Q^L_1 ]
\begin{bmatrix}
 ~~Q^R_1\\
 -Q^L_1
\end{bmatrix}
+
[Q^R_2 \;  -Q^L_2 ]
\begin{bmatrix}
 ~~Q^R_2\\
 -Q^L_2
\end{bmatrix},
\end{equation}
with
\begin{eqnarray*}
&&Q^R_1=\bm{\tau}_x\{q_{i}\}_{N\times 1}, \; q_{i} = |e_{i,R}|
  \bm{n}_{x,i_R},\\
&&Q^R_2=\bm{\tau}_y\{q_{i}\}_{N\times 1}, \; q_{i} = |e_{i,R}|
  \bm{n}_{y,i_R},
\end{eqnarray*}
and
\begin{eqnarray*}
&&Q^L_1=\bm{\tau}_x\{q_{i}\}_{N\times 1}, \; q_{i} =  |e_{i,L}|
  \bm{n}_{x,i_L},\\
&&Q^L_2=\bm{\tau}_y\{q_{i}\}_{N\times 1}, \; q_{i} =  |e_{i,L}|
  \bm{n}_{y,i_L}.
\end{eqnarray*}
\end{theorem}

\begin{proof}
For any interior edge $e$ of the partition $\T_h$, we have
\begin{eqnarray}
&&h_e^\gamma\langle [\nabla_{w,\bm{\tau}}  \bm{u}_b],[\nabla_{w,\bm{\tau}} \bm{v}_b]\rangle_e \nonumber\\
&=&h_e^\gamma \langle ({\nabla_{w,\bm{\tau}} \bm{u}_b}_{T_{L}} - {\nabla_{w,\bm{\tau}} \bm{u}_b}_{T_R}, {\nabla_{w,\bm{\tau}} \bm{v}_b}_{T_{L}} - {\nabla_{w,\bm{\tau}} \bm{v}_b}_{T_R})\rangle_e \nonumber \\
&=&h_e^\gamma \langle
\begin{bmatrix}
\nabla_{w} {u}_b \cdot \bm{\tau}\\
\nabla_{w} {v}_b \cdot \bm{\tau}
\end{bmatrix}_{T_{L}}
-
\begin{bmatrix}
\nabla_{w} {u}_b \cdot \bm{\tau}\\
\nabla_{w} {v}_b \cdot \bm{\tau}
\end{bmatrix}_{T_{R}},
\begin{bmatrix}
\nabla_{w} {w}_b \cdot \bm{\tau}\\
\nabla_{w} {z}_b \cdot \bm{\tau}
\end{bmatrix}_{T_{L}}
-
\begin{bmatrix}
\nabla_{w} {w}_b \cdot \bm{\tau}\\
\nabla_{w} {z}_b \cdot \bm{\tau}
\end{bmatrix}_{T_{R}}\rangle \nonumber\\
&=&\displaystyle h_e^\gamma \langle
\begin{bmatrix}
\displaystyle\frac{1}{|T_R|} \displaystyle\sum_{i_R=1}^{N_{R}} u_{b,i_R}\bm{n}_{i_R} \cdot \bm{\tau}\\
\displaystyle\frac{1}{|T_R|} \displaystyle\sum_{i_R=1}^{N_{R}} v_{b,i_R}\bm{n}_{i_R} \cdot \bm{\tau}
\end{bmatrix}
-
\begin{bmatrix}
\displaystyle\frac{1}{|T_L|} \displaystyle\sum_{i_L=1}^{N_{L}} u_{b,i_L}\bm{n}_{i_L} \cdot \bm{\tau}\\
\displaystyle\frac{1}{|T_L|} \displaystyle\sum_{i_L=1}^{N_{L}} v_{b,i_L}\bm{n}_{i_L} \cdot \bm{\tau}
\end{bmatrix}, \nonumber\\
&&~~~~ \displaystyle
\begin{bmatrix}
\displaystyle\frac{1}{|T_R|} \displaystyle\sum_{i_R=1}^{N_{R}} w_{b,i_R}\bm{n}_{i_R} \cdot \bm{\tau}\\
\displaystyle\frac{1}{|T_R|} \displaystyle\sum_{i_R=1}^{N_{R}} z_{b,i_R}\bm{n}_{i_R} \cdot \bm{\tau}
\end{bmatrix}
-
\begin{bmatrix}
\displaystyle\frac{1}{|T_L|} \displaystyle\sum_{i_L=1}^{N_{L}} w_{b,i_L}\bm{n}_{i_L} \cdot \bm{\tau}\\
\displaystyle\frac{1}{|T_L|} \displaystyle\sum_{i_L=1}^{N_{L}} z_{b,i_L}\bm{n}_{i_L} \cdot \bm{\tau}
\end{bmatrix}\rangle \label{EdgeMatrix}\\
&=&h_e^\gamma |e|
\begin{bmatrix} W & Z \end{bmatrix}
\left[
\begin{bmatrix*}[l]
\cellcolor{red!05}~~\mathbb{Q} &\cellcolor{red!05} & & \\
\cellcolor{red!05}~~&\cellcolor{red!05}&& \\
\cellcolor{red!00}~~&&\cellcolor{red!05}~~\mathbb{Q}&\cellcolor{red!05} \\
\cellcolor{red!00}~~&&\cellcolor{red!05}~~&\cellcolor{red!05} \\
\end{bmatrix*}
\right.
\begin{bmatrix}
   U^T\\
   V^T
\end{bmatrix} , \nonumber
\end{eqnarray}
where
$U=\begin{bmatrix}u_{b,i_R}, & u_{b,i_L}\end{bmatrix}$,
$V=\begin{bmatrix}v_{b,i_R}, & v_{b,i_L}\end{bmatrix}$,
$W=\begin{bmatrix}w_{b,j_R}, & w_{b,j_L}\end{bmatrix}$,
and $Z=\begin{bmatrix}z_{b,j_R}, & z_{b,j_L}\end{bmatrix}$. This completes the derivation of the matrix representation \eqref{EQ:EdgeSM:01} and \eqref{EQ:EdgeSM:01:200}. Note that the degree of freedom on the common edge $e$ was treated as independent local variables in the matrix formula \eqref{EQ:EdgeSM:01} and \eqref{EQ:EdgeSM:01:200}; the shared nature of the edge unknown can be easily taken care of in the matrix assembling process through a local-to-global map of the local variables.
\end{proof}

\section{An equivalent algorithm based on the primal formulation}\label{Section:4}
A weak formulation for \eqref{elasbdy}-\eqref{elasbcn} in the primal form is given by seeking $\bm{u} \in [H^1(\Omega)]^d$ satisfying $\bm{u}=\bm{g}$ on $\Gamma_D$ such that
\begin{eqnarray}
2(\mu \varepsilon(\bm{u}), \varepsilon(\bm{v})) + (\lambda \nabla\cdot\bm{u}, \nabla\cdot\bm{v}) =(\bm{f},\bm{v}) + \langle \varrho, \bm{v}\rangle_{\Gamma_N}, \quad \forall \; \bm{v}\in [H^1_{0,\Gamma_D}(\Omega)]^d.
\end{eqnarray}
The primitive $P_0$ finite element method for the \eqref{elasbdy}-\eqref{elasbcn} {\em seeks $\bm{u}_b \in [W_h(\T_h)]^d$  with $\bm{u}_b=Q_b^{(0)}(\bm{g})$ on $\Gamma_D$ such that}
\begin{eqnarray}\label{SWG.Primal}
a_s(\bm{u}_b, \bm{v}_b)=(\bm{f}, \S(\bm{v}_b)) + \langle \varrho, \bm{v}_b\rangle_{\Gamma_N}, \quad \forall \;\bm{v}_b \in [W_h^0(\T_h)]^d,
\end{eqnarray}
where the bilinear form is given by
\begin{eqnarray}
 a_s(\bm{u}_b, \bm{v}_b) &=&  S(\bm{u}_b,\bm{v}_b) + \sum_{T\in \T_h}2(\mu \varepsilon_w(\bm{u}_b),\varepsilon_w(\bm{v}_b))_T + (\lambda \nabla_w \cdot \bm{u}_b, \nabla_w \cdot \bm{v}_b)_T. \nonumber
\end{eqnarray}

\begin{lemma}
The $P_0$ finite element algorithms \eqref{SWG.Primal} and \eqref{equ.elasticity-Mix-SWG} are equivalent in the sense that the solution $\bm{u}_b$ from the two algorithms are identical to each other.
\end{lemma}

\begin{proof}
Assume that ${\tilde{\bm{u}}}_b$ and $\tilde p_h$ solves \eqref{equ.elasticity-Mix-SWG}.
Note that the second equation of \eqref{equ.elasticity-Mix-SWG} can be rewritten as
\begin{eqnarray}
(\nabla_w \cdot \tilde{\bm{u}}_b,q)-\lambda^{-1}(\tilde{p}_h,q) =0, \qquad\forall q \in P_0(\T_h).
\end{eqnarray}
Since $\nabla_w \cdot \tilde{\bm{u}}_b\in P_0(\T_h)$, then $\tilde p_h$ can be solved as
\begin{eqnarray}\label{ph}
\tilde{p}_h = \lambda \nabla_w \cdot \tilde{\bm{u}}_b.
\end{eqnarray}
Substituting \eqref{ph} into the first equation of \eqref{equ.elasticity-Mix-SWG} yields exactly the same equation as \eqref{SWG.Primal}.
From the solution uniqueness, the solution $\tilde{\bm{u}}_h$ of \eqref{equ.elasticity-Mix-SWG} is then identical to the numerical solution of the Algorithm \eqref{SWG.Primal}.
\end{proof}

\medskip
The reformulation of the elasticity problem as a generalized Stokes system with nonzero divergence constraint may lead to numerical approximations that are locking-free as $\lambda \rightarrow \infty$, provided that some necessary stability conditions are satisfied. In Section \ref{Section:Stability}, we shall establish these
stability conditions for the mixed weak Galerkin algorithm
\eqref{equ.elasticity-Mix-SWG} so that the Algorithm \eqref{SWG.Primal} is locking-free from the equivalence of \eqref{equ.elasticity-Mix-SWG} and the primal formulation \eqref{SWG.Primal}.

\section{Stability and well-posedness}\label{Section:Stability}
This section is devoted to a stability analysis for the mixed weak Galerkin finite element method \eqref{equ.elasticity-Mix-SWG} by establishing an {\it inf-sup} condition necessary in the theory of Babu\u{s}ka \cite{babuska} and Brezzi \cite{brezzi} for saddle-point problems. The result provides a locking-free convergence for the numerical scheme \eqref{SWG.Primal} for the linear elasticity problem in the displacement formulation.

In the finite element space $[W_h(\T_h)]^d$, we introduce the following semi-norm:
\begin{equation}\label{EQ:Semi-norm}
\3bar \bm{u}_b\3bar^2 =\displaystyle \sum_{T \in T_h } \|\varepsilon_w(\bm{u}_b)\|^2_T + h^{-1}_{T}\|Q_b^{(0)}\S(\bm{u}_b) - \bm{u}_b \|^2_{\partial T} + \sum_{e\in \E_h} h^\gamma_e[\nabla_{w,\bm{\tau}} \bm{u}_b]_e^2.
\end{equation}

\begin{lemma}\label{LEMMA:October:20:01}
The semi-norm $\3bar \cdot \3bar$, as defined in \eqref{EQ:Semi-norm}, is a norm in the linear subspace $[W_h^0(\T_h)]^d$.
\end{lemma}

\begin{proof}
We shall only verify the positivity property for $\3bar \cdot \3bar$.
To this end, assume $\3bar \bm{u}_b \3bar = 0$ for some $ \bm{u}_b \in [W_h^0(\T_h)]^d$.
It follows that
\begin{eqnarray}\label{EQ:Oct:19:601}
&&\varepsilon_w(\bm{u}_b)=0, \mbox { in } T,\\
&&Q_b^{(0)} \S(\bm{u}_b)-\bm{u}_b =0,  \mbox{ on } \partial T,\label{EQ:Oct:19:602}\\
&&[\nabla_{w,\bm{\tau}} \bm{u}_b]\big|_e = 0, \mbox{ on each edge } e\subset\pT. \label{EQ:Oct:19:603}
\end{eqnarray}
From \eqref{EQ:Oct:19:602}, we have $( \S(\bm{u}_b)(M_i) - \bm{u}_{b,i})|e_i|=\bm{0}$ at each edge center. Thus
\begin{eqnarray}
&&~~~~~~\displaystyle \nabla \S(\bm{u}_b)=\frac{1}{|T|}\sum_{i=1}^N \S(\bm{u}_b)(M_i) \otimes\bm{n}_i |e_i|=\frac{1}{|T|}\sum_{i=1}^N \bm{u}_{b,i}\otimes\bm{n}_i |e_i|=\nabla_w \bm{u}_b, \\
&&~~~~~~\displaystyle \nabla_{\bm{\tau}}\S(\bm{u}_b)=\frac{1}{|T|}\sum_{i=1}^N |e_i| \S(\bu_{b})(M_i) \otimes \bm{n}_i \cdot \bm{\tau}=\frac{1}{|T|}\sum_{i=1}^N |e_i| \bu_{b,i} \otimes \bm{n}_i \cdot \bm{\tau} = \nabla_{w,\bm{\tau}} \bm{u}_b,
\end{eqnarray}
where $\bm{x}\otimes\bm{y} = \{x_iy_j\}_{2\times 2}$ for any two vectors $\bm{x}=(x_1,x_2)$ and $\bm{y}=(y_1,y_2)$ in 2D.
That is $\nabla \S(\bm{u}_b) =\nabla_w \bm{u}_b$, which gives $\varepsilon_w(\S(\bm{u}_b))=\varepsilon_w(\bm{u}_b)=0$ so that $\S(\bm{u}_b)\in RM(T)\subset [P_1(T)]^d$.
Also from $\S(\bm{u}_b)(M_i)= \bm{u}_{b,i}$ and
$[\nabla_{\bm{\tau}} \S(\bm{u}_b)]= [\nabla_{w,\bm{\tau}} \bm{u}_b]=0$ on each edge, we see that $\S(\bm{u}_b)$ is continuous across each edge of the partition $\T_h$ and $\S(\bm{u}_b)\big|_{\Gamma_D} = 0$.

For any interior edge $e\in\E_h$, let $T_1$ and $T_2$ be the two elements that share $e$ in common. As $\S(\bm{u}_b)\in RM(T)$, it follows that $\S(\bm{u}_b)\big|_{T_1}$ and $\S(\bm{u}_b)\big|_{T_2}$ have the following representation:
\begin{eqnarray}\label{EQ:October:20:101}
\S(\bm{u}_b)\big|_{T_1}= \bm{a} + \bm{\eta} x,\\
\S(\bm{u}_b)\big|_{T_2}= \bm{b} + \bm{\xi} x,\label{EQ:October:20:102}
\end{eqnarray}
where $\bm{a}, \bm{b}\in\mathbb{R}^d $ and $\bm{\eta},  \bm{\xi} \in \mathbb{R}^{d\times d}$ are skew-symmetric. We claim that the following holds true:
\begin{equation}\label{EQ:October:20:100}
\bm{a}=\bm{b},\ \ \bm{\eta}=\bm{\xi}.
\end{equation}
If so, then we have $\S(\bm{u}_b)\in RM(\Omega)$ and $\S(\bm{u}_b)\big|_{\Gamma_D} = 0$. It follows that $\S(\bm{u}_b)\equiv 0$ in $\Omega$, and so does $\bm{u}_b\equiv 0$.

To verify \eqref{EQ:October:20:100}, consider the case of $d=2$ for simplicity of presentation. From the representation \eqref{EQ:October:20:101}-\eqref{EQ:October:20:102} we obtain
\begin{eqnarray}
\nabla_{\bm{\tau}} \S(\bm{u}_b)_{T_1} - \nabla_{\bm{\tau}} \S(\bm{u}_b)_{T_2}
&=&
\nabla_{\bm{\tau}}
\begin{bmatrix}
a_x + \eta_2 y \\
a_y - \eta_2 x \\
\end{bmatrix}
-
\nabla_{\bm{\tau}}
\begin{bmatrix}
b_x + \xi_2 y \\
b_y - \xi_2 x \\
\end{bmatrix} \nonumber \\
&=&
\begin{bmatrix}
~~ \eta_2 \bm{\tau}_y\\
-\eta_2 \bm{\tau}_x
\end{bmatrix}
 -
 \begin{bmatrix}
~~ \xi_2\bm{\tau}_y,\\
-\xi_2\bm{\tau}_x
\end{bmatrix},\nonumber
\end{eqnarray}
where $\bm{\tau}=(\bm{\tau}_x, \bm{\tau}_y)^t
$ is the unit tangent direction of the edge $e$, and
$$
\bm{\eta} =\begin{bmatrix}
0 & \eta_2\\
-\eta_2 & 0
\end{bmatrix},\ \  \bm{\xi} =\begin{bmatrix}
0 & \xi_2\\
-\xi_2 & 0
\end{bmatrix}
$$
are two skew-symmetric matrices in 2D. It follows that $\eta_2=\xi_2$, which together with $\S(\bm{u}_b)\big|_{T_1\cap e}=\S(\bm{u}_b)\big|_{T_2\cap e}$ leads to $a_x=b_x$ and $a_y=b_y$.  This completes \eqref{EQ:October:20:100} and furthermore the proof of the lemma.
\end{proof}

\begin{lemma}\label{lem:ellip}
There exist positive constants $\alpha_1$ and $\alpha_2$ such that
\begin{eqnarray}\label{elliptic}
\alpha_1 \3bar \bm{u}_b \3bar^2 \leq a(\bm{u}_b,\bm{u}_b) \leq \alpha_2 \3bar \bm{u}_b \3bar^2, \; \quad \forall \;\bm{u}_b \in [W_h^0(\T_h)]^2.
\end{eqnarray}
\end{lemma}
\begin{proof}
By definition, we have
\begin{eqnarray*}
a(\bm{u}_b,\bm{u}_b)
&=&2(\mu \varepsilon_w(\bm{u}_b),\varepsilon_w(\bm{u}_b)) + S(\bm{u}_b,\bm{u}_b)\\
&=&\sum_{T\in \T_h}2{\mu} \|\varepsilon_w(\bm{u}_b)\|^2_T + h^{-1}_{T}\|Q_b^{(0)}\S(\bm{u}_b) - \bm{u}_b \|^2_{\partial T}
+ \kappa\sum_{e\in \E_h} h_e^\gamma \|[\nabla_{w,\bm{\tau}} \bm{u}_b]\|_e^2.
\end{eqnarray*}
Since ${\mu}$ and $\kappa$ are bounded from below and above, it is not hard to see that \eqref{elliptic} is verified for all $\bm{u}_b \in [W_h^0(\T_h)]^2$.
\end{proof}

\begin{lemma}\label{lem:inf_sup}
There exists a constant $\beta >0$ such that
\begin{eqnarray}
\sup_{\bm{u}_b \in [W_h^0(\T_h)]^2, \bm{u}_b\neq 0}\frac{b(\bm{u}_b,q)}{\3bar \bm{u}_b\3bar}
\geq \beta \| q \|_0, \quad \forall \; q \in P_0(\T_h).
\end{eqnarray}
\end{lemma}
\begin{proof}
Consider the following auxiliary problem:
\begin{equation}\label{inf-sup.1}
\left\{
\begin{array}{rll}
\nabla\cdot\bm{\varphi}&=&q, \quad \text{ in } \Omega, \\
\bm{\varphi}&=&0, \quad \text{ on } \Gamma_D.
\end{array}
\right.
\end{equation}
The problem \eqref{inf-sup.1} has a solution $\bm{\varphi} \in [H_{0,\Gamma_D}^1 (\Omega)]^d$. By setting $\bm{u}_b = \bm{v}_q:=Q_b^{(0)} \bm{\varphi}$,
we arrive at
\begin{eqnarray*}
 b(\bm{u}_b, q)& = &\sum_{T\in\T_h} (\nabla_w \cdot \bm{v}_q,q)_T= \sum_{T\in\T_h} (\nabla_w \cdot(Q_b^{(0)}\bm{\varphi}), q)_T \\
 &=& \sum_{T\in\T_h} \langle Q_b^{(0)}\bm{\varphi} \cdot \bn , q\rangle_{\partial T}
  = \sum_{T\in\T_h} \langle \bm{\varphi} \cdot \bn , q\rangle_{\partial T}\\
 &=& \sum_{T\in\T_h} (\nabla \cdot\bm{\varphi} ,q)_T
  = \| q \|^2_0.
\end{eqnarray*}
Furthermore, it is not hard to see that there exist a constant $C_0$ such that
\begin{eqnarray}
\3bar \bm{v}_q \3bar  \leq C_0 \| q \|_0.
\end{eqnarray}
In fact, by definition
\begin{eqnarray}
\3bar \bm{v}_q \3bar^2=\displaystyle \sum_{T \in T_h } \|\varepsilon_w(\bm{v}_q)\|^2_T + h^{-1}_{T}\|Q_b^{(0)}\S(\bm{v}_q) - \bm{v}_q \|^2_{\partial T} + \sum_{e\in \E_h} h_e^\gamma\|[\nabla_{w,\bm{\tau}} \bm{v}_q]\|_e^2,
\end{eqnarray}
and the following estimates hold true:
\begin{eqnarray}
\sum_{T \in T_h } \|\varepsilon_w(\bm{v}_q)\|^2_T &=& \sum_{T \in T_h }\|\frac{1}{2}(\nabla_w \bm{v}_q + \nabla_w \bm{v}_q^T) \|^2_T \nonumber\\
& \leq &  C \sum_{T \in T_h }\|\nabla_w \bm{v}_q\|^2_T  \leq C \sum_{T \in T_h }\|\nabla_w Q_b^{(0)}\bm{\varphi}\|^2_T \nonumber\\
& \leq &  C \sum_{T \in T_h }\|\mbbQ_h^{(0)}\nabla \bm{\varphi}\|^2_T
  \leq   C \sum_{T \in T_h }\|\nabla \bm{\varphi}\|^2_T \nonumber\\
& \leq & C\| q \|_0, ~~~ \text{ by the regularity assumption for \eqref{inf-sup.1}},\label{inf-sup-equation.1}~~~~~~
\end{eqnarray}

\begin{eqnarray*}
~~ \|Q_b^{(0)}\S(\bm{v}_q)- \bm{v}_q\|^2_{\partial T}
&=&\langle Q_b^{(0)}\S(Q_b^{(0)} \bm{\varphi}) - Q_b^{(0)} \bm{\varphi}, Q_b^{(0)}\S(Q_b^{(0)} \bm{\varphi}) - Q_b^{(0)} \bm{\varphi}\rangle_{\partial T}\\
&=&\langle Q_b^{(0)}\S(Q_b^{(0)} \bm{\varphi}) - Q_b^{(0)} \bm{\varphi}, Q_b^{(0)}\mbbQ_h^{(1)} \bm{\varphi} - Q_b^{(0)} \bm{\varphi}\rangle_{\partial T} \nonumber\\
&\leq & \|Q_b^{(0)}\S(Q_b^{(0)} \bm{\varphi}) - Q_b^{(0)} \bm{\varphi}\|_{\partial T} \|Q_b^{(0)}\mbbQ_h^{(1)} \bm{\varphi} - Q_b^{(0)} \bm{\varphi}\|_{\partial T}, \nonumber
\end{eqnarray*}
which gives
\begin{eqnarray*}
\|Q_b^{(0)}\S(\bm{v}_q)- \bm{v}_q\|^2_{\partial T}
&\leq& \|Q_b^{(0)}(\mbbQ_h^{(1)} \bm{\varphi}) - Q_b^{(0)} \bm{\varphi}\|^2_{\partial T} \\
&\leq& \|\mbbQ_h^{(1)} \bm{\varphi} - \bm{\varphi}\|^2_{\partial T}\nonumber\\
&\leq&  C_1 (h^{-1}\|\mbbQ_h^{(1)} \bm{\varphi} - \bm{\varphi}\|^2_{T} + h\|\nabla (\mbbQ_h^{(1)}\bm{\varphi}-\bm{\varphi})\|^2_{T})\nonumber\\
&\leq&  C_2 h \|\bm{\varphi}\|_{1,T}^2, \nonumber
\end{eqnarray*}
so that
\begin{eqnarray}\label{equ.norm.2}
 h^{-1}\sum_{T\in\T_h}\|Q_b^{(0)}\S(\bm{v}_q)- \bm{v}_q\|^2_{\partial T} \leq C_3 \|\bm{\varphi}\|_1 \leq  C_3 \| q \|_0.
\end{eqnarray}
Moreover, we have
\begin{eqnarray}
\sum_{e\in \E_h} h^\gamma_e\|[\nabla_{w,\bm{\tau}} \bm{v}_q]\|_e^2
&\leq & \sum_{T\in \T_h}\sum_{e\in T} h^\gamma_e\|[\nabla_{w,\bm{\tau}} \bm{v}_q]\|_e^2\nonumber\\
&\leq & C \sum_{T\in \T_h}\|\nabla_{w,\bm{\tau}} \bm{v}_q\|_T^2 \nonumber\\
&\leq & C \|\nabla_w \bm{v}_q\|^2
 \leq  C \| q \|_0^2,  \label{inf-sup-equation.2}
\end{eqnarray}

By combining the estimates \eqref{inf-sup-equation.1}, \eqref{equ.norm.2}, and \eqref{inf-sup-equation.2}, we arrive at $\3bar \bm{v}_q \3bar  \leq C_0 \| q \|_0$.
Hence
\begin{eqnarray}
\sup_{\bm{u}_b \in [W_h^0(\T_h)]^2, \bm{u}_b\neq 0}\frac{b(\bm{u}_b,q)}{\3bar \bm{u}_b\3bar}
\geq \frac{b(\bm{v}_q,q)}{\3bar \bm{v}_q\3bar} = \frac{\| q \|^2_0}{\3bar \bm{v}_q\3bar} \geq C_0^{-1} \| q \|_0,
\end{eqnarray}
which gives the desired {\em inf-sup} condition.
\end{proof}

\begin{theorem}\label{uniqueness}
The numerical scheme \eqref{SWG.Primal} has one and only one solution for any positive stabilization parameter $\kappa>0$.
\end{theorem}

\begin{proof}
Since the number of equations equals the number of unknowns in \eqref{SWG.Primal}, it suffices to prove the solution uniqueness or to show that the homogeneous problems has only trivial solution. To this end, let $\bm{u}_b \in [W_h^0(\T_h)]^d$ be the solution of the numerical scheme \eqref{SWG.Primal} with homogeneous data $\bm{f}=\bm{0}$, $\bm{g}=0$, and $\bm{\varrho}=0$. By taking $\bm{v}_b =\bm{u}_b$ we arrive at
\begin{eqnarray}
 \sum_{T\in \T_h}(\mu \varepsilon_w(\bm{u}_b),\varepsilon_w(\bm{u}_b))_T +  \sum_{T\in \T_h}(\lambda \nabla_w \cdot \bm{u}_b, \nabla_w \cdot \bm{u}_b) + \kappa S(\bm{u}_b,\bm{u}_b)=\bm{0},
\end{eqnarray}
which leads to $\3bar \bm{u}_b\3bar =0$ and furthermore to $\bm{u}_b\equiv 0$ by using Lemma \ref{LEMMA:October:20:01}. This completes the proof of the theorem.
\end{proof}

\section{Preparation for error estimates}\label{Section:Preparation}
For any given integer $k\ge 0$, let $P_k(\T_h)$ be the space of piecewise polynomials of degree $k$. Denote by $\mathbb{Q}_h^{(k)} $ the $L^2$ projection operator onto $P_k(\T_h)$, $[P_k(\T_h)]^d$, or $[P_k(\T_h)]^{d\times d}$ wherever appropriate in the specific context. Recall that $Q_b^{(0)}$ is the $L^2$ projection operator onto either the finite element space $W_h(\T_h)$ or $[W_h(\T_h)]^d$, as appropriate.

\begin{lemma}
For any $\bm{v} \in [H^1(\Omega)]^d$, the following identities hold true for the projection operators $Q_b^{(0)}$ and $\mathbb{Q}_h^{(k)}$:
\begin{eqnarray}
&&\nabla_w\cdot(Q_b^{(0)} \bm{v}) = \mbbQ_h^{(0)}(\nabla \cdot \bm{v}), \label{projection.1}\\
&&\nabla_w(Q_b^{(0)} \bm{v}) = \mbbQ_h^{(0)}(\nabla\bm{v}),    \label{projection.2}\\
&&\nabla_w\times(Q_b^{(0)} \bm{v}) = \mbbQ_h^{(0)}(\nabla\times\bm{v}).    \label{projection.4}
\end{eqnarray}
Consequently, one has
\begin{eqnarray}
\varepsilon_w(Q_b^{(0)} \bm{v})=\mbbQ_h^{(0)}\varepsilon(\bm{v}). \label{projection.3}
\end{eqnarray}
\end{lemma}

\begin{proof}
For any $\bm{v} \in [H^1(\Omega)]^d$, we have
\begin{eqnarray}
(\nabla_w Q_b^{(0)} \bm{v},q)_T
&=&\langle Q_b^{(0)} \bm{v},q \cdot \bn \rangle_{\partial T} \label{equ.projection.nabla}\\
&=&\langle \bm{v},q \cdot \bn \rangle_{\partial T} \nonumber\\
&=& (\nabla \bm{v},q)_T \nonumber\\
&=& (\mbbQ_h^{(0)} (\nabla\bm{v}),q)_T \nonumber,
\end{eqnarray}
for all $q \in [P_0(\T_h)]^{d\times d}$, which implies $\nabla_w Q_b^{(0)} \bm{v}= \mbbQ_h^{(0)} (\nabla\bm{v})$.
The identities \eqref{projection.2} and \eqref{projection.4} can be verified through a similar approach, and the details are omitted.
\end{proof}

\begin{lemma}\label{lem:error_estimates}
Assume that $(\bm{w};\rho)\in[H^{1+\epsilon}(\Omega)]^d \times H^1(\Omega)$, $\epsilon > \frac{1}{2}$, and $\bmsigma=2\mu \varepsilon (\bm{w})+ \rho I $ satisfies the following equation
\begin{eqnarray}\label{Test_problem}
- \nabla\cdot \bm{\sigma} = \bm{\eta}, \quad \text{in } \Omega.
\end{eqnarray}
Then we have
\begin{equation}\label{Error_Estimates}
\begin{split}
&2(\mu \varepsilon_w(Q_b^{(0)} \bm{w}),\varepsilon_w(\bm{v}_b))_h +
 (\nabla_w\cdot \bm{v}_b, \mbbQ_h^{(0)} \rho)_h\\
=&(\bm{\eta},\S(\bm{v}_b)) + \theta_{\bm{w},\rho}(\bm{v}_b)+ \zeta_{\bm{w},\rho}(\bm{v}_b),
\end{split}
\end{equation}
for all $\bm{v}_b \in  [W^0_h(\T_h)]^d$, where $\theta_{\bm{w},\rho}$ and $\zeta_{\bm{w},\rho}$ are two functionals in the linear space $[W_h(\T_h)]^d$ given by
\begin{eqnarray}
&&\theta_{\bm{w},\rho}(\bm{v}_b)=\sum_{T\in\T_h}\langle \S(\bm{v}_b)-\bm{v_b}, (\bm{\sigma} - \mbbQ_h^{(0)} \bm{\sigma})
\bm{n}\rangle_{\partial T},\label{Error_Estimates.1}\\
&& \zeta_{\bm{w},\rho}(\bm{v}_b)=\langle \bm{v}_b, \bm{\sigma}\bm{n}\rangle_{\Gamma_N}.\label{Error_Estimates.3}
\end{eqnarray}
\end{lemma}

\begin{proof} For any $T\in\T_h$, from \eqref{projection.3} and the integration by parts, we have
\begin{eqnarray}
& &2(\mu\varepsilon_w(Q_b^{(0)} \bm{w}),\varepsilon_w(\bm{v}_b))_T\nonumber\\
&=&2(\mu\mbbQ_h^{(0)}\varepsilon(\bm{w}),\varepsilon_w(\bm{v}_b))_T \nonumber\\
&=&2\langle \bm{v}_b, \mu\mbbQ_h^{(0)}\varepsilon(\bm{w})\bm{n}\rangle_{\partial T}
  \nonumber\\
&=&2 \langle \bm{v}_b, \mu\mbbQ_h^{(0)}\varepsilon(\bm{w})\bm{n}\rangle_{\partial T}
  +2(\mu\nabla \S(\bm{v}_b),\varepsilon(\bm{w}))_T
  -2(\mu\nabla \S(\bm{v}_b),\varepsilon(\bm{w}))_T
 \label{pre.eq1} \\
&=&2\langle \bm{v}_b, \mu\mbbQ_h^{(0)}\varepsilon(\bm{w})\bm{n}\rangle_{\partial T}
   +2(\nabla \S(\bm{v}_b),\mu\varepsilon(\bm{w}))_T
   -2(\nabla \S(\bm{v}_b),\mu\mbbQ_h^{(0)}\varepsilon(\bm{w}))_T
    \nonumber\\
&=&2\langle \bm{v}_b, \mu\mbbQ_h^{(0)}\varepsilon(\bm{w})\bm{n}\rangle_{\partial T}
   +2(\nabla \S(\bm{v}_b),\mu\varepsilon(\bm{w}))_T
   -2\langle \S(\bm{v}_b), \mu\mbbQ_h^{(0)}\varepsilon(\bm{w})\bm{n}\rangle_{\partial T}
    \nonumber\\
&=& 2(\nabla \S(\bm{v}_b),\mu\varepsilon(\bm{w}))_T
    - 2\langle \S(\bm{v}_b)- \bm{v}_b, \mu\mbbQ_h^{(0)}\varepsilon(\bm{w})\bm{n}\rangle_{\partial T} \nonumber .
\end{eqnarray}
Next, by the definition of the weak divergence, the integration by parts, and the fact that $\displaystyle\sum_{T\in\T_h}\langle \bm{v}_b, \rho\bm{n}\rangle_{\partial T}=\langle \bm{v}_b, \rho\bm{n}\rangle_{\Gamma_N}$, we have
\begin{eqnarray*}
(\nabla_w \cdot \bm{v}_b, \mbbQ_h^{(0)} \rho)_h&=&\sum_{T\in \T_h}(\nabla_w \cdot \bm{v}_b, \mbbQ_h^{(0)} \rho)_T\\
&=& \sum_{T\in\T_h} \langle \bm{v}_b, (\mbbQ_h^{(0)} \rho)\bm{n}\rangle_{\partial T}\\
&=& \sum_{T\in\T_h}\left[ (\nabla \cdot \S(\bm{v}_b), \mbbQ_h^{(0)} \rho)_T - (\nabla \cdot \S(\bm{v}_b), \mbbQ_h^{(0)} \rho)_T + \langle \bm{v}_b, (\mbbQ_h^{(0)} \rho)\bm{n}\rangle_{\partial T}\right]\\
&=& \sum_{T\in\T_h}\left[(\nabla \cdot \S(\bm{v}_b), \rho)_T - \langle  \S(\bm{v}_b), (\mbbQ_h^{(0)} \rho) \bm{n})_{\partial T} + \langle \bm{v}_b, (\mbbQ_h^{(0)} \rho)\bm{n}\rangle_{\partial T}\right]\\
&=& \sum_{T\in\T_h}
\left[-(\S(\bm{v}_b), \nabla \rho)_T
      + \langle    \S(\bm{v}_b),\rho\bm{n} \rangle_{\partial T}
      - \langle    \S(\bm{v}_b)-\bm{v}_b,(\mbbQ_h^{(0)} \rho)\bm{n}\rangle_{\partial T}\right]\\
&=& \sum_{T\in\T_h}
\left[-(\S(\bm{v}_b), \nabla \rho)_T
      + \langle \S(\bm{v}_b)-\bm{v_b}, \rho\bm{n} \rangle_{\partial T}
      - \langle \S(\bm{v}_b)-\bm{v}_b,(\mbbQ_h^{(0)} \rho)\bm{n}\rangle_{\partial T}\right]\\
& & + \langle \bm{v}_b, \rho\bm{n}\rangle_{\Gamma_N}\\
&=& - (\S(\bm{v}_b), \nabla \rho)
    + \sum_{T\in\T_h} \langle \S(\bm{v}_b)-\bm{v_b}, (\rho - \mbbQ_h^{(0)} \rho)\bm{n}\rangle_{\partial T}+ \langle \bm{v}_b, \rho\bm{n}\rangle_{\Gamma_N},
\end{eqnarray*}
which leads to
\begin{eqnarray}\label{pre.eq2}
(\S(\bm{v}_b), \nabla \rho)
&=& -(\nabla_w \cdot \bm{v}_b, \mbbQ_h^{(0)} \rho)_h + \langle \bm{v}_b, \rho\bm{n}\rangle_{\Gamma_N}\\
&&+ \sum_{T\in\T_h} \langle \S(\bm{v}_b)-\bm{v_b}, (\rho - \mbbQ_h^{(0)} \rho)\bm{n}\rangle_{\partial T}.\nonumber
\end{eqnarray}
Now by testing \eqref{Test_problem} with $\S(\bm{v}_b)$ we arrive at
\begin{eqnarray*}
-2(\nabla \cdot (\mu \varepsilon(\bm{w})),\S(\bm{v}_b)) -(\nabla \rho, \S(\bm{v}_b))
=(\bm{\eta}, \S(\bm{v}_b)).
\end{eqnarray*}
From the integration by parts, the above equation can be rewritten as
\begin{eqnarray}\label{pre.eq3}
\sum_{T\in\T_h} \left\{(2\mu \varepsilon(\bm{w}),\nabla \S(\bm{v}_b))_T
-\langle\mu \varepsilon(\bm{w})\bm{n}, \S(\bm{v}_b)\rangle_{\partial T} \right\}
-(\nabla \rho, \S(\bm{v}_b))=(\bm{\eta}, \S(\bm{v}_b)).
\end{eqnarray}
Substituting equation \eqref{pre.eq1} and \eqref{pre.eq2} into \eqref{pre.eq3} yields
\begin{eqnarray*}
&& 2\sum_{T\in\T_h} \left\{(\mu\varepsilon_w(Q_b^{(0)} \bm{w}),\varepsilon_w(\bm{v}_b))_T
+\langle \S(\bm{v}_b)- \bm{v}_b, \mu\mbbQ_h^{(0)}\varepsilon(\bm{w})\bm{n}\rangle_{\partial T}
-\langle \mu\varepsilon(\bm{w})\bm{n}, \S(\bm{v}_b)\rangle_{\partial T}\right\}\\
&& +(\nabla_w \cdot \bm{v}_b, \mbbQ_h^{(0)} \rho)_h
-\sum_{T\in\T_h} \langle \S(\bm{v}_b)-\bm{v_b}, (\rho - \mbbQ_h^{(0)} \rho)\bm{n}\rangle_{\partial T} - \langle \bm{v}_b, \rho\bm{n}\rangle_{\Gamma_N}
=(\bm{\eta}, \S(\bm{v}_b)),
\end{eqnarray*}
which implies
\begin{eqnarray*}
&& 2 \sum_{T\in\T_h} \left\{(\mu\varepsilon_w(Q_b^{(0)} \bm{w}),\varepsilon_w(\bm{v}_b))_T
+\langle \S(\bm{v}_b)- \bm{v}_b, \mu(\mbbQ_h^{(0)}\varepsilon(\bm{w})-\varepsilon(\bm{w}))\bm{n}\rangle_{\partial T}
-\langle \mu\varepsilon(\bm{w})\bm{n}, \bm{v}_b\rangle_{\partial T}\right\}\\
&&+(\nabla_w \cdot \bm{v}_b, \mbbQ_h^{(0)} \rho)_h
-\sum_{T\in\T_h} \langle \S(\bm{v}_b)-\bm{v_b}, (\rho - \mbbQ_h^{(0)} \rho)\bm{n}\rangle_{\partial T}- \langle \bm{v}_b, \rho\bm{n}\rangle_{\Gamma_N}=(\bm{\eta}, \S(\bm{v}_b)).
\end{eqnarray*}
Using the boundary condition $\bm{v}_b =0$ on $\Gamma_D$, we obtain
\begin{eqnarray*}
&&2(\mu\varepsilon_w(Q_b^{(0)} \bm{w}),\varepsilon_w(\bm{v}_b))_h + (\nabla_w \cdot \bm{v}_b, \mbbQ_h^{(0)} \rho)_h\\
&=& \sum_{T\in\T_h} \left\{2\langle \bm{v}_b-\S(\bm{v}_b), \mu(\mbbQ_h^{(0)}\varepsilon(\bm{w})-\varepsilon(\bm{w}))\bm{n}\rangle_{\partial T}
+\langle \S(\bm{v}_b)-\bm{v_b}, (\rho - \mbbQ_h^{(0)} \rho)\bm{n}\rangle_{\partial T}\right\}\\
&&+\langle \bm{v}_b, (2\mu\varepsilon(\bm{w})+\rho I)\bm{n}\rangle_{\Gamma_N}+(\bm{\eta}, \S(\bm{v}_b)),
\end{eqnarray*}
which is precisely the equation \eqref{Error_Estimates}.
This completes the proof.
\end{proof}

\section{An error estimate in $H^1$}\label{Section:H1}
Let $(\bm{u}_b;p_h)\in [W_h(\T_h)]^d \times P_0(\T_h)$ be the numerical approximation of the linear elasticity problem \eqref{elasbdy}-\eqref{elasbcn} arising from \eqref{equ.elasticity-Mix-SWG}. The corresponding error functions $\bm{e}_h$ and $\xi_h$ are given by
\begin{eqnarray}\label{error_function}
\bm{e}_h=Q_b^{(0)}\bm{u} -\bm{u}_b,\quad \xi_h=\mbbQ_h^{(0)} p -p_h,
\end{eqnarray}
where $(\bm{u};p)$ is the exact solution of the variational problem \eqref{elasticity_mixed1}-\eqref{elasticity_mixed2}.
It is clear that $\bm{e}_h\in [W^0_h(\T_h)]^d$ and $\xi_h \in P_0(\T_h)$.

\begin{lemma}
The error functions $\bm{e}_h$ and $\xi_h$ defined in \eqref{error_function} satisfy the following error equations
\begin{eqnarray}
a(\bm{e}_h,\bm{v}_b) + b(\bm{v}_b,\xi_h)&=&\varphi_{\bm{u},p}(\bm{v}_b),\quad \forall \bm{v}_b \in W^0_h(\T_h), \label{error_equation.1}\\
b(\bm{e}_h,q)-d(\xi_h,q)&=&0,\quad \forall q\in P_0(\T_h), \label{error_equation.2}
\end{eqnarray}
where
\begin{eqnarray}
\varphi_{\bm{u},p}(\bm{v}_b) = S(Q_b^{(0)}\bm{u},\bm{v}_b)+
\theta_{\bm{u},p}(\bm{v}_b).
\end{eqnarray}
\end{lemma}
\begin{proof}
Observe that the exact solution $(\bm{u};p)$ satisfies the equation \eqref{Test_problem}
with $\eta=f$. Thus from Lemma \ref{lem:error_estimates}, we have
\begin{eqnarray}\label{eq:error_estimates_H1.1}
2(\mu\varepsilon(Q_b^{(0)} \bm{u}),\varepsilon_w(\bm{v}_b))_h + (\nabla_w \cdot \bm{v}_b, \mbbQ_h^{(0)} p)_h = (f,\S(\bm{v}_b)) + \theta_{\bm{u},p}(\bm{v}_b)+\zeta_{\bm{u},p}(\bm{v}_b)
\end{eqnarray}
for any $ \bm{v}_b\in W^0_h(\T_h)$, which leads to
\begin{eqnarray}\label{eq:error_estimates_H1.2}
&&a(Q_b^{(0)}\bm{u},\bm{v}_b) + b(\bm{v}_b,\mbbQ_h^{(0)} p) \\
&=& (f,\S(\bm{v}_b)) + \theta_{\bm{u},p}(\bm{v}_b)+\zeta_{\bm{u},p}(\bm{v}_b) + S(Q_b^{(0)}\bm{u},\bm{v}_b).\nonumber
\end{eqnarray}
Note that the boundary condition \eqref{elasbcn} implies $\bm{\sigma}\bm{n}=\bm{\varrho}$ on $\Gamma_N$. Thus,
by subtracting the first equation of \eqref{equ.elasticity-Mix-SWG} from \eqref{eq:error_estimates_H1.2} we arrive at the equation \eqref{error_equation.1}.

To derive \eqref{error_equation.2}, from \eqref{projection.1} we have for any $q \in P_0(\T_h)$
\begin{equation}\label{eq.error_estimates.3}
\begin{split}
& \ (\nabla_w\cdot(Q_b^{(0)} \bm{u}),q) -\lambda^{-1}(\mbbQ_h^{(0)}p,q)\\
=&\ (\mbbQ_h^{(0)}(\nabla \cdot \bm{u}),q) -\lambda^{-1}(\mbbQ_h^{(0)}p,q)\\
=& \ (\nabla \cdot \bm{u},q) -\lambda^{-1}(p,q)\\
=& \ 0.
\end{split}
\end{equation}
The difference of \eqref{eq.error_estimates.3} and the second equation of
\eqref{equ.elasticity-Mix-SWG} yields \eqref{error_equation.2}.
\end{proof}

The following theorem is concerned with an error estimate for the weak Galerkin finite element approximation $(\bm{u}_b;p_h)$.

\begin{theorem}\label{thm.error.1}
Let $(\bm{u};p) \in [H^{2}(\Omega)]^d\times H^{1}(\Omega)$ be the solution of \eqref{elasbdy}-\eqref{elasbcn}, and $(\bm{u}_b;p_h) \in [W_h(\T_h)]^d \times P_0(\T_h)$ be the solution of the SWG scheme \eqref{equ.elasticity-Mix-SWG}, respectively.
Then, the following error estimate holds true
\begin{eqnarray}\label{eq:error_estimates.h1.3}
\3bar Q_b^{(0)} \bm{u} - \bm{u}_b\3bar +(1+\lambda^{-\frac{1}{2}})\|\mbbQ_h^{(0)} p -p_h\|_0
\leq Ch(\|\bm{u}\|_2 + \|\bm{\sigma}\|_1 + h^{s-1}\|\bm{\sigma}\|_{s,\partial\Omega}),
\end{eqnarray}
where $C$ is a generic constant independent of $(\bm{u};p)$, $s\in [0,1]$ is arbitrary as long as the solution has the corresponding regularity, and $\bm{\sigma} = 2\mu\varepsilon(\bm{u})+ p I$.
Consequently, the following error estimates holds true
\begin{eqnarray}\label{eq:error_estimates.h1.4}
\3bar \bm{u} - \bm{u}_b\3bar +(1+\lambda^{-\frac{1}{2}})\| p -p_h\|_0
\leq Ch(\|\bm{u}\|_2 + \|\bm{\sigma}\|_1 + h^{s-1}\|\bm{\sigma}\|_{s,\partial\Omega}).
\end{eqnarray}
\end{theorem}

\begin{proof}
By choosing $\bm{v}_b=\bm{e}_h$  and $q=\xi_h$
in \eqref{error_equation.1}, we have
\begin{eqnarray}
a(\bm{e}_h, \bm{e}_h) + \lambda^{-1}\|\xi_h\|^2=\varphi_{\bm{u},p}(\bm{e}_h).
\end{eqnarray}
Using Lemma \ref{lem:H1} for the term $\varphi_{\bm{u},p}(\bm{e}_h)$ we arrive at
\begin{eqnarray}
a(\bm{e}_h, \bm{e}_h) + \lambda^{-1}\|\xi_h\|^2\leq C(h\|\bm{\sigma}\|_1 +h^s\|\bm{\sigma}\|_{s,\partial\Omega})\3bar \bm{e}_h\3bar,
\end{eqnarray}
where $\bm{\sigma} = 2\mu\varepsilon(\bm{u})+p I$.
Next, from Lemma \ref{lem:ellip} and the above estimate we obtain
\begin{eqnarray}\label{eq:error_estimates.h1.1}
\alpha_1 \3bar \bm{e}_h\3bar + \lambda^{-1}\|\xi_h\|^2\leq C(h \|\bm{\sigma}\|_1 + h^s\|\bm{\sigma}\|_{s,\partial\Omega})\3bar \bm{e}_h\3bar.
\end{eqnarray}

To derive an error estimate for the pseudo-pressure $p$ in a $\lambda$-independent norm, we use the {\em inf-sup} condition \ref{lem:inf_sup} to obtain
\begin{eqnarray}\label{eq:inf_sup.2}
\beta \| \xi_h \|_0 \leq \sup_{\bm{v}_b \in [W_h^0(\T_h)]^d, \bm{v}_b\neq 0}\frac{b(\bm{v}_b,\xi_h)}{\3bar \bm{v}_b\3bar}.
\end{eqnarray}
From the error equation \ref{error_equation.1} we have
\begin{eqnarray}
 b(\bm{v}_b,\xi_h)&=&-a(\bm{e}_h,\bm{v}_b) + \varphi_{\bm{u},p}(\bm{v}_b),\quad \forall \; \bm{v}_b \in [W_h^0(\T_h)]^d.
\end{eqnarray}
Thus, from Lemma \ref{lem:ellip}, the error estimate \eqref{eq:error_estimates.h1.1}, and Lemma \ref{lem:H1} we have
\begin{eqnarray*}
 |b(\bm{v}_b,\xi_h)|&=& \alpha_2 \3bar \bm{e}_h\3bar \3bar \bm{v}_b\3bar+ |\varphi_{\bm{u},p}(\bm{v}_b)|\\
 &\leq& C(h\|\bm{u}\|_2 + h\|\bm{\sigma}\|_1 + h^s\|\bm{\sigma}\|_{s,\partial\Omega})\3bar \bm{v}_b\3bar.
\end{eqnarray*}
Substituting the above estimate into \eqref{eq:inf_sup.2} yields
\begin{eqnarray}\label{eq:error_estimates.h1.2}
\| \xi_h \|_0 \leq C(h\|\bm{u}\|_2 + h\|\bm{\sigma}\|_1+ h^s \|\bm{\sigma}\|_{s,\partial\Omega}).
\end{eqnarray}
Combining \eqref{eq:error_estimates.h1.1} with \eqref{eq:error_estimates.h1.2} gives rise to the error estimate \eqref{eq:error_estimates.h1.3}.
Finally, \eqref{eq:error_estimates.h1.4} stems from the usual triangle inequality, the estimate \eqref{eq:error_estimates.h1.3}, and the error estimate for the $L^2$ projections.
\end{proof}

\section{Error estimates in negative norms}\label{Section:L2}
For any given vector-valued function $\bm{\eta}$, consider the auxiliary problem of seeking $\bm{\psi}\in[H^1(\Omega)]^d$ and $\zeta\in L^2(\Omega)$ satisfying
\begin{eqnarray}
-\nabla \cdot(2\mu\varepsilon(\bm{\psi}) + \zeta I)=\bm{\eta},&& \qquad \text{ in } \Omega, \label{eq:duality.1}\\
  \nabla \cdot \bm{\psi} - \lambda^{-1}\zeta=0,&& \qquad \text{ in } \Omega, \label{eq:duality.2}\\
   \bm{\psi}=\bm{0}, && \qquad \text{ on } \Gamma_D. \label{eq:duality.3}\\
   (2\mu\varepsilon(\bm{\psi}) + \zeta I)\bm{n}=0, && \qquad \text{ on } \Gamma_N. \label{eq:duality.4}
\end{eqnarray}
Assume that the dual problem \eqref{eq:duality.1}-\eqref{eq:duality.4} has the $[H^{2+\epsilon}(\Omega)]^d \times H^{1+\epsilon}(\Omega)$-regularity in the sense that the solution $(\bm{\psi};\zeta)\in [H^{2+\epsilon}(\Omega)]^d\times H^{1+\epsilon}(\Omega)$ and satisfies the following a priori estimate:
\begin{eqnarray}\label{eq:dual_regularity}
\|\bm{\psi}\|_{2+\epsilon} + \|\zeta\|_{1+\epsilon} \leq C\|\bm{\eta}\|_{\epsilon},
\end{eqnarray}
with some $\epsilon\in [0, \frac12]$, see \cite{Dauge_EBVP_1988} for details.

\begin{theorem}
Assume that the solution of \eqref{eq:general_stokes.1}-\eqref{eq:general_stokes.2} is sufficiently smooth such that $(\bm{u};p)\in[H^{2}(\Omega)]^d\times H^1(\Omega)$. Let $(\bm{u}_h;p_h) \in [W_h(\T_h)]^d\times P_0(\T_h)$ be the finite element solution arising from \eqref{equ.elasticity-Mix-SWG}.
Under the regularity assumption \eqref{eq:dual_regularity} and $s=\epsilon+\frac12 \in [\frac12,1]$, there exists a constant $C$ such that
\begin{equation} \label{equ.thm.error.3}
\begin{split}
\|\S(Q_b^{(0)} \bm{u}) - \S(\bm{u}_b)\|_{-\epsilon} \leq& Ch^2 (\|\bm{u}\|_2+\|\bm{\varrho}\|_{1,\Gamma_N})
+ Ch^{1+s} \|\bm{\sigma}(\bm{u},p)\|_1 \\
& + Ch^{2s} \|\bm{\sigma}(\bm{u},p)\|_{s,\partial\Omega}.
\end{split}
\end{equation}
\end{theorem}

\begin{proof}
Let $(\bm{\psi};\zeta)$ be the solution of \eqref{eq:duality.1}-\eqref{eq:duality.4} with $\bm{\eta} \in [H^\epsilon(\Omega)]^d$. It follows from Lemma \ref{lem:error_estimates}, namely the identity \eqref{Error_Estimates}, that
\begin{eqnarray}\label{eq:error_estimates.1}
2(\mu \varepsilon_w(Q_b^{(0)} \bm{\psi}),\varepsilon_w(\bm{v}_b))_h +
 (\nabla_w\cdot \bm{v}_b, \mbbQ_h^{(0)} \zeta)_h
=(\bm{\eta},\S(\bm{v}_b)) + \theta_{\bm{\psi},\zeta}(\bm{v}_b),
\end{eqnarray}
for any $\bm{v}_b \in [W_h^0(\T_h)]^d$. By choosing $\bm{v}_b=\bm{e}=Q_b^{(0)} \bm{u}-\bm{u}_b$ in \ref{eq:error_estimates.1}, we obtain
\begin{eqnarray}\label{eq:error_estimates.3}
(\S(\bm{e}), \bm{\eta})=a(Q_b^{(0)} \bm{\psi},\bm{e}) + b(\bm{e}, \mbbQ_h^{(0)} \zeta) -\varphi_{\bm{\psi},\zeta}(\bm{e}),
\end{eqnarray}
where
\begin{eqnarray}
\varphi_{\bm{\psi},\zeta}(\bm{e}) = \S(Q_b^{(0)}\bm{\psi},\bm{e})+
 \theta_{\bm{\psi},\zeta}(\bm{e}).
\end{eqnarray}

From the error equation \eqref{error_equation.2} we have
\begin{eqnarray}\label{eq:error_estimates.2}
b(\bm{e},\mbbQ_h^{(0)} \zeta)-d(\mbbQ_h^{(0)} \zeta,\xi)&=&0.
\end{eqnarray}
Next, from equation \eqref{projection.1} and \eqref{eq:duality.2}, we have
\begin{eqnarray*}
b(Q_b^{(0)} \bm{\psi},\xi) &= &\sum_{T\in \T_h}(\nabla_w\cdot Q_b^{(0)} \bm{\psi},\xi)_T
 = \sum_{T\in \T_h}(\mbbQ_h^{(0)}(\nabla\cdot  \bm{\psi}),\xi)_T\\
&= &\sum_{T\in \T_h}(\mbbQ_h^{(0)}\lambda^{-1}\zeta,\xi)_T
 =  d(\mbbQ_h^{(0)}\zeta,\xi).
\end{eqnarray*}
Substituting the above into \eqref{eq:error_estimates.2} yields
\begin{eqnarray}
b(\bm{e},\mbbQ_h^{(0)} \zeta)=b(Q_b^{(0)} \bm{\psi},\xi).
\end{eqnarray}

Combining the last equation with \eqref{eq:error_estimates.1}, we obtain
\begin{eqnarray}
(\S(\bm{e}), \bm{\eta})= a(Q_b^{(0)} \bm{\psi},\bm{e})+b(Q_b^{(0)} \bm{\psi},\xi)- \varphi_{\bm{\psi},\zeta}(\bm{e}),
\end{eqnarray}
which, together with the error equation \eqref{error_equation.1}, leads to
\begin{eqnarray}\label{eq:error_estimates.l2.2}
(\S(\bm{e}), \bm{\eta})=\varphi_{\bm{u},p}(Q_b^{(0)} \bm{\psi}) -\varphi_{\bm{\psi},\zeta}(\bm{e}).
\end{eqnarray}
The rest of the proof will be devoted to an analysis for the two terms on the right-hand side of \eqref{eq:error_estimates.l2.2}.

The following are two useful estimates in our mathematical analysis.
For any $\bm{v}_b\in[W(T)]^d$, we have
\begin{eqnarray}
&&\|Q_b^{(0)}\S(\bm{v}_b)-\bm{v}_b\|_{\partial T}\leq\|Q_b^{(0)}\varphi-\bm{v}_b\|_{\partial T},\; \forall \varphi \in [P_1(T)]^d, \label{eq:error_estimates.l2.3}\\
&&\|\S(\bm{v}_b)\|_{\partial T}\leq C\|\bm{v}_b\|_{\partial T}, \label{eq:error_estimates.l2.4}
\end{eqnarray}
The inequality \eqref{eq:error_estimates.l2.3} is a direct consequence of the definition of $\S(\bm{v}_b)$ as the projection of $\bm{v}_b$ into the space of linear functions with respect to the norm $\|\cdot\|_{\partial T}$.
The inequality \eqref{eq:error_estimates.l2.4} can be derived by using an equivalent form of  $\|\S(\bm{v}_b)\|_{\partial T}$ arising from the edge midpoints $M_i$.

Recall that the term $\varphi_{\bm{u},p}(Q_b^{(0)} \bm{\psi})$ is given by
\begin{eqnarray*}
& & \varphi_{\bm{u},p}(Q_b^{(0)} \bm{\psi})
= {S}(\bm{u}, Q_b^{(0)} \bm{\psi}) + \theta_{\bm{u},p}(Q_b^{(0)} \bm{\psi}),
\end{eqnarray*}
which can be estimated as follows:

\noindent
(1) For the stability term $ {S}(\bm{u},Q_b^{(0)}\bm{\psi})$, we have
\begin{eqnarray*}
|{S}(\bm{u}, Q_b^{(0)} \bm{\psi})|\leq |{S_1}(\bm{u}, Q_b^{(0)} \bm{\psi})| + \kappa|{S_2}(\bm{u}, Q_b^{(0)} \bm{\psi})|.
\end{eqnarray*}
 For ${S_1}(\bm{u}, Q_b^{(0)} \bm{\psi})$, from the orthogonality property of the extension operator $\S(\cdot)$ we have
\begin{equation*}
\begin{split}
|{S_1}(\bm{u}, Q_b^{(0)} \bm{\psi})|=&\ \bigg|h^{-1}\sum_T\langle Q_b^{(0)} \S(\bm{u})- Q_b^{(0)}\bm{u}, Q_b^{(0)}\S(Q_b^{(0)} \bm{\psi}) - Q_b^{(0)} \bm{\psi}\rangle_\pT \bigg| \\
=&\ \bigg|h^{-1}\sum_T\langle Q_b^{(0)}(\mbbQ_h^{(1)} \bm{u})- Q_b^{(0)}\bm{u}, Q_b^{(0)}\S(Q_b^{(0)} \bm{\psi}) - Q_b^{(0)} \bm{\psi}\rangle_\pT \bigg| \\
\leq &\ h^{-1}\sum_T \|\mbbQ_h^{(1)}\bm{u}-\bm{u}\|_\pT \|Q_b^{(0)}\S(Q_b^{(0)} \bm{\psi}) - Q_b^{(0)} \bm{\psi}\|_\pT.
\end{split}
\end{equation*}

From \eqref{eq:error_estimates.l2.3} with $\bm{v}_b=Q_b^{(0)} \bm{\psi}$ and $\bm{\phi}=\mbbQ_h^{(1)}\bm{\psi}$ we have
\begin{equation}\label{equ.EL2.1}
\begin{split}
|{S}_1(\bm{u}, Q_b^{(0)} \bm{\psi})| \leq &\ h^{-1}\sum_T \|\mbbQ_h^{(1)}\bm{u}-\bm{u}\|_\pT \|Q_b^{(0)}\S(Q_b^{(0)} \bm{\psi}) - Q_b^{(0)} \bm{\psi}\|_\pT\\
\leq &\ h^{-1}\sum_T \|\mbbQ_h^{(1)}\bm{u}-\bm{u}\|_\pT \|Q_b^{(0)}(\mbbQ_h^{(1)}\bm{\psi}) - Q_b^{(0)} \bm{\psi}\|_\pT\\
\leq& \ h^{-1} \left( \sum_{T}\|\mbbQ_h^{(1)} \bm{u}- \bm{u}\|_{\partial T}^2 \right)^{\frac{1}{2}}\left( \sum_{T}\|\mbbQ_h^{(1)} \bm{\psi}- \bm{\psi}\|_{\partial T}^2 \right)^{\frac{1}{2}} \\
\leq&\ Ch^{2}\|\bm{u}\|_2\|\bm{\psi}\|_{2}.
\end{split}
\end{equation}

For ${S_2}(\bm{u}, Q_b^{(0)} \bm{\psi})$, from equation \eqref{projection.2} we have
\begin{equation}\label{equ.EL2.2}
\begin{split}
|{S_2}(\bm{u}, Q_b^{(0)} \bm{\psi})|=&
 \left| \sum_{e\in \E_h}h_e^\gamma\langle [\nabla_{w,\bm{\tau}} (Q_b^{(0)}\bm{u}) ],[\nabla_{w,\bm{\tau}} (Q_b^{(0)}\bm{\psi}) ]\rangle_e\right|\\
 =&
 \left| \sum_{e\in \E_h}h_e^\gamma\langle [\mbbQ_h^{(0)}(\nabla \bm{u}) \bm{\tau}],[\mbbQ_h^{(0)}(\nabla \bm{\psi}) \bm{\tau}]\rangle_e\right|\\
  \leq & Ch^{2}\|\bm{u}\|_2\|\bm{\psi}\|_{2}.
\end{split}
\end{equation}

\noindent
(2) For the second term $\theta_{\bm{u},p}(Q_b^{(0)} \bm{\psi})$, with $\bm{\sigma}=2\mu\varepsilon(\bm{u})+pI$,
 we have from the boundary condition $\psi=0$ on $\Gamma_D$ and $\bm{\sigma}\bm{n}=\bm{\varrho}$ on $\Gamma_N$ that
 \begin{equation}\label{EQ:Oct:100.0}
\begin{split}
\left|\sum_{T\in\T_h}\langle Q_b^{(0)} \bm{\psi}-  \bm{\psi}, (\mathbb{Q}_h^{(0)}\bm{\sigma}-\bm{\sigma})\bm{n}\rangle_{\partial T}\right|
=&\left|\sum_{T\in\T_h}\langle \bm{\psi} - Q_b^{(0)} \bm{\psi}, \bm{\sigma}\bm{n}\rangle_{\partial T}\right| \\
=& |\langle \bm{\psi} - Q_b^{(0)} \bm{\psi}, \bm{\varrho}\rangle_{\Gamma_N }|\\
=& |\langle \bm{\psi} - Q_b^{(0)} \bm{\psi}, \bm{\varrho} -Q_b^{(0)} \bm{\varrho} \rangle_{\Gamma_N }|\\
\leq & C h^2 \|\bm{\psi}\|_{1,\Gamma_N} \|\bm{\varrho}\|_{1,\Gamma_N}.
\end{split}
\end{equation}

Next, using the fact that $\S(Q_b^{(0)}(\mbbQ_h^{(1)} \bm{\psi})) = \mbbQ_h^{(1)} \bm{\psi}$ we obtain
 \begin{equation}\label{EQ:Oct:100}
 \begin{split}
 &\left|\sum_{T\in\T_h}\langle \S(Q_b^{(0)} \bm{\psi})- \bm{\psi},(\mathbb{Q}_h^{(0)}\bm{\sigma}-\bm{\sigma})\bm{n}\rangle_{\partial T}\right|\\
 =&\left|\sum_{T\in\T_h}\langle \S(Q_b^{(0)} \bm{\psi})- \S(Q_b^{(0)}(\mbbQ_h^{(1)} \bm{\psi})) + \mbbQ_h^{(1)} \bm{\psi} - \bm{\psi},(\mathbb{Q}_h^{(0)}\bm{\sigma}-\bm{\sigma})\bm{n}\rangle_{\partial T}\right|\\
 \leq&
 \left(  \sum_{T} \| (\mathbb{Q}_h^{(0)}\bm{\sigma}-\bm{\sigma})\bm{n}\|_{\partial T}^2\right)^{\frac{1}{2}}\\
&\cdot \left( \sum_{T} \| \mbbQ_h^{(1)}\bm{\psi} - \bm{\psi}\|_{\partial T}^2 + \|\S(Q_b^{(0)}(\bm{\psi} - \mbbQ_h^{(1)}\bm{\psi} ))\|_{\partial T}^2\right)^{\frac{1}{2}}\\
\leq& Ch\|\bm{\sigma}\|_1\left( \sum_{T} \| \mbbQ_h^{(1)}\bm{\psi} - \bm{\psi}\|_{\partial T}^2\right)^{\frac{1}{2}}\\
\leq& Ch^{2}\|\bm{\sigma}\|_1\|\bm{\psi}\|_{2}.
 \end{split}
 \end{equation}

From \eqref{EQ:Oct:100.0} and \eqref{EQ:Oct:100} we arrive at the following estimate:
\begin{equation}\label{EQ:Oct:100.10}
|\theta_{\bm{u},p}(Q_b^{(0)}\bm{\psi})| \leq Ch^{2}\|\bm{\sigma}\|_1\|\bm{\psi}\|_{2} + C h^2 \|\bm{\psi}\|_{1,\Gamma_N} \|\bm{\varrho}\|_{1,\Gamma_N}.
\end{equation}

The estimates \eqref{equ.EL2.1}, \eqref{equ.EL2.2}, and \eqref{EQ:Oct:100.10}, together with the regularity \eqref{eq:dual_regularity}, collectively yield
\begin{eqnarray}\label{equ.elasticity.L2.1}
|\varphi_{\bm{u},p}(Q_b^{(0)} \bm{\psi})|
&\leq& Ch^{2}(\|\bm{u}\|_2+\|\bm{\sigma}(\bm{u},p)\|_1 +\|\bm{\varrho}\|_{1,\Gamma_N} )\|\bm{\psi}\|_{2}.
\end{eqnarray}

To estimate the term $\varphi_{\bm{\psi},\zeta}(\bm{e})$ in \eqref{eq:error_estimates.l2.2}, we apply the inequalities \eqref{H1.1}-\eqref{H1.3} of Lemma \ref{lem:H1} to obtain
\begin{eqnarray*}
|\varphi_{\bm{\psi},\zeta}(\bm{e})|
&\leq& C(h\|\bm{\psi}\|_{2} + h\|\bm{\sigma}(\bm{\psi},\zeta)\|_1+h^s\| \bm{\sigma}(\bm{\psi},\zeta)\|_{s,\partial\Omega})\vertiii{\bm{e}},
\end{eqnarray*}
where $s\in [0,1]$ is arbitrary and $\bm{\sigma}(\bm{\psi},\zeta) = 2\mu\varepsilon(\bm{\psi})+\zeta I$. From the following Sobolev trace inequality
$$
\| \bm{\sigma}(\bm{\psi},\zeta)\|_{s,\partial\Omega} \leq C \|\bm{\sigma}(\bm{\psi},\zeta)\|_{s+\frac12}
$$
we have
\begin{eqnarray}\label{equ.stokes.L2.2}
|\varphi_{\bm{\psi},\zeta}(\bm{e})|
&\leq& C(h\|\bm{\psi}\|_{2} + h\|\bm{\sigma}(\bm{\psi},\zeta)\|_1+h^s\| \bm{\sigma}(\bm{\psi},\zeta)\|_{s+\frac12})\vertiii{\bm{e}}.
\end{eqnarray}

Now substituting \eqref{equ.elasticity.L2.1} and \eqref{equ.stokes.L2.2} into \eqref{eq:error_estimates.l2.2} yields
\begin{eqnarray*}
|(\S(\bm{e}), \bm{\eta})|
&\leq &
Ch^{2}(\|\bm{u}\|_2+\|\bm{\sigma}(\bm{u},p)\|_1 +\|\bm{\varrho}\|_{1,\Gamma_N})\|\bm{\psi}\|_{2}\\
&& + C(h\|\bm{\psi}\|_{2} + h\|\bm{\sigma}(\bm{\psi},\zeta)\|_1+h^s\| \bm{\sigma}(\bm{\psi},\zeta)\|_{s+\frac12})\vertiii{\bm{e}}\\
&\leq &
C\left\{h^{2}(\|\bm{u}\|_2+\|\bm{\sigma}(\bm{u},p)\|_1 +\|\bm{\varrho}\|_{1,\Gamma_N})
+ h \vertiii{\bm{e}}\right\} (\|\bm{\psi}\|_{2}+\|\zeta\|_1)\\
&& + Ch^s\vertiii{\bm{e}}\
\|\bm{\sigma}(\bm{\psi},\zeta)\|_{s+\frac12}.
\end{eqnarray*}
The regularity assumption \eqref{eq:dual_regularity} implies the validity of the following estimate:
$$
\|\bm{\sigma}(\bm{\psi},\zeta)\|_{s+\frac12} \leq C (\|\bm{\psi}\|_{s+\frac32}+\|\zeta\|_{s+\frac12})\leq C \|\bm{\eta}\|_{s-\frac12}.
$$
By combining the last two estimates we arrive at
\begin{eqnarray*}
|(\S(\bm{e}), \bm{\eta})|
&\leq &
C\left\{h^{2}(\|\bm{u}\|_2+\|\bm{\sigma}(\bm{u},p)\|_1 +\|\bm{\varrho}\|_{1,\Gamma_N})
+ h \vertiii{\bm{e}}\right\} \|\bm{\eta}\|_0 + Ch^s\vertiii{\bm{e}}\
\|\bm{\eta}\|_{s-\frac12},
\end{eqnarray*}
which, together with the error estimate \eqref{eq:error_estimates.h1.3}, leads to
\begin{eqnarray*}
|(\S(\bm{e}), \bm{\eta})|
&\leq &
Ch^{1+s}\left(h^{1-s}(\|\bm{u}\|_2+\|\bm{\sigma}(\bm{u},p)\|_1 +\|\bm{\varrho}\|_{1,\Gamma_N})
+ \|\bm{\sigma}(\bm{u},p)\|_{s,\partial\Omega}\right) \|\bm{\eta}\|_0 \\
&&+ Ch^{1+s}\left(\|\bm{\sigma}(\bm{u},p)\|_1 + h^{s-1}\|\bm{\sigma}(\bm{u},p)\|_{s,\partial\Omega}\right)
\|\bm{\eta}\|_{s-\frac12}
\end{eqnarray*}
for all $\bm{\eta}\in H^{s-\frac12}(\Omega)$. The desired error estimate \eqref{equ.thm.error.3} stems from the above estimate without any difficulty. This completes the proof of the theorem.
\end{proof}

\section{Some technical inequalities}\label{Section:TI}
In this section, we present some technical inequalities that support the error analysis established in previous sections.

Recall that $\T_h$ is a shape-regular finite element partition of $\Omega$. It is known that
there exists a constant $C>0$ such that, for any edge $e \subset T$ with $T\in \T_h$, the following trace inequality holds true
\begin{eqnarray}\label{trace_inq}
\|\phi\|_e^2 \leq C(h^{-1}_T\|\phi\|^2_T + h_T\|\nabla \phi\|^2_T),\quad \forall \phi \in H^1(T),
\end{eqnarray}
where $h_T$ is the size of $T$.
Furthermore, in the polynomial space $P_j(T),j \geq0$, we have from the inverse inequality that
\begin{eqnarray}
\|\phi\|_e^2 \leq Ch^{-1}_T\|\phi\|^2_T ,\quad \forall \phi \in P_j(T), T \in \T_h.
\end{eqnarray}

\begin{lemma}
Assume that $\T_h$ is a finite element partition of $\Omega$ satisfying the shape regularity assumption as defined in \cite{wy3655}. For any $\bm{w}\in [H^{2}(\Omega)]^d$, $\rho\in H^1(\Omega)$, and $0\leq m \leq 1$, there holds
\begin{eqnarray}
&&\sum_{T\in \T_h}h^{2m}_T\| \bm{w} -\mbbQ_h^{(1)}\bm{w}\|^2_{m,T}\leq Ch^{4}\|\bm{w}\|^2_{2},\label{eq:support.1}\\
&&\sum_{T\in \T_h}h^{2m}_T\| \varepsilon(\bm{w}) -\mbbQ_h^{(0)}\varepsilon(\bm{w})\|^2_{m,T}\leq Ch^{2}\|\bm{w}\|^2_{2}, \label{eq:support.2}\\
&&\sum_{T\in \T_h}h^{2m}_T\| \rho -\mbbQ_h^{(0)}\rho\|^2_{m,T}\leq Ch^{2}\|\rho\|^2_{1}, \label{eq:support.3}\\
&&\sum_{e\in \E_h}h_e^\gamma\|[\mbbQ_h^{(0)}\bm{w}]\|_e^2 \leq C h^3\|\nabla \bm{w}\|^2_{0} + h^2 \|\bm{w}\|_{0,\partial\Omega}^2,\label{eq:support.4}
\end{eqnarray}
where $\gamma=1$ on interior edges and $\gamma=2$ on boundary edges.
\end{lemma}

\begin{proof}
Let us outline a proof for \eqref{eq:support.4} only, as the other three can be found in existing literature. To this end, for any interior edge $e$ shared by two elements $T_1$ and $T_2$, we have
$$
[\mbbQ_h^{(0)}\bm{w}]\big|_e = (\mbbQ_h^{(0)}\bm{w}\big|_{T_1\cap e} - \bm{w}\big|_e) + (\bm{w}\big|_e - \mbbQ_h^{(0)}\bm{w}\big|_{T_2\cap e} ).
$$
It follows from the trace inequality \eqref{trace_inq} that
\begin{equation}\label{EQ:Oct:001}
\begin{split}
h_e^\gamma\|[\mbbQ_h^{(0)}\bm{w}]\|_e^2 &\leq Ch^{\gamma-1} \|\mbbQ_h^{(0)}\bm{w} - \bm{w}\|_{T_1\cup T_2}^2 + Ch^{\gamma+1} \|\nabla \bm{w}\|_{T_1\cup T_2}^2\\
& \leq Ch^{\gamma+1} \| \nabla\bm{w}\|_{T_1\cup T_2}^2.
\end{split}
\end{equation}
For boundary edge $e\in(\partial\Omega)\cap (\pT)$, we have
$$
[\mbbQ_h^{(0)}\bm{w}]\big|_e = (\mbbQ_h^{(0)}\bm{w}\big|_{T\cap e} - \bm{w}\big|_e) + \bm{w}\big|_e.
$$
Again, from the trace inequality \eqref{trace_inq} we have
\begin{equation}\label{EQ:Oct:002}
\begin{split}
h^\gamma\|[\mbbQ_h^{(0)}\bm{w}]\|_e^2 &\leq Ch^{\gamma-1} \|\mbbQ_h^{(0)}\bm{w} - \bm{w}\|_{T}^2 + Ch^{\gamma+1} \| \nabla\bm{w}\|_{T}^2 + C h^\gamma \|\bm{w}\|_e^2\\
&\leq Ch^{\gamma+1} \| \nabla\bm{w}\|_{T}^2 + C h^\gamma \|\bm{w}\|_e^2.
\end{split}
\end{equation}
Summing \eqref{EQ:Oct:001} and \eqref{EQ:Oct:002} over all the edges yields the desired estimate \eqref{eq:support.4}. This completes the proof.
\end{proof}

The following is a characterization of the space of rigid motions as the kernel of the strain operator.

\begin{lemma}
Let $\Omega$ be an open bounded and connected domain in $\mathbb{R}^d$.
The space of rigid motions $RM(\Omega)$ is identical to the kernel of the strain tensor operator; i.e., for $\bm{v}\in[H^1(\Omega)]^d$,
there holds $\varepsilon(\bm{v})=0$  if and only if $\bm{v}\in RM(\Omega)$.
\end{lemma}

\begin{proof}
For any $\bm{v} \in RM(\Omega)$, there exist $\bm{a} \in \mathbb{R}^d$ and a skew-symmetric $d\times d$ matrix $\eta$ such that $\bm{u}= \bm{a} + \eta x$. It is easy to check that $\varepsilon(\bm{v})=0$.

For any $\bm{v}\in [H^1(\Omega)]^d$, it is not hard to verify the following identity:
\begin{eqnarray}
 \partial_j\partial_k \bm{v}_i
=\partial_j\varepsilon_{ik}(\bm{v})
+\partial_k\varepsilon_{ij}(\bm{v})
-\partial_i\varepsilon_{jk}(\bm{v}).
\end{eqnarray}
Thus, for any $\bm{v}\in [H^1(\Omega)]^d$ satisfying $\epsilon(\bm{v})=0$, we have
$\partial_j\partial_k \bm{v}_i=0$ and hence $\bm{v} \in [P_1(\Omega)]^d$.
It follows that there exist $\bm{a}\in\mathbb{R}^d $ and $\eta \in \mathbb{R}^{d\times d}$ such that $\bm{v}= \bm{a} + \eta x$.
Since $\varepsilon(\bm{v})=0$, we get
\begin{eqnarray}
0 = \varepsilon(\bm{v})=\varepsilon(\bm{a} + \eta x)
=\frac{1}{2}(\nabla(\bm{a} + \eta x) + \nabla(\bm{a} + \eta x)^t)
=\frac{1}{2}(\eta + \eta^t).
\end{eqnarray}
It follows that $\eta =- \eta^t$, which means that $\eta$ is skew-symmetric, and hence $\bm{v}\in RM(\Omega)$.
\end{proof}


\begin{lemma}\label{lem:H1}
Assume that the finite element partition $\T_h$ of $\Omega$ is shape regular, $\bm{w}\in [H^2(\Omega)]^d$, and $\rho \in H^1(\Omega)$. Then for any $\bm{v}_b \in [W_h^0(\T_h)]^d$ and $s\in [0,1]$, the following estimates hold true
\begin{eqnarray}
&&|S(Q_b^{(0)}{\bm{w}},\bm{v}_b)|\leq Ch \|\bm{w}\|_2\3bar\bm{v}_b\3bar, \label{H1.1}\\
&&|\theta_{\bm{w},\rho}(\bm{v}_b)|\leq C(h\|\bm{\sigma}\|_1 + h^s \|\bm{\sigma}\|_{s,\partial\Omega}) \3bar\bm{v}_b\3bar, \label{H1.3}
\end{eqnarray}
where $\theta_{\bm{w},\rho}(\bm{v}_b)$ is given as in \eqref{Error_Estimates.1} and $\bm{\sigma}=2\mu\varepsilon(\bm{w})+ \rho I$.
\end{lemma}

\begin{proof}
From the definition of the stabilizer $S(\cdot,\cdot)$, we have
\begin{eqnarray*}
&& |S(Q_b^{(0)}{\bm{w}},\bm{v}_b)|\\
&=&\left| \sum_{T \in \T_h} h^{-1}\langle Q_b^{(0)} \S(Q_b^{(0)}\bm{w})-Q_b^{(0)}\bm{w},Q_b^{(0)} \S(\bm{v}_b)- \bm{v}_b\rangle_{\partial T}
+\kappa\sum_{e\in \E_h}h_e^\gamma\langle [\nabla_{w,\bm{\tau}} Q_b^{(0)}\bm{w}],[\nabla_{w,\bm{\tau}} \bm{v}_b]\rangle_e  \right|\\
&\leq& \left| \sum_{T \in \T_h} h^{-1}\langle Q_b^{(0)} \S(Q_b^{(0)}\bm{w})-Q_b^{(0)}\bm{w},Q_b^{(0)} \S(\bm{v}_b)- \bm{v}_b\rangle_{\partial T}\right|
     + \kappa\left|\sum_{e\in \E_h}h_e^\gamma\langle [\nabla_{w,\bm{\tau}} Q_b^{(0)}\bm{w}],[\nabla_{w,\bm{\tau}} \bm{v}_b]\rangle_e  \right|.
\end{eqnarray*}
To derive \eqref{H1.1}, we use  \eqref{eq:support.1}, the Cauchy-Schwarz inequality, the inverse inequality and the trace inequality to obtain
\begin{eqnarray*}
&&\left| \sum_{T \in \T_h} h^{-1}\langle Q_b^{(0)} \S(Q_b^{(0)}\bm{w})-Q_b^{(0)}\bm{w},Q_b^{(0)} \S(\bm{v}_b)- \bm{v}_b\rangle_{\partial T}\right|\\
&=&\left| \sum_{T \in \T_h} h^{-1}\langle Q_b^{(0)} (\mbbQ_h^{(1)}\bm{w})-Q_b^{(0)}\bm{w},Q_b^{(0)} \S(\bm{v}_b)- \bm{v}_b\rangle_{\partial T}\right|\\
&\leq&
\left( \sum_{T \in \T_h}h^{-2}\|\mbbQ_h^{(1)} \bm{w}- \bm{w}\|^2_T +
\|\nabla(\mbbQ_h^{(1)}\bm{w}- \bm{w})\|_T^2\right)^{\frac{1}{2}}
\left( h^{-1}\sum_{T \in \T_h}\|Q_b^{(0)}\S(\bm{v}_b)- \bm{v}_b\|_\pT^2\right)^{\frac{1}{2}}\\
&\leq& Ch \|\bm{w}\|_2 \3bar\bm{v}_b\3bar
\end{eqnarray*}
and from \eqref{eq:support.4}
\begin{eqnarray*}
&&\left|\sum_{e\in \E_h}h_e^\gamma\langle [\nabla_{w,\bm{\tau}} Q_b^{(0)}\bm{w}],[\nabla_{w,\bm{\tau}} \bm{v}_b]\rangle_e  \right|\\
&=& \left|\sum_{e\in \E_h}h_e^\gamma\langle [\mbbQ_h^{(0)}(\nabla \bm{w})\otimes \bm{\tau}],[\nabla_{w,\bm{\tau}}\bm{v}_b]\rangle_e  \right|\\
& \leq & \left(\sum_{e\in \E_h}h_e^\gamma\|[\mbbQ_h^{(0)}\nabla \bm{w}]\|_e^2 \right)^{\frac{1}{2}}
\left(\sum_{e\in \E_h}h_e^\gamma\|[\nabla_{w,\bm{\tau}} \bm{v}_b]\|_e^2 \right)^{\frac{1}{2}}\\
& \leq & Ch \|\bm{w}\|_2 \3bar\bm{v}_b\3bar.
\end{eqnarray*}
Combing the last three inequalities gives the estimate \eqref{H1.1}.
\medskip

Next, with $\bm{\sigma}=2\mu\varepsilon(\bm{w})+\rho I$, we have
\begin{equation}\label{EQ:Oct:005}
\begin{split}
&|\theta_{\bm{w},\rho}(\bm{v}_b)|\\
=&\left|\sum_{T\in\T_h}\langle \S(\bm{v}_b)- \bm{v}_b, (\mathbb{Q}_h^{(0)}\bm{\sigma}-\bm{\sigma})\bm{n}\rangle_{\partial T}\right|, \\
\leq&
\left|\sum_{T\in\T_h}\langle \S(\bm{v}_b)- Q_b^{(0)}\S(\bm{v}_b), (\mathbb{Q}_h^{(0)}\bm{\sigma}-\bm{\sigma})\bm{n}\rangle_{\partial T}
\right|\\
&+\left|
\sum_{T\in\T_h}\langle Q_b^{(0)}\S(\bm{v}_b) - \bm{v}_b, (\mathbb{Q}_h^{(0)}\bm{\sigma}-\bm{\sigma})\bm{n}\rangle_{\partial T}
\right|.
\end{split}
\end{equation}
From the trace inequality \eqref{trace_inq}, it is not hard to obtain the following estimate:
\begin{eqnarray}
&&\left|\sum_{T\in\T_h}\langle  Q_b^{(0)}(\S(\bm{v}_b))-\bm{v}_b, (\mathbb{Q}_h\bm{\sigma}-\bm{\sigma})\bm{n}\rangle_{\partial T}\right|
 \leq  Ch\|\bm{\sigma}\|_1 \3bar\bm{v}_b\3bar. \label{eq:Lw1}
\end{eqnarray}
As for the first term, we have
\begin{eqnarray}
&&\left|\sum_{T\in\T_h}\langle \S(\bm{v}_b)- Q_b^{(0)}(\S(\bm{v}_b)), (\mathbb{Q}_h\bm{\sigma}-\bm{\sigma})\bm{n}\rangle_{\partial T}\right|\label{eq:Lw2} \\
&= &
\left|\sum_{T\in\T_h} \langle (x-x_m)\partial_{\bm{\tau}}\S(\bm{v}_b), (Q_b^{(0)}\bm{\sigma}-\bm{\sigma})\bm{n}\rangle_{\partial T}\right|\nonumber\\
&= &
\left|\sum_{e\in\E_h} \langle (x-x_m)[\partial_{\bm{\tau}}\S(\bm{v}_b)], (Q_b^{(0)}\bm{\sigma}-\bm{\sigma})\bm{n}\rangle_{e}\right|\nonumber\\
&\leq&
C\sum_{e\in\E_h^0} h_e\|[\partial_{\bm{\tau}}\S(\bm{v}_b)]\|_e \|Q_b^{(0)}\bm{\sigma}-\bm{\sigma}\|_{e}
+ C\sum_{e\in(\E_h\cap \partial\Omega)} h_e \|\partial_{\bm{\tau}}\S(\bm{v}_b)\|_e \|\bm{\sigma}-Q_b^{(0)}\bm{\sigma}\|_e\nonumber\\
&\leq&
C h \left(\sum_{e\in\E_h^0} h_e\|[\partial_{\bm{\tau}}\S(\bm{v}_b)]\|_e^2\right)^{\frac12} \|\bm{\sigma}\|_{1}
 + Ch^s\left(\sum_{e\in(\E_h\cap \partial\Omega)} h_e^2 \|\partial_{\bm{\tau}}\S(\bm{v}_b)\|_e^2\right)^{\frac12} \|\bm{\sigma}\|_{s,\partial\Omega}\nonumber\\
&\leq & C(h\|\bm{\sigma}\|_1 + h^s\|\bm{\sigma}\|_{s,\partial\Omega}) \3bar\bm{v}_b\3bar. \nonumber
\end{eqnarray}
Substituting \eqref{eq:Lw1}-\eqref{eq:Lw2} into \eqref{EQ:Oct:005} yields the estimate \eqref{H1.3}. This completes the proof of the lemma.
\end{proof}

\section{Numerical experiments}\label{Section:NE}
The objective of this section is two-fold. First, we shall numerically verify the theoretical error estimates developed in previous sections by applying the algorithm to some testing problems. Secondly, we demonstrate the flexility and efficiency of the new numerical method via two benchmark problems, both in 2D and 3D. For simplicity, the $P_0$ finite element method developed in this paper shall be called ``SWG" method in this section. It is indeed a simplified WG when the tangential stability is absent from the scheme.

\subsection{Numerical experiments in 2D}
\subsubsection{Test Case 1: Dirichlet boundary condition}\label{Testcase2D_1}
The computational domain in the first test case is given by $\Omega=(0,1)\times (0,1)$, which is partitioned into different type of meshes of size $h$. The exact solution of the model problem is given by
\begin{equation*}
\begin{array}{lll}
\bm{u} =\left(
\begin{array}{lll}
x^2-y^2\\
x^2+y^2
\end{array}
\right),\\
\end{array}
\end{equation*}
The right-hand side function and the Dirichlet boundary data are chosen to match the exact solution.
The Young modulus and the Poisson ratio are taken as $E=1.$ and $\nu =0.45$.

The numerical deformation of the elastic material on different type of meshes are shown in Figure \ref{fig:testcase1_deformation_pressure}, where the color is provided by the magnitude of the numerical pseudo-pressure $p_h$.

\begin{figure}[!h]
             \centering
             \subfigure[Triangular mesh]{\label{Fig.sub1.4.l}
             \includegraphics [width=0.225\textwidth]{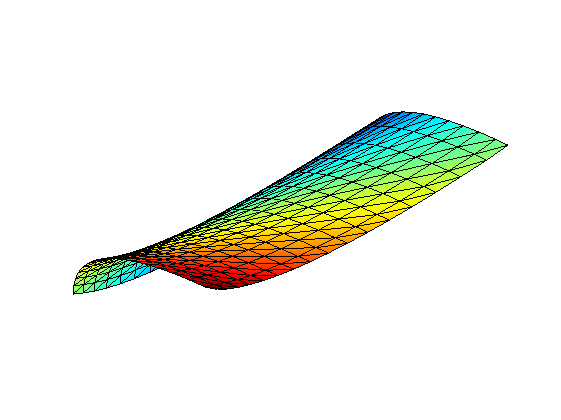}}
             \subfigure[Rectangular mesh]{\label{Fig.sub1.4.2}
             \includegraphics [width=0.225\textwidth]{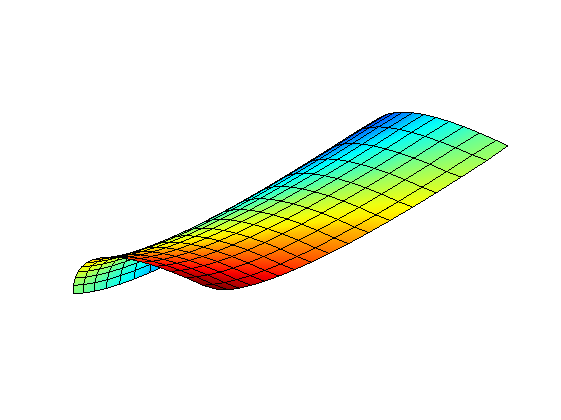}}
             \subfigure[Hexagonal mesh]{\label{Fig.sub1.4.3}
             \includegraphics [width=0.225\textwidth]{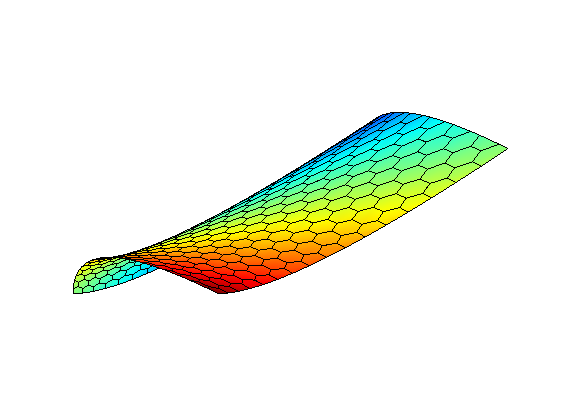}}
             \subfigure[Octagonal mesh]{\label{Fig.sub1.4.4}
             \includegraphics [width=0.225\textwidth]{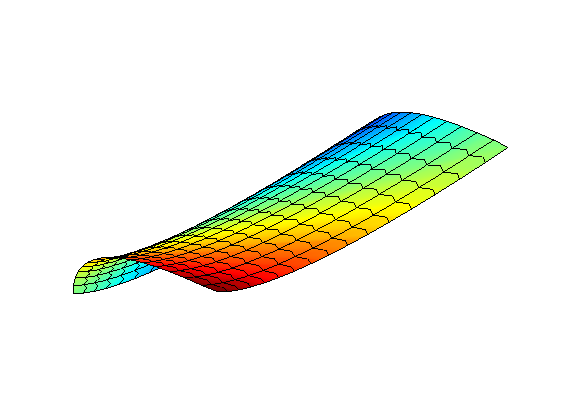}}

            \caption{Numerical results of test case 1 obtained with the SWG scheme \eqref{equ.elasticity-Mix-SWG}, deformation of the elastic material with color provided by the magnitude of the numerical pseudo-pressure $p_h$: (a), (b), (c), (d), using polygonal meshes with $h=1/16$.}
            \label{fig:testcase1_deformation_pressure}
\end{figure}

Tables \ref{table1.1} and \ref{table1.3} illustrate the numerical performance of the SWG scheme when the stability includes the jump of the tangential derivatives (i.e., $\kappa=1$). Table \ref{table1.1} does not involve the tangential derivative on the domain boundary, while Table \ref{table1.3} does with $\gamma=2$ for all boundary edge $e\in \partial\Omega$. The results indicate that the numerical displacement is converging at the optimal order of $r = 2$ in the $L^2$ norm and $r = 1$ in the $H^1$ norm, for all mesh types. For the numerical pressure, the error at the cell center is at the order of $r>1.5$, and for the numerical displacement in $H^1$ norm is at the order of $r = 1$. This implies that the pressure is of superconvergent to the exact solution at the cell centers.

Table \ref{table1.2} illustrates the numerical performance of the SWG scheme when the stability does not include the jump of the tangential derivatives (i.e., $\kappa=0$). We note that, in this case, the stability of the scheme is not known from the current theory. But it can be seen that, when $\kappa=0$, the scheme \eqref{equ.elasticity-Mix-SWG} does not converge for triangular meshes. However, for all other meshes (e.g., rectangular, hexagonal, and octagonal meshes) we tested, the numerical displacement is converging at the optimal order of $r = 2$ in the $L^2$ norm. Moreover, for the numerical pressure, the convergence at the cell center is at the order of $r>1.5$, and for the numerical displacement in $H^1$ norm, the convergence is also at the order of $r>1.5$. The numerical results indicate that the pressure and the numerical displacement in $H^1$ norm are of superconvergent to the exact solutions. We conjecture that the numerical method \eqref{equ.elasticity-Mix-SWG} is stable and convergent for all polygonal finite element partitions except the triangular one.
\begin{table}[!h]
\begin{center}
\caption{Error and convergence performance of the SWG scheme ($\kappa=1$) for the elasticity equation on polygonal meshes of test case 1, where $r$ refers to the order of convergence in $O(h^r).$ }\label{table1.1}
{\tiny
\begin{tabular}{||c|cc|cc|cc|cc|cc|cc||}
\hline
\multicolumn{7}{|>{\columncolor{mypink}}c|}{ Triangular elements }&\multicolumn{6}{|>{\columncolor{green!15}}c|}{ Rectangular elements }\\
\hline
$n$ & $\|e_{(u,v)} \|_{0}$ & $r=$ & $\|e_{(u,v)}\|_{1}$ & $r=$  &$\|e_{p} \|_{0}$& $r=$& $\|e_{(u,v)} \|_{0}$ & $r=$ & $\|e_{(u,v)}\|_{1}$ & $r=$  &$\|e_{p} \|_{0}$& $r=$ \\
\hline
 4  &  1.12e-01 &  -     &   1.27e+00 &  -    &    2.00e-01 &  -     &  1.25e-01 &  -     &   4.39e-01 &  -    &    1.95e-01 &  -    \\
 8  &  3.21e-02 &  1.81  &   6.95e-01 &  0.87 &    7.17e-02 &  1.48  &  3.23e-02 &  1.95  &   1.29e-01 &  1.77 &    6.09e-02 &  1.68\\
 16 &  8.47e-03 &  1.92  &   3.58e-01 &  0.95 &    2.29e-02 &  1.64  &  8.27e-03 &  1.97  &   3.65e-02 &  1.83 &    1.76e-02 &  1.79\\
 32 &  2.16e-03 &  1.97  &   1.81e-01 &  0.99 &    7.72e-03 &  1.57  &  2.10e-03 &  1.98  &   1.01e-02 &  1.85 &    4.91e-03 &  1.84\\
\hline
\multicolumn{7}{|>{\columncolor{yellow!20}}c|}{ Hexagonal elements }&\multicolumn{6}{|>{\columncolor{blue!20}}c|}{ Polygonal elements}\\
\hline
$n$ & $\|e_{(u,v)} \|_{0}$ & $r=$ & $\|e_{(u,v)}\|_{1}$ & $r=$  &$\|e_{p} \|_{0}$& $r=$& $\|e_{(u,v)} \|_{0}$ & $r=$ & $\|e_{(u,v)}\|_{1}$ & $r=$  &$\|e_{p} \|_{0}$& $r=$ \\
\hline
 4  &  9.30e-02  & -     &   2.41e-01 &  -     &   5.63e-01 &  -     &  1.80e-01  & -    &    3.96e-01 &  -     &   6.12e-01 &  -    \\
 8  &  2.79e-02  & 1.74  &   9.93e-02 &  1.28  &   2.06e-01 &  1.45  &  4.86e-02  & 1.89 &    1.19e-01 &  1.73  &   2.14e-01 &  1.51\\
 16 &  8.81e-03  & 1.66  &   4.31e-02 &  1.20  &   7.10e-02 &  1.54  &  1.27e-02  & 1.94 &    3.52e-02 &  1.76  &   7.27e-02 &  1.56\\
 32 &  2.80e-03  & 1.65  &   1.86e-02 &  1.21  &   2.38e-02 &  1.57  &  3.26e-03  & 1.96 &    1.06e-02 &  1.74  &   2.48e-02 &  1.55\\
\hline
\end{tabular}}
\end{center}
\end{table}

\begin{table}[!h]
\begin{center}
\caption{Error and convergence performance of the SWG scheme ($\kappa=1$) with the boundary stability term for the elasticity equation on polygonal meshes of test case 1, where $r$ refers to the order of convergence in $O(h^r).$ }\label{table1.3}
{\tiny
\begin{tabular}{||c|cc|cc|cc|cc|cc|cc||}
\hline
\multicolumn{7}{|>{\columncolor{mypink}}c|}{ Triangular elements }&\multicolumn{6}{|>{\columncolor{green!15}}c|}{ Rectangular elements }\\
\hline
$n$ & $\|e_{(u,v)} \|_{0}$ & $r=$ & $\|e_{(u,v)}\|_{1}$ & $r=$  &$\|e_{p} \|_{0}$& $r=$& $\|e_{(u,v)} \|_{0}$ & $r=$ & $\|e_{(u,v)}\|_{1}$ & $r=$  &$\|e_{p} \|_{0}$& $r=$ \\
\hline
 4  &  1.14e-01  & -     &   1.28e+00 &  -     &   2.26e-01 &  -     &  1.10e-01 &  -    &    3.73e-01 &  -     &   3.71e-01 &  -   \\
 8  &  3.25e-02  & 1.81  &   6.97e-01 &  0.87  &   7.39e-02 &  1.61  &  3.12e-02 &  1.82 &    1.19e-01 &  1.65  &   8.79e-02 &  2.08\\
 16 &  8.52e-03  & 1.93  &   3.59e-01 &  0.96  &   2.32e-02 &  1.67  &  8.22e-03 &  1.92 &    3.49e-02 &  1.77  &   2.14e-02 &  2.04\\
 32 &  2.16e-03  & 1.98  &   1.81e-01 &  0.99  &   7.84e-03 &  1.57  &  2.10e-03 &  1.97 &    9.88e-03 &  1.82  &   5.44e-03 &  1.98\\
\hline
\multicolumn{7}{|>{\columncolor{yellow!20}}c|}{ Hexagonal elements }&\multicolumn{6}{|>{\columncolor{blue!20}}c|}{ Polygonal elements}\\
\hline
$n$ & $\|e_{(u,v)} \|_{0}$ & $r=$ & $\|e_{(u,v)}\|_{1}$ & $r=$  &$\|e_{p} \|_{0}$& $r=$& $\|e_{(u,v)} \|_{0}$ & $r=$ & $\|e_{(u,v)}\|_{1}$ & $r=$  &$\|e_{p} \|_{0}$& $r=$ \\
\hline
 4  &  9.86e-02 &  -     &   2.66e-01 &  -     &   6.62e-01 &  -     &  1.72e-01  & -     &   3.68e-01 &  -    &    7.87e-01 &  -   \\
 8  &  2.91e-02 &  1.76  &   1.12e-01 &  1.24  &   2.27e-01 &  1.54  &  4.81e-02  & 1.83  &   1.14e-01 &  1.69 &    2.48e-01 &  1.67\\
 16 &  9.10e-03 &  1.68  &   4.70e-02 &  1.26  &   7.48e-02 &  1.60  &  1.27e-02  & 1.92  &   3.50e-02 &  1.71 &    7.81e-02 &  1.67\\
 32 &  2.86e-03 &  1.67  &   1.97e-02 &  1.25  &   2.45e-02 &  1.61  &  3.27e-03  & 1.96  &   1.08e-02 &  1.70 &    2.56e-02 &  1.61\\
\hline
\end{tabular}}
\end{center}
\end{table}

\begin{table}[!h]
\begin{center}
\caption{Error and convergence performance of the SWG scheme ($\kappa=0$ or without the tangential stability) for the elasticity equation on polygonal meshes of test case 1, where $r$ refers to the order of convergence in $O(h^r).$ }\label{table1.2}
{\tiny
\begin{tabular}{||c|cc|cc|cc|cc|cc|cc||}
\hline
\multicolumn{7}{|>{\columncolor{mypink}}c|}{ Triangular elements }&\multicolumn{6}{|>{\columncolor{green!15}}c|}{ Rectangular elements }\\
\hline
$n$ & $\|e_{(u,v)} \|_{0}$ & $r=$ & $\|e_{(u,v)}\|_{1}$ & $r=$  &$\|e_{p} \|_{0}$& $r=$& $\|e_{(u,v)} \|_{0}$ & $r=$ & $\|e_{(u,v)}\|_{1}$ & $r=$  &$\|e_{p} \|_{0}$& $r=$ \\
\hline
4  &   2.04e-01 &  -     &   2.53e+00 &  -     &   2.38e-01  & -    & 1.33e-01 &  -     &   4.64e-01 &  -      &  1.74e-01 &  -   \\
8  &   1.95e-01 &  0.07  &   4.73e+00 & -0.90  &   2.20e-01  & 0.12 & 3.40e-02 &  1.97  &   1.36e-01 &  1.78   &  5.56e-02 &  1.65\\
16 &   1.95e-01 & -0.00  &   9.42e+00 & -0.99  &   2.18e-01  & 0.01 & 8.52e-03 &  1.99  &   3.82e-02 &  1.83   &  1.64e-02 &  1.76\\
32 &   1.96e-01 & -0.00  &   1.88e+01 & -1.00  &   2.17e-01  & 0.00 & 2.13e-03 &  2.00  &   1.05e-02 &  1.86   &  4.64e-03 &  1.82\\
\hline
\multicolumn{7}{|>{\columncolor{yellow!20}}c|}{ Hexagonal elements }&\multicolumn{6}{|>{\columncolor{blue!20}}c|}{ Polygonal elements}\\
\hline
$n$ & $\|e_{(u,v)} \|_{0}$ & $r=$ & $\|e_{(u,v)}\|_{1}$ & $r=$  &$\|e_{p} \|_{0}$& $r=$& $\|e_{(u,v)} \|_{0}$ & $r=$ & $\|e_{(u,v)}\|_{1}$ & $r=$  &$\|e_{p} \|_{0}$& $r=$ \\
\hline
4  &   1.20e-01 &  -     &   3.97e-01 &  -      &  1.52e-01  & -    &  2.23e-01 &  -      &  4.89e-01&   -      &  2.53e-01 &  -   \\
8  &   3.16e-02 &  1.92  &   1.43e-01 &  1.48   &  5.38e-02  & 1.50 &  6.11e-02 &  1.87   &  1.53e-01&   1.68   &  8.13e-02 &  1.64\\
16 &   8.57e-03 &  1.88  &   4.77e-02 &  1.58   &  1.47e-02  & 1.87 &  1.59e-02 &  1.95   &  4.59e-02&   1.74   &  2.44e-02 &  1.74\\
32 &   2.23e-03 &  1.94  &   1.59e-02 &  1.59   &  4.16e-03  & 1.82 &  4.03e-03 &  1.98   &  1.38e-02&   1.74   &  7.09e-03 &  1.78\\
\hline
\end{tabular}}
\end{center}
\end{table}

\subsubsection{Test Case 2: locking-free tests}
The computational domain in test case 2 is given by $\Omega=(0,1)\times (0,1)$, which is partitioned into different type of meshes of size $h$. The exact solution of the model problem is given by
\begin{equation*}
\begin{array}{lll}
\bm{u} =\left(
\begin{array}{lll}
\sin(x)\sin(y)\\
\cos(x)\cos(y)
\end{array}
\right)
+ \lambda^{-1}
\left(
\begin{array}{lll}
x\\
y
\end{array}
\right) ,\\
\end{array}
\end{equation*}
The right-hand side function and the Dirichlet boundary data are chosen to match the exact solution. Note that this example has been considered in \cite{WangWang_2016}. In this test case, the Young modulus is taken as $E=1$. Two values of the Poisson ratio $\nu =0.45$ and $\nu=0.4999999$ are considered for this test case.

The numerical deformation of the elastic material on different type of meshes are shown in Figure \ref{fig:testcase2_deformation_pressure}, and the material is colored by the numerical pseudo-pressure $p_h$.

\begin{figure}[!h]
             \centering
             \subfigure[Triangular mesh]{\label{Fig.sub2.4.l}
             \includegraphics [width=0.225\textwidth]{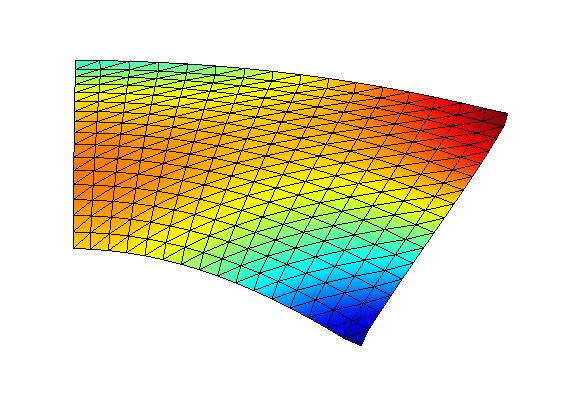}}
             \subfigure[Rectangular mesh]{\label{Fig.sub2.4.2}
             \includegraphics [width=0.225\textwidth]{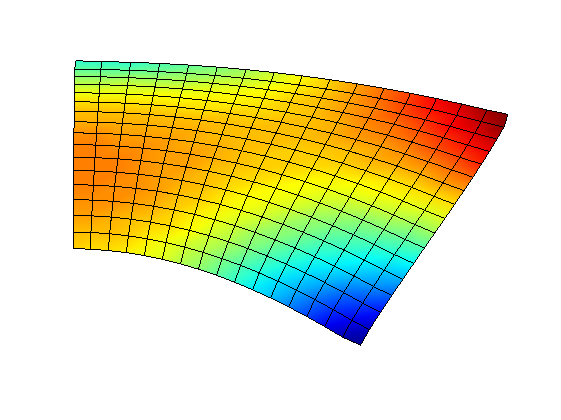}}
             \subfigure[Hexagonal mesh]{\label{Fig.sub2.4.3}
             \includegraphics [width=0.225\textwidth]{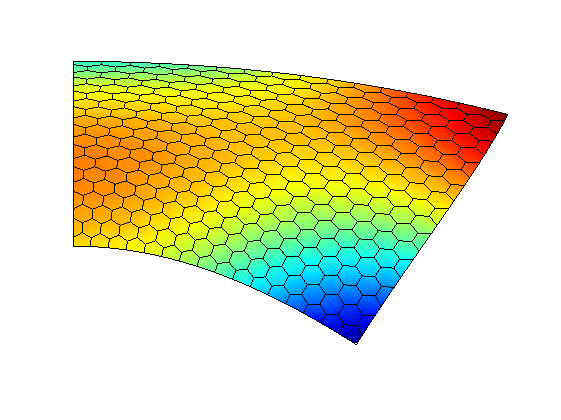}}
             \subfigure[Octagonal mesh]{\label{Fig.sub2.4.4}
             \includegraphics [width=0.225\textwidth]{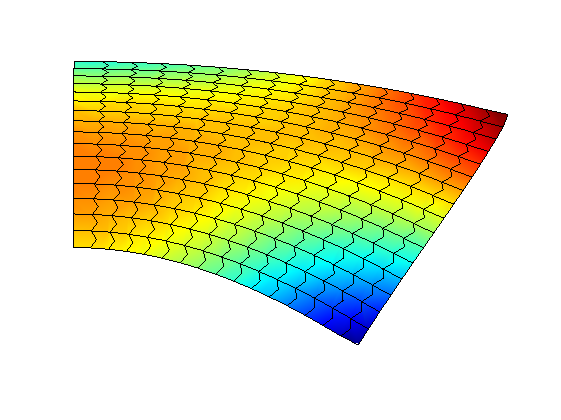}}

            \caption{Numerical results for test case 2 obtained with the SWG scheme \eqref{equ.elasticity-Mix-SWG}, deformation of the elastic material with color provided by the magnitude of the numerical pseudo-pressure $p_h$: (a), (b), (c), (d), using polygonal meshes with $h=1/16$.}
            \label{fig:testcase2_deformation_pressure}
\end{figure}

Tables \ref{table2.1} and \ref{table2.3} illustrate the numerical performance of the SWG scheme when the stability includes the tangential derivatives (i.e., $\kappa =1.0$). Table \ref{table2.3} is concerned with the tangential stability with boundary included, while Table \ref{table2.1} does not consider the boundary stability. It can be observed that the numerical displacement is converging at the order of $r = 2$ in the $L^2$ norm for all partitions except the hexagonal one for which the convergence is of order $r \approx 1.5$. For the numerical pressure, the error at the cell center is at the order of $r>1.5$, and for the numerical displacement in $H^1$ norm is at the order of $r=1$. The experiments suggest that the pressure is of superconvergent to the exact solution at the cell centers. This superconvergence phenomena has been observed for all mesh types.
The results for the Poisson ratio of $\nu=0.4999999 \ (\lambda\rightarrow \infty)$, as shown in the second part of Tables \ref{table2.1} and \ref{table2.3}, clearly indicate a locking-free convergence for the SWG method in various norms, which is consistent with our theory.

Table \ref{table2.2} illustrates the numerical performance of the SWG scheme when no tangential stability is employed (i.e., $\kappa=0$). It can be seen that the numerical scheme does not converge for the triangular element when $\kappa=0$. However, for all other polygonal elements we tested, the numerical displacement converges at the optimal order of $r = 2$ in the $L^2$ norm. For the numerical pressure, the error at the cell center is at the order of $r>1.5$, and for the numerical displacement in $H^1$ norm, the convergence is at the order of $r>1.5$. These results are similar to that for the test case 1.

\begin{table}[!h]
\begin{center}
\caption{Error and convergence performance of the SWG scheme ($\kappa=1.0$ or with interior tangential stability) for the elasticity equation on polygonal meshes of test case 2. $r$ refers to the order of convergence in $O(h^r).$ }\label{table2.1}
{\tiny
\begin{tabular}{||c|cc|cc|cc|cc|cc|cc||}
\hline
\multicolumn{13}{|c|}{$\nu \approx 0.45$}\\
\hline
\multicolumn{7}{|>{\columncolor{mypink}}c|}{ Triangular elements }&\multicolumn{6}{|>{\columncolor{green!15}}c|}{ Rectangular elements }\\
\hline
$n$ & $\|e_{(u,v)} \|_{0}$ & $r=$ & $\|e_{(u,v)}\|_{1}$ & $r=$  &$\|e_{p} \|_{0}$& $r=$& $\|e_{(u,v)} \|_{0}$ & $r=$ & $\|e_{(u,v)}\|_{1}$ & $r=$  &$\|e_{p} \|_{0}$& $r=$ \\
\hline
 4  &  8.84e-03 &  -     &   7.46e-02 &  -     &   1.85e-02 &  -    & 9.18e-03 &  -     &   3.17e-02  & -     &   2.41e-02  & -   \\
 8  &  2.37e-03 &  1.90  &   3.97e-02 &  0.91  &   6.40e-03 &  1.53 & 2.15e-03 &  2.10  &   9.36e-03  & 1.76  &   7.98e-03  & 1.59\\
 16 &  5.93e-04 &  2.00  &   2.00e-02 &  0.99  &   2.13e-03 &  1.59 & 5.12e-04 &  2.07  &   2.55e-03  & 1.87  &   2.03e-03  & 1.98\\
 32 &  1.46e-04 &  2.02  &   9.99e-03 &  1.00  &   7.68e-04 &  1.47 & 1.24e-04 &  2.05  &   6.60e-04  & 1.95  &  4.95e-04   & 2.03 \\
\hline
\multicolumn{7}{|>{\columncolor{yellow!20}}c|}{ Hexagonal elements }&\multicolumn{6}{|>{\columncolor{blue!20}}c|}{ Polygonal elements}\\
\hline
$n$ & $\|e_{(u,v)} \|_{0}$ & $r=$ & $\|e_{(u,v)}\|_{1}$ & $r=$  &$\|e_{p} \|_{0}$& $r=$& $\|e_{(u,v)} \|_{0}$ & $r=$ & $\|e_{(u,v)}\|_{1}$ & $r=$  &$\|e_{p} \|_{0}$& $r=$ \\
\hline
 4  &  1.48e-02 &  -     &   5.63e-02 &  -     &   5.64e-02 &  -    &    1.93e-02  & -     &   4.74e-02 &  -     &   5.19e-02 &  -    \\
 8  &  4.20e-03 &  1.82  &   2.36e-02 &  1.25  &   2.58e-02 &  1.13 &    5.69e-03  & 1.76  &   1.94e-02 &  1.29  &   2.45e-02 &  1.08\\
 16 &  1.39e-03 &  1.60  &   8.83e-03 &  1.42  &   9.33e-03 &  1.47 &    1.57e-03  & 1.86  &   6.74e-03 &  1.53  &   9.09e-03 &  1.43\\
 32 &  4.45e-04 &  1.64  &   3.37e-03 &  1.39  &   3.22e-03 &  1.54 &    3.94e-04  & 2.00  &   2.21e-03 &  1.61  &   3.13e-03 &  1.54\\
\hline
\hline
\multicolumn{13}{|c|}{$\nu = 0.4999999$ }\\
\hline
\multicolumn{7}{|>{\columncolor{mypink}}c|}{ Triangular elements }&\multicolumn{6}{|>{\columncolor{green!15}}c|}{ Rectangular elements }\\
\hline
$n$ & $\|e_{(u,v)} \|_{0}$ & $r=$ & $\|e_{(u,v)}\|_{1}$ & $r=$  &$\|e_{p} \|_{0}$& $r=$& $\|e_{(u,v)} \|_{0}$ & $r=$ & $\|e_{(u,v)}\|_{1}$ & $r=$  &$\|e_{p} \|_{0}$& $r=$ \\
\hline
4  &   8.78e-03 &   -    &   7.43e-02 &   -     &  2.39e-02 &  -    & 9.13e-03 &  -     &   3.38e-02 &  -     &   3.49e-02 &  -   \\
8  &   2.43e-03 &  1.85  &   3.95e-02 &  0.91   &  8.50e-03 &  1.49 & 2.12e-03 &  2.11  &   1.05e-02 &  1.69  &   1.20e-02 &  1.54\\
16 &   6.19e-04 &  1.97  &   2.00e-02 &  0.98   &  2.75e-03 &  1.63 & 5.08e-04 &  2.06  &   2.93e-03 &  1.84  &   3.05e-03 &  1.98\\
32 &   1.53e-04 &  2.01  &   9.95e-03 &  1.00   &  9.46e-04 &  1.54 & 1.23e-04 &  2.05  &   7.74e-04 &  1.92  &   7.43e-04 &  2.04\\
\hline
\multicolumn{7}{|>{\columncolor{yellow!20}}c|}{ Hexagonal elements }&\multicolumn{6}{|>{\columncolor{blue!20}}c|}{ Polygonal elements}\\
\hline
$n$ & $\|e_{(u,v)} \|_{0}$ & $r=$ & $\|e_{(u,v)}\|_{1}$ & $r=$  &$\|e_{p} \|_{0}$& $r=$& $\|e_{(u,v)} \|_{0}$ & $r=$ & $\|e_{(u,v)}\|_{1}$ & $r=$  &$\|e_{p} \|_{0}$& $r=$ \\
\hline
4  &   1.33e-02  &  -    &   5.43e-02 &   -    &   9.57e-02 &   -   &   1.77e-02 &  -  &    4.45e-02&   -   &   8.82e-02 &  - \\
8  &   3.45e-03  & 1.95  &   2.25e-02 &  1.27  &   4.28e-02 &  1.16 &   4.66e-03 &  1.93 &    1.79e-02&   1.31  &   4.06e-02 &  1.12\\
16 &   1.19e-03  & 1.53  &   8.41e-03 &  1.42  &   1.47e-02 &  1.54 &   1.14e-03 &  2.03 &    5.79e-03&   1.63  &   1.43e-02 &  1.51\\
32 &   4.15e-04  & 1.52  &   3.31e-03 &  1.35  &   4.92e-03 &  1.58 &   2.60e-04 &  2.14 &    1.83e-03&   1.66  &   4.70e-03 &  1.60\\
\hline
\end{tabular}}
\end{center}
\end{table}

\begin{table}[!h]
\begin{center}
\caption{Error and convergence performance of the SWG scheme ($\kappa=1.0$ and full tangential stability including the boundary terms) for the elasticity equation on polygonal meshes of test case 2. $r$ refers to the order of convergence in $O(h^r).$ }\label{table2.3}
{\tiny
\begin{tabular}{||c|cc|cc|cc|cc|cc|cc||}
\hline
\multicolumn{13}{|c|}{$\nu = 0.4999999$ }\\
\hline
\multicolumn{7}{|>{\columncolor{mypink}}c|}{ Triangular elements }&\multicolumn{6}{|>{\columncolor{green!15}}c|}{ Rectangular elements }\\
\hline
$n$ & $\|e_{(u,v)} \|_{0}$ & $r=$ & $\|e_{(u,v)}\|_{1}$ & $r=$  &$\|e_{p} \|_{0}$& $r=$& $\|e_{(u,v)} \|_{0}$ & $r=$ & $\|e_{(u,v)}\|_{1}$ & $r=$  &$\|e_{p} \|_{0}$& $r=$ \\
\hline
4  &   1.01e-02 &  -    &    1.03e-01 &  -    &    3.66e-02 &  -    & 1.04e-02 &  -     &   3.86e-02 &  -     &   5.85e-02 &  -     \\
8  &   2.82e-03 &  1.84 &    5.00e-02 &  1.04 &    1.35e-02 &  1.44 & 2.45e-03 &  2.08  &   1.23e-02 &  1.65  &   1.73e-02 &  1.76\\
16 &   7.00e-04 &  2.01 &    2.34e-02 &  1.10 &    4.42e-03 &  1.61 & 5.60e-04 &  2.13  &   3.36e-03 &  1.87  &   3.96e-03 &  2.13\\
32 &   1.66e-04 &  2.08 &    1.09e-02 &  1.10 &    1.50e-03 &  1.56 & 1.30e-04 &  2.11  &   8.56e-04 &  1.97  &   8.81e-04 &  2.17\\
\hline
\multicolumn{7}{|>{\columncolor{yellow!20}}c|}{ Hexagonal elements }&\multicolumn{6}{|>{\columncolor{blue!20}}c|}{ Polygonal elements}\\
\hline
$n$ & $\|e_{(u,v)} \|_{0}$ & $r=$ & $\|e_{(u,v)}\|_{1}$ & $r=$  &$\|e_{p} \|_{0}$& $r=$& $\|e_{(u,v)} \|_{0}$ & $r=$ & $\|e_{(u,v)}\|_{1}$ & $r=$  &$\|e_{p} \|_{0}$& $r=$ \\
\hline
4  &   1.65e-02 &  -     &   6.72e-02  & -     &   1.16e-01 &  -    &   2.05e-02 &  -     &   4.97e-02  & -     &   1.05e-01 &  -    \\
8  &   3.96e-03 &  2.06  &   2.77e-02  & 1.28  &   4.66e-02 &  1.32 &   5.05e-03 &  2.02  &   2.00e-02  & 1.32  &   4.64e-02 &  1.18\\
16 &   1.23e-03 &  2.03  &   6.86e-03  & 1.54  &   1.54e-02 &  1.59 &   1.23e-03 &  2.03  &   6.86e-03  & 1.54  &   1.54e-02 &  1.59\\
32 &   4.31e-04 &  1.57  &   4.02e-03  & 1.39  &   5.06e-03 &  1.61 &   2.78e-04 &  2.15  &   2.39e-03  & 1.52  &   4.90e-03 &  1.66\\
\hline
\end{tabular}}
\end{center}
\end{table}

\begin{table}[!h]
\begin{center}
\caption{Error and convergence performance of the SWG scheme ($\kappa=0.0$ or no tangential stability) for the elasticity equation on polygonal meshes of test case 2. $r$ refers to the order of convergence in $O(h^r).$ }\label{table2.2}
{\tiny
\begin{tabular}{||c|cc|cc|cc|cc|cc|cc||}
\hline
\multicolumn{13}{|c|}{$\nu \approx 0.45$ }\\
\hline
\multicolumn{7}{|>{\columncolor{mypink}}c|}{ Triangular elements }&\multicolumn{6}{|>{\columncolor{green!15}}c|}{ Rectangular elements }\\
\hline
$n$ & $\|e_{(u,v)} \|_{0}$ & $r=$ & $\|e_{(u,v)}\|_{1}$ & $r=$  &$\|e_{p} \|_{0}$& $r=$& $\|e_{(u,v)} \|_{0}$ & $r=$ & $\|e_{(u,v)}\|_{1}$ & $r=$  &$\|e_{p} \|_{0}$& $r=$ \\
\hline
4   &  4.64e-02 &  -     &   5.99e-01 &  -     &   5.26e-02 &  -    & 7.11e-03  & -     &   2.56e-02 &  -     &   1.08e-02 &  -   \\
8   &  4.82e-02 & -0.06  &   1.19e+00 & -0.99  &   5.00e-02 &  0.07 & 1.76e-03  & 2.02  &   7.51e-03 &  1.77  &   3.39e-03 &  1.67\\
16  &  4.88e-02 & -0.02  &   2.38e+00 & -1.00  &   4.93e-02 &  0.02 & 4.35e-04  & 2.01  &   2.12e-03 &  1.83  &   9.92e-04 &  1.77\\
32  &  4.89e-02 & -0.00  &   4.75e+00 & -1.00  &   4.91e-02 &  0.01 & 1.08e-04  & 2.01  &   5.83e-04 &  1.86  &   2.79e-04 &  1.83\\
\hline
\multicolumn{7}{|>{\columncolor{yellow!20}}c|}{ Hexagonal elements }&\multicolumn{6}{|>{\columncolor{blue!20}}c|}{ Polygonal elements}\\
\hline
$n$ & $\|e_{(u,v)} \|_{0}$ & $r=$ & $\|e_{(u,v)}\|_{1}$ & $r=$  &$\|e_{p} \|_{0}$& $r=$& $\|e_{(u,v)} \|_{0}$ & $r=$ & $\|e_{(u,v)}\|_{1}$ & $r=$  &$\|e_{p} \|_{0}$& $r=$ \\
\hline
4  &   7.48e-03 &  -     &   2.04e-02 &  -     &   9.97e-03 &  -    &  1.18e-02 &  -     &   2.67e-02 &  -     &   1.25e-02 &  -   \\
8  &   1.97e-03 &  1.92  &   7.35e-03 &  1.47  &   3.53e-03 &  1.50 &  2.90e-03 &  2.02  &   7.69e-03 &  1.79  &   3.75e-03 &  1.74\\
16 &   5.12e-04 &  1.95  &   2.40e-03 &  1.62  &   9.31e-04 &  1.92 &  7.14e-04 &  2.02  &   2.15e-03 &  1.83  &   1.07e-03 &  1.81\\
32 &   1.30e-04 &  1.98  &   8.18e-04 &  1.55  &   2.68e-04 &  1.79 &  1.76e-04 &  2.02  &   6.01e-04 &  1.84  &   2.97e-04 &  1.84\\
\hline
\hline
\multicolumn{13}{|c|}{$\nu = 0.4999999$ }\\
\hline
\multicolumn{7}{|>{\columncolor{mypink}}c|}{ Triangular elements }&\multicolumn{6}{|>{\columncolor{green!15}}c|}{ Rectangular elements }\\
\hline
$n$ & $\|e_{(u,v)} \|_{0}$ & $r=$ & $\|e_{(u,v)}\|_{1}$ & $r=$  &$\|e_{p} \|_{0}$& $r=$& $\|e_{(u,v)} \|_{0}$ & $r=$ & $\|e_{(u,v)}\|_{1}$ & $r=$  &$\|e_{p} \|_{0}$& $r=$ \\
\hline
4  &   4.65e-02 &  -     &   5.74e-01 &  -      &  6.13e-02 &  -    & 7.22e-03 &  -    &    2.77e-02 &  -     &   1.41e-02 &  -      \\
8  &   4.88e-02 & -0.07  &   1.14e+00 & -0.99   &  5.89e-02 &  0.06 & 1.81e-03 &  2.00 &    8.41e-03 &  1.72  &   4.68e-03 &  1.59\\
16 &   4.95e-02 & -0.02  &   2.29e+00 & -1.00   &  5.82e-02 &  0.02 & 4.50e-04 &  2.00 &    2.43e-03 &  1.79  &   1.41e-03 &  1.73\\
32 &   4.97e-02 & -0.00  &   4.57e+00 & -1.00   &  5.80e-02 &  0.00 & 1.12e-04 &  2.00 &    6.83e-04 &  1.83  &   4.07e-04 &  1.80\\
\hline
\multicolumn{7}{|>{\columncolor{yellow!20}}c|}{ Hexagonal elements }&\multicolumn{6}{|>{\columncolor{blue!20}}c|}{ Polygonal elements}\\
\hline
$n$ & $\|e_{(u,v)} \|_{0}$ & $r=$ & $\|e_{(u,v)}\|_{1}$ & $r=$  &$\|e_{p} \|_{0}$& $r=$& $\|e_{(u,v)} \|_{0}$ & $r=$ & $\|e_{(u,v)}\|_{1}$ & $r=$  &$\|e_{p} \|_{0}$& $r=$ \\
\hline
4  &   8.04e-03  & -     &   2.32e-02 &  -     &   1.38e-02 &  -    & 1.25e-02 &  -     &   2.87e-02 &  -     &   1.72e-02 &  -    \\
8  &   2.16e-03  & 1.90  &   8.29e-03 &  1.49  &   4.99e-03 &  1.46 & 3.13e-03 &  2.00  &   8.57e-03 &  1.75  &   5.31e-03 &  1.69\\
16 &   5.65e-04  & 1.93  &   2.66e-03 &  1.64  &   1.33e-03 &  1.91 & 7.74e-04 &  2.02  &   2.45e-03 &  1.81  &   1.54e-03 &  1.79\\
32 &   1.45e-04  & 1.96  &   8.90e-04 &  1.58  &   3.83e-04 &  1.79 & 1.92e-04 &  2.01  &   6.90e-04 &  1.83  &   4.33e-04 &  1.83\\
\hline
\end{tabular}}
\end{center}
\end{table}

\subsubsection{Test Case 3: mixed Dirichlet/Neumann boundary conditions}
The computational domain in this test case is given by $\Omega=(0,1)\times (0,1)$, which is partitioned into different type of meshes of size $h$. The exact solution of the model problem is given by
\begin{equation*}
\begin{array}{lll}
\bm{u} =\left(
\begin{array}{lll}
\sin(x)\sin(y)\\
\cos(x)\cos(y)
\end{array}
\right),
\end{array}
\end{equation*}
\begin{figure}[!h]
\begin{center}
\includegraphics [width=0.35\textwidth]{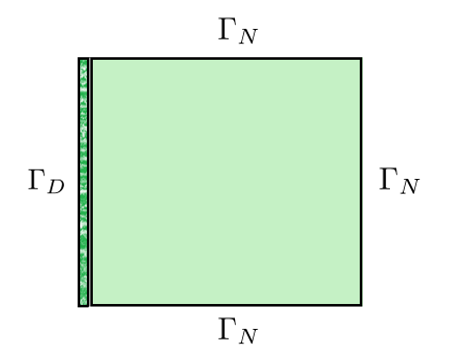}
\end{center}
\caption{Mixed Boundary conditions for test case 3.}
\label{fig.mixed_bounary}
\end{figure}
The right-hand side function is chosen to match the exact solution.
The Dirichlet boundary condition is imposed on the boundary $\Gamma_D=\{x=0,\; 0\leq y \leq 1\}$ and the Neumann boundary condition is imposed on $\Gamma_N=\Gamma/\Gamma_D$.

\begin{table}[!h]
\begin{center}
\caption{Error and convergence performance of the SWG scheme ($\kappa=1$, with tangential stability only on interior edges) for the Elasticity equation on polygonal meshes of test case 3. $r$ refers to the order of convergence in $O(h^r).$ }\label{table3.1}
{\tiny
\begin{tabular}{||c|cc|cc|cc|cc|cc|cc||}
\hline
\multicolumn{7}{|>{\columncolor{mypink}}c|}{ Triangular elements }&\multicolumn{6}{|>{\columncolor{green!15}}c|}{ Rectangular elements }\\
\hline
$n$ & $\|e_{(u,v)} \|_{0}$ & $r=$ & $\|e_{(u,v)}\|_{1}$ & $r=$  &$\|e_{p} \|_{0}$& $r=$& $\|e_{(u,v)} \|_{0}$ & $r=$ & $\|e_{(u,v)}\|_{1}$ & $r=$  &$\|e_{p} \|_{0}$& $r=$ \\
\hline
 4  &  1.23e-01 &  -     &   3.24e-01 &  -      &  1.02e-02 &  -     &  3.34e-01  & -    &    6.18e-01 &  -     &   1.93e-02 &  -   \\
 8  &  4.32e-02 &  1.51  &   1.23e-01 &  1.40   &  3.72e-03 &  1.46  &  1.82e-01  & 0.87 &    3.70e-01 &  0.74  &   1.16e-02 &  0.74\\
 16 &  1.21e-02 &  1.83  &   3.99e-02 &  1.62   &  1.12e-03 &  1.73  &  6.74e-02  & 1.44 &    1.47e-01 &  1.33  &   4.53e-03 &  1.36\\
 32 &  3.14e-03 &  1.95  &   1.38e-02 &  1.53   &  3.25e-04 &  1.79  &  1.94e-02  & 1.80 &    4.46e-02 &  1.72  &   1.37e-03 &  1.73\\
 64 &  7.95e-04 &  1.98  &   5.66e-03 &  1.29   &  1.01e-04 &  1.68  &  5.06e-03  & 1.94 &    1.22e-02 &  1.87  &   3.71e-04 &  1.88\\
\hline
\multicolumn{7}{|>{\columncolor{yellow!20}}c|}{ Hexagonal elements }&\multicolumn{6}{|>{\columncolor{blue!20}}c|}{ Polygonal elements}\\
\hline
$n$ & $\|e_{(u,v)} \|_{0}$ & $r=$ & $\|e_{(u,v)}\|_{1}$ & $r=$  &$\|e_{p} \|_{0}$& $r=$& $\|e_{(u,v)} \|_{0}$ & $r=$ & $\|e_{(u,v)}\|_{1}$ & $r=$  &$\|e_{p} \|_{0}$& $r=$ \\
\hline
 4  &  4.37e-01  & -      &  6.38e-01  & -     &   2.07e-02 &  -     & 5.96e-01 &  -     &   7.56e-01  & -     &   2.22e-02  & -   \\
 8  &  2.43e-01  & 0.85   &  3.83e-01  & 0.74  &   1.28e-02 &  0.69  & 3.62e-01 &  0.72  &   5.05e-01  & 0.58  &   1.52e-02  & 0.55\\
 16 &  9.22e-02  & 1.40   &  1.58e-01  & 1.28  &   5.03e-03 &  1.35  & 1.52e-01 &  1.25  &   2.28e-01  & 1.15  &   6.84e-03  & 1.15\\
 32 &  2.71e-02  & 1.77   &  4.95e-02  & 1.67  &   1.54e-03 &  1.71  & 4.68e-02 &  1.70  &   7.43e-02  & 1.62  &   2.24e-03  & 1.61\\
 64 &  7.09e-03  & 1.93   &  1.36e-02  & 1.86  &   4.27e-04 &  1.85  & 1.25e-02 &  1.91  &   2.08e-02  & 1.84  &   6.31e-04  & 1.82\\
\hline
\end{tabular}}
\end{center}
\end{table}

\begin{table}[!h]
\begin{center}
\caption{Error and convergence performance of the SWG scheme ($\kappa=1$, with tangential stability on all edges including the domain boundary) for the elasticity equation on polygonal meshes of test case 3. $r$ refers to the order of convergence in $O(h^r).$ }\label{table3.3}
{\tiny
\begin{tabular}{||c|cc|cc|cc|cc|cc|cc||}
\hline
\multicolumn{7}{|>{\columncolor{mypink}}c|}{ Triangular elements }&\multicolumn{6}{|>{\columncolor{green!15}}c|}{ Rectangular elements }\\
\hline
$n$ & $\|e_{(u,v)} \|_{0}$ & $r=$ & $\|e_{(u,v)}\|_{1}$ & $r=$  &$\|e_{p} \|_{0}$& $r=$& $\|e_{(u,v)} \|_{0}$ & $r=$ & $\|e_{(u,v)}\|_{1}$ & $r=$  &$\|e_{p} \|_{0}$& $r=$ \\
\hline
 4  &  2.41e-01 &  0.00  &   5.64e-01 &  0.00  &   1.72e-02 &  0.00  &  4.22e-01 &  0.00  &   7.43e-01 &  0.00  &   2.25e-02 &  0.00 \\
 8  &  1.06e-01 &  1.18  &   2.71e-01 &  1.06  &   8.20e-03 &  1.07  &  2.52e-01 &  0.74  &   4.98e-01 &  0.58  &   1.54e-02 &  0.55\\
 16 &  3.41e-02 &  1.64  &   9.53e-02 &  1.51  &   2.85e-03 &  1.53  &  1.05e-01 &  1.26  &   2.26e-01 &  1.14  &   6.93e-03 &  1.15\\
 32 &  9.27e-03 &  1.88  &   2.91e-02 &  1.71  &   8.47e-04 &  1.75  &  3.24e-02 &  1.70  &   7.43e-02 &  1.61  &   2.26e-03 &  1.62\\
 64 &  2.38e-03 &  1.96  &   9.05e-03 &  1.69  &   2.41e-04 &  1.81  &  8.63e-03 &  1.91  &   2.10e-02 &  1.83  &   6.38e-04 &  1.83\\
\hline
\multicolumn{7}{|>{\columncolor{yellow!20}}c|}{ Hexagonal elements }&\multicolumn{6}{|>{\columncolor{blue!20}}c|}{ Polygonal elements}\\
\hline
$n$ & $\|e_{(u,v)} \|_{0}$ & $r=$ & $\|e_{(u,v)}\|_{1}$ & $r=$  &$\|e_{p} \|_{0}$& $r=$& $\|e_{(u,v)} \|_{0}$ & $r=$ & $\|e_{(u,v)}\|_{1}$ & $r=$  &$\|e_{p} \|_{0}$& $r=$ \\
\hline
 4  &  5.16e-01  & -     &   7.33e-01 &  -     &   2.34e-02 &  -     & 5.95e-01 &  -     &   7.54e-01 &  -    &    2.26e-02 &  -       \\
 8  &  3.12e-01  & 0.73  &   4.83e-01 &  0.60  &   1.60e-02 &  0.54  & 3.72e-01 &  0.68  &   5.19e-01 &  0.54 &    1.58e-02 &  0.51    \\
 16 &  1.31e-01  & 1.25  &   2.22e-01 &  1.12  &   7.04e-03 &  1.19  & 1.61e-01 &  1.21  &   2.41e-01 &  1.11 &    7.30e-03 &  1.12    \\
 32 &  4.06e-02  & 1.69  &   7.41e-02 &  1.58  &   2.28e-03 &  1.62  & 5.02e-02 &  1.68  &   7.97e-02 &  1.60 &    2.41e-03 &  1.60    \\
 64 &  1.09e-02  & 1.90  &   2.10e-02 &  1.82  &   6.51e-04 &  1.81  & 1.35e-02 &  1.90  &   2.24e-02 &  1.83 &    6.78e-04 &  1.83    \\
\hline
\end{tabular}}
\end{center}
\end{table}

\begin{table}[!h]
\begin{center}
\caption{Error and convergence performance of the SWG scheme ($\kappa=0$, no tangential stability) for the elasticity equation on polygonal meshes of test case 3. $r$ refers to the order of convergence in $O(h^r).$ }\label{table3.2}
{\tiny
\begin{tabular}{||c|cc|cc|cc|cc|cc|cc||}
\hline
\multicolumn{7}{|>{\columncolor{mypink}}c|}{ Triangular elements }&\multicolumn{6}{|>{\columncolor{green!15}}c|}{ Rectangular elements }\\
\hline
$n$ & $\|e_{(u,v)} \|_{0}$ & $r=$ & $\|e_{(u,v)}\|_{1}$ & $r=$  &$\|e_{p} \|_{0}$& $r=$& $\|e_{(u,v)} \|_{0}$ & $r=$ & $\|e_{(u,v)}\|_{1}$ & $r=$  &$\|e_{p} \|_{0}$& $r=$ \\
\hline
 4  &  -    &  -     &   -    &  -      &  -    &  -     &  4.18e-02 &  -     &   5.87e-02 &  -     &   1.27e-03 &  -    \\
 8  &  -    &  -     &   -    &  -      &  -    &  -     &  9.28e-03 &  2.17  &   1.38e-02 &  2.08  &   3.24e-04 &  1.98\\
 16 &  -    &  -     &   -    &  -      &  -    &  -     &  2.20e-03 &  2.08  &   3.38e-03 &  2.03  &   8.16e-05 &  1.99\\
 32 &  -    &  -     &   -    &  -      &  -    &  -     &  5.36e-04 &  2.04  &   8.39e-04 &  2.01  &   2.05e-05 &  1.99\\
\hline
\multicolumn{7}{|>{\columncolor{yellow!20}}c|}{ Hexagonal elements }&\multicolumn{6}{|>{\columncolor{blue!20}}c|}{ Polygonal elements}\\
\hline
$n$ & $\|e_{(u,v)} \|_{0}$ & $r=$ & $\|e_{(u,v)}\|_{1}$ & $r=$  &$\|e_{p} \|_{0}$& $r=$& $\|e_{(u,v)} \|_{0}$ & $r=$ & $\|e_{(u,v)}\|_{1}$ & $r=$  &$\|e_{p} \|_{0}$& $r=$ \\
\hline
 4  &  1.13e-02 &  -     &   1.87e-02 &  -     &   4.61e-04 &  -     & 4.29e-02 &  -    &    5.52e-02 &  -    &    1.80e-03  & -   \\
 8  &  2.58e-03 &  2.14  &   6.20e-03 &  1.59  &   1.46e-04 &  1.65  & 1.11e-02 &  1.95 &    1.69e-02 &  1.71 &    5.66e-04  & 1.67\\
 16 &  7.12e-04 &  1.86  &   2.29e-03 &  1.44  &   4.24e-05 &  1.78  & 2.79e-03 &  1.99 &    4.71e-03 &  1.84 &    1.65e-04  & 1.78\\
 32 &  1.72e-04 &  2.04  &   7.96e-04 &  1.53  &   1.33e-05 &  1.67  & 6.94e-04 &  2.01 &    1.26e-03 &  1.90 &    4.70e-05  & 1.81\\
\hline
\end{tabular}}
\end{center}
\end{table}

Tables \ref{table3.1} and \ref{table3.3} illustrate the numerical performance of the SWG scheme when the tangential stability $S_2(\cdot)$ is included (i.e., $\kappa =1.0$) in the scheme. Table \ref{table3.1} is concerned with tangential stability only on interior edges, and Table \ref{table3.3} includes all the edges. It can be seen that the numerical displacement  converges at the order of $r = 2$ in the $L^2$ norm. For the numerical pressure, the convergence at the cell center is of the order $r>1.5$, and for the numerical displacement in $H^1$ norm, the rate of convergence is at $r\approx 1$ for all element types.

Table \ref{table3.2} illustrates the numerical performance of the SWG scheme when no tangential stability is employed (i.e., $\kappa=0$). Our code shows that the global stiffness matrix of the SWG scheme obtained with the use of triangular elements is singular so that no results are possible for reporting. However, for all other polygonal elements (rectangular, hexagonal, octagonal etc), the numerical displacement converges at the optimal order of $r = 2$ in the $L^2$ norm. For the numerical pressure, the convergence at the cell centers is of the order of $r>1.5$. For the numerical displacement in $H^1$ norm, the convergence is at the order of $r>1.5$. We conjecture that the WG scheme is stable for truly polygonal elements even without the use of tangential stability in the method.

\subsubsection{Cook's membrane test case}
The Cook's membrane is a classical bending dominated test case employed to validate the susceptibility of linear elastic solvers to volumetric locking. The problem consists of a tapered plate clamped on one end and subject to a shear load, taken as $\bm{g} = (0, 1/16)$ in our numerical experiments, on the opposite end, as illustrated in Figure \ref{fig.cook_membrane}.

\begin{figure}[!ht]
\begin{center}
\includegraphics [width=0.35\textwidth]{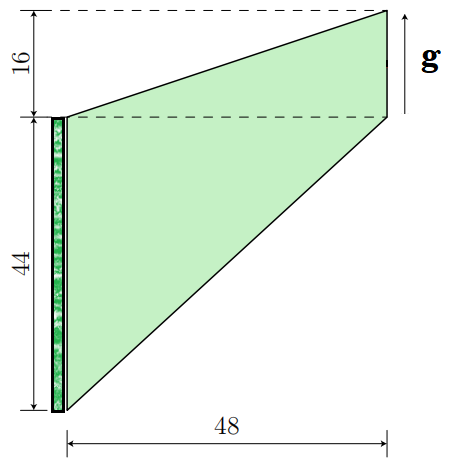}
\end{center}
\caption{Domain and parameter settings of the Cook's membrane computation.}
\label{fig.cook_membrane}
\end{figure}

We considered two cases for demonstrating the performance of the SWG method proposed in this paper. The first case involves a material with Young's modulus $E = 1$ and Poisson ratio $\nu = 1/3$, and the second case is concerned with a nearly incompressible material with Young's modulus $E = 1.12499998125$ and Poisson ratio $\nu = 0.499999975$. Our numerical tests are conducted on several polygonal partitions shown as in Figure \ref{fig:testcaseCK_polymesh}.

\begin{figure}[!h]
             \centering
             \subfigure[Triangle mesh ]{\label{Fig.sub8.4.l}
             \includegraphics [width=0.215\textwidth]{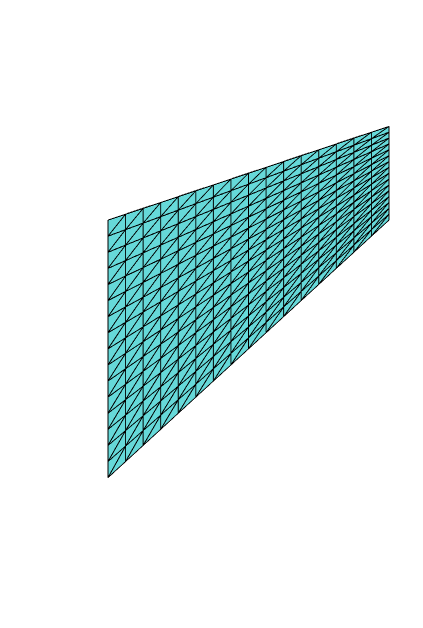}
             }
             \subfigure[Quadrilateral mesh]{\label{Fig.sub8.4.2}
             \includegraphics [width=0.215\textwidth]{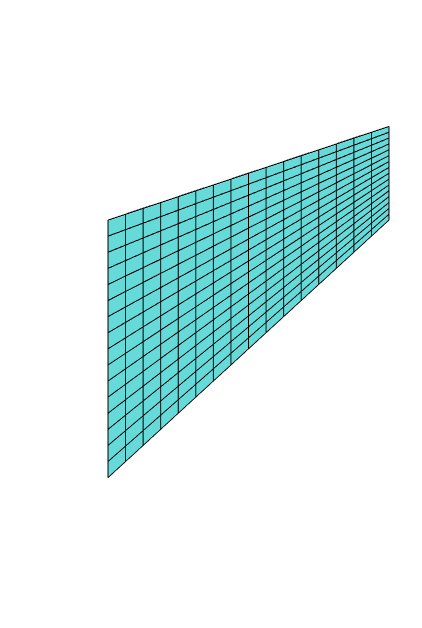}
             }
             \subfigure[Hexagonal mesh]{\label{Fig.sub8.4.3}
             \includegraphics [width=0.215\textwidth]{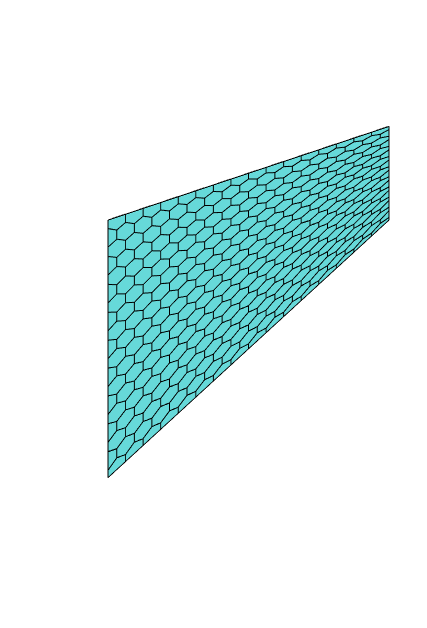}
             \hskip5pt
             }
             \subfigure[Octagonal mesh]{\label{Fig.sub8.4.4}
             \includegraphics [width=0.215\textwidth]{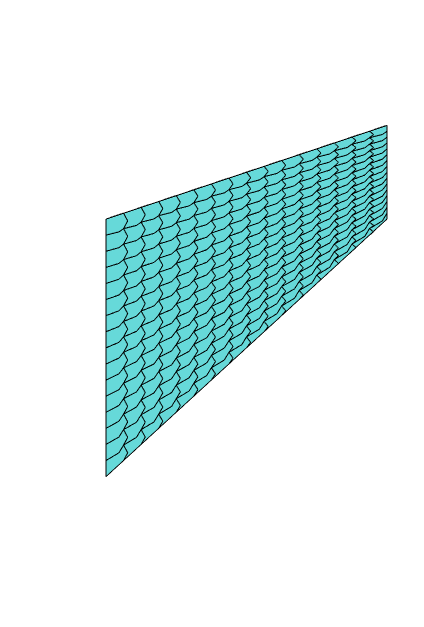}
             }
            \caption{Polygonal partitions employed in the Cook's membrane computation.}
            \label{fig:testcaseCK_polymesh}
\end{figure}

Out numerical experiment shows that the SWG method works really well for both cases on all these polygonal partitions when tangential stability term $S_2(\cdot)$ (i.e., $\kappa =1$) is chosen. Figure \ref{fig:testcaseCKdef_polymesh} illustrates the deformation of the material colored by the pseudo pressure field for the case of $E = 1.12499998125$ and Poisson ratio $\nu = 0.499999975$.

\begin{figure}[!h]
             \centering
             \subfigure[Triangles ]{\label{Fig.sub88.4.l}
             \includegraphics [width=0.215\textwidth]{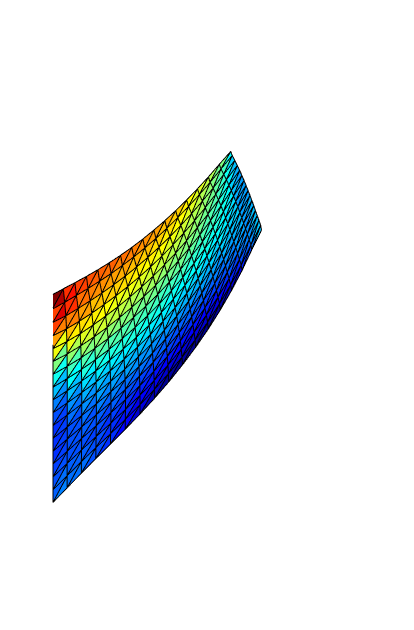}
             }
             \subfigure[Quadrilaterals]{\label{Fig.sub88.4.2}
             \includegraphics [width=0.215\textwidth]{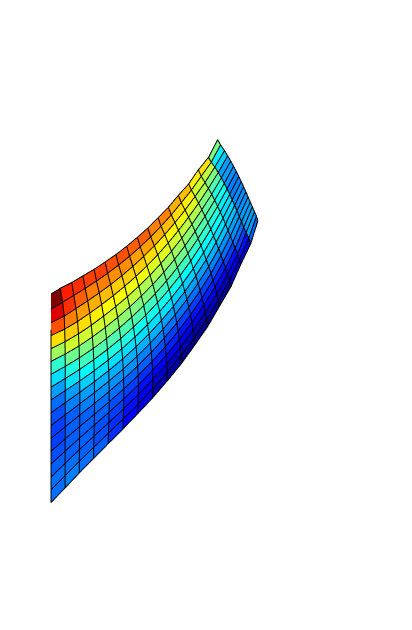}
             }
             \subfigure[Hexagons]{\label{Fig.sub88.4.3}
             \includegraphics [width=0.215\textwidth]{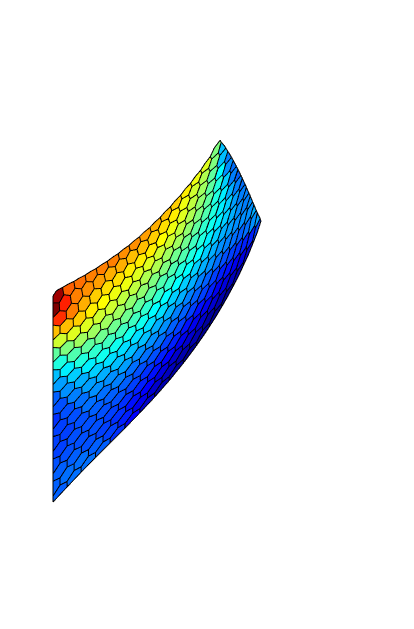}
             \hskip5pt
             }
             \subfigure[Octagons]{\label{Fig.sub88.4.4}
             \includegraphics [width=0.215\textwidth]{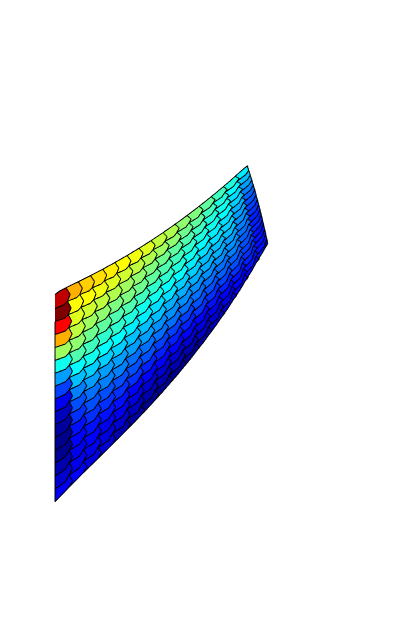}
             }
            \caption{Cook's membrane computation with the scheme \eqref{equ.elasticity-Mix-SWG}, deformation and pressure field: (a), (b), (c), (d) with various polygonal partitions, $h=1/16$,$E = 1.12499998125$, $\nu = 0.499999975$, $\kappa=1.$}
            \label{fig:testcaseCKdef_polymesh}
\end{figure}

The SWG method without tangential stability (i.e., $\kappa=0$ in the stabilizer $S(\cdot)$) was also tested in the Cook's membrane computing. Recall that the numerical results from Test Cases 1 and 2 indicate the method is indeed stable and accurate for any non-triangular finite element partitions. The numerical results for the Cook's membrane problem obtained from the SWG method with $\kappa=0$ are shown in Figure \ref{fig:testcaseCK_deformation_pressure}. It can be seen that, when quadrilateral elements are employed, the approximate solution obtained from the SWG scheme with $\kappa=0$ are very close to those obtained with $\kappa=1$ for both cases; the performance of the SWG method is equally satisfactory on all other polygonal partitions that we tested.

As there is no analytical solution available, the vertical displacement at the mid point of the right end of the plate, $Q = (48, 52)$, is compared against the reference values reported in \cite{Ferdinando_2005,Sevilla_linearelastostatic_2019}, given by $16.442$ for the second test case.

\begin{figure}[!h]
             \centering
\subfigure[ $\kappa=1$,~~~~~ case 1: $E = 1$, $\nu = 1/3$,~~~~~~~~~~~~~~~~case 2: $E = 1.12499998125$, $\nu = 0.499999975$,~~~~~~~~]{\label{Fig.sub7.4.1}
             \includegraphics [width=0.475\textwidth]{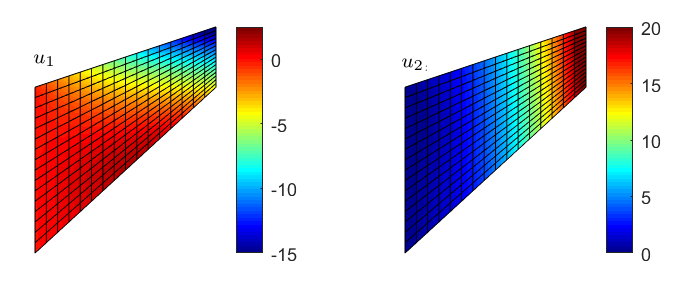}
             \includegraphics [width=0.475\textwidth]{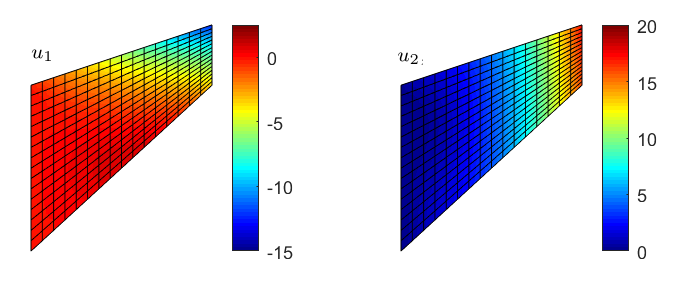}
             }
             \subfigure[$\kappa=0$,~~~~~ case 1: $E = 1$, $\nu = 1/3$,~~~~~~~~~~~~~~~~case 2: $E = 1.12499998125$, $\nu = 0.499999975$,~~~~~~~~ ]{\label{Fig.sub7.4.2}
             \includegraphics [width=0.475\textwidth]{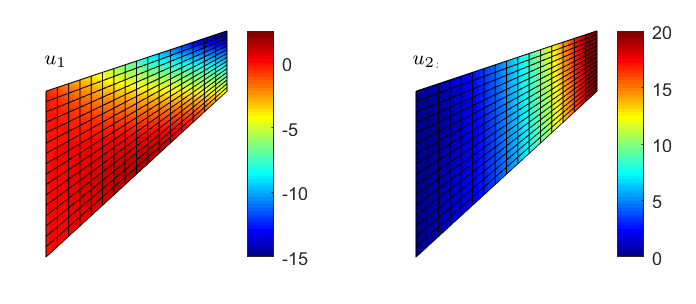}
             \includegraphics [width=0.475\textwidth]{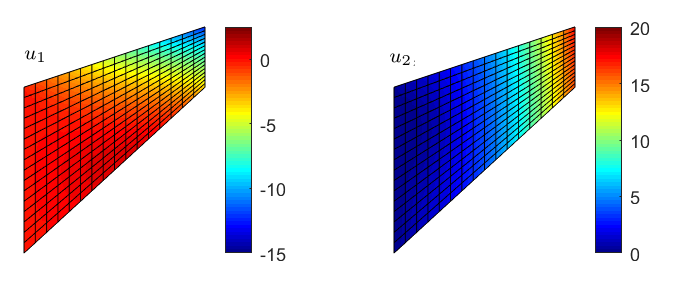}
             }
            \caption{Cook's membrane computation using the SWG scheme \eqref{equ.elasticity-Mix-SWG} with $\kappa=1$ and $\kappa=0$, contour plots of the displacement $u_1$ and $u_2$, quadrilateral partition with $h=1/16$.}
            \label{fig:testcaseCK_deformation_pressure}
\end{figure}

Figure \ref{fig:testcaseCK_convergence} shows the convergence of the vertical displacement at the point $Q$ for the second case with $\nu = 0.499999975$ on four different type of polygonal partitions. The results indicate that convergence occurs in the fifth refinement for the triangular mesh, with $12,416$ degree of freedom (i.e. $8,192$ triangular elements). The computed displacement at the mid point of the right end of the plate is within a $1\%$ difference with respect to the results reported in \cite{Ferdinando_2005,Sevilla_linearelastostatic_2019}. The numerical solution converges with a bit more degree of freedoms for the quadrilateral partitions, and so as for other polygonal meshes. In general, we need less than $18,000$ degree of freedoms to achieve an error less than $5\%$ with respect to the reference value for all type of meshes.

\begin{figure}[!h]
\centering
\includegraphics [width=0.675\textwidth]{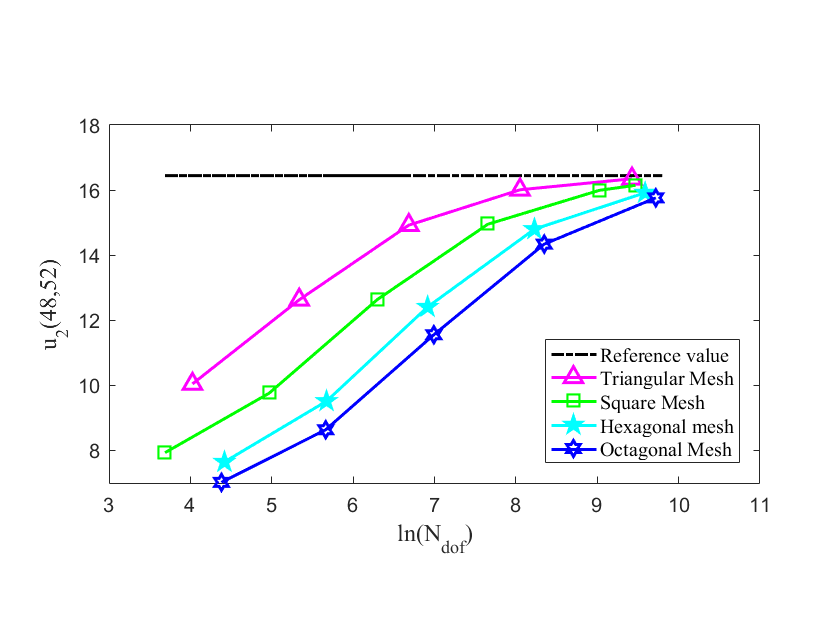}
\caption{Cook's membrane problem: evolution of the vertical displacement at the mid point of the right end of the plate as a function of the total number of degrees of freedom with various polygonal partitions.}\label{fig:testcaseCK_convergence}
\end{figure}

\subsection{Numerical experiments in 3D}
For the numerical experiments in 3D, we first consider some synthetic test cases followed by the shear loaded beam bench mark.
The computational domain is partitioned into different type of meshes, including tetrahedron, cube and hexahedron prism as shown in Figure \ref{fig:testcase3D_polymesh}.

\begin{figure}[!h]
             \subfigure[Tetrahedral mesh ]{\label{Fig.sub12.4.l}
             \includegraphics [width=0.305\textwidth]{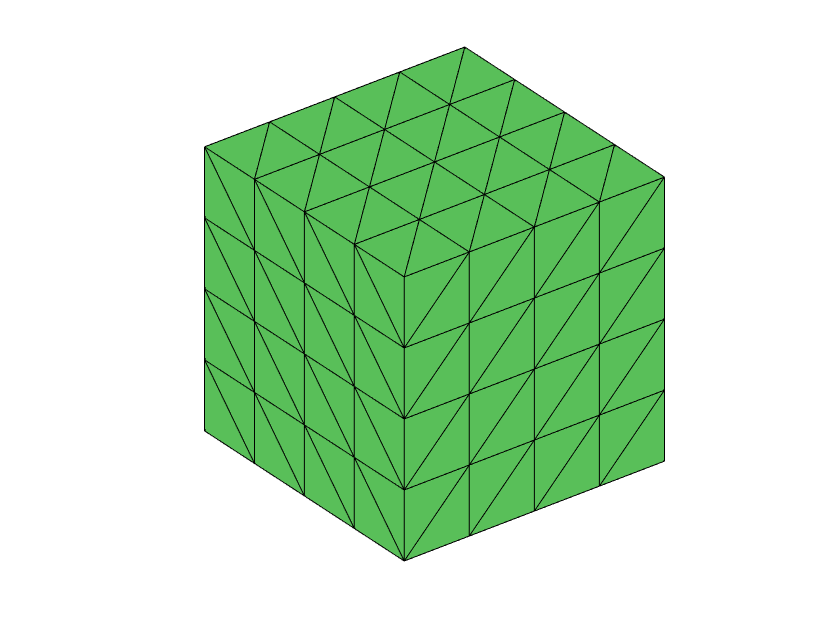}
             }
             \subfigure[Cubic mesh]{\label{Fig.sub12.4.2}
             \includegraphics [width=0.305\textwidth]{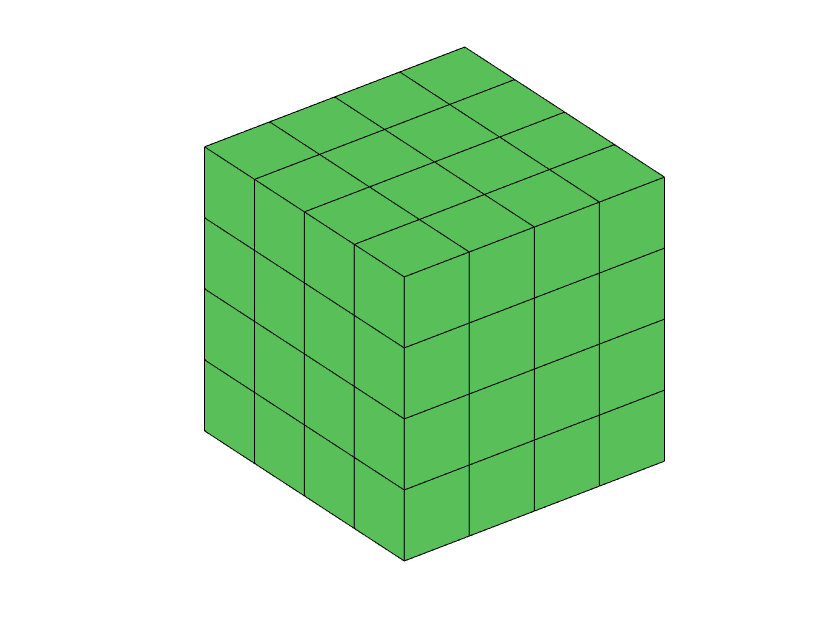}
             }
             \subfigure[Hexagon-based prismal mesh]{\label{Fig.sub12.4.3}
             \includegraphics [width=0.305\textwidth]{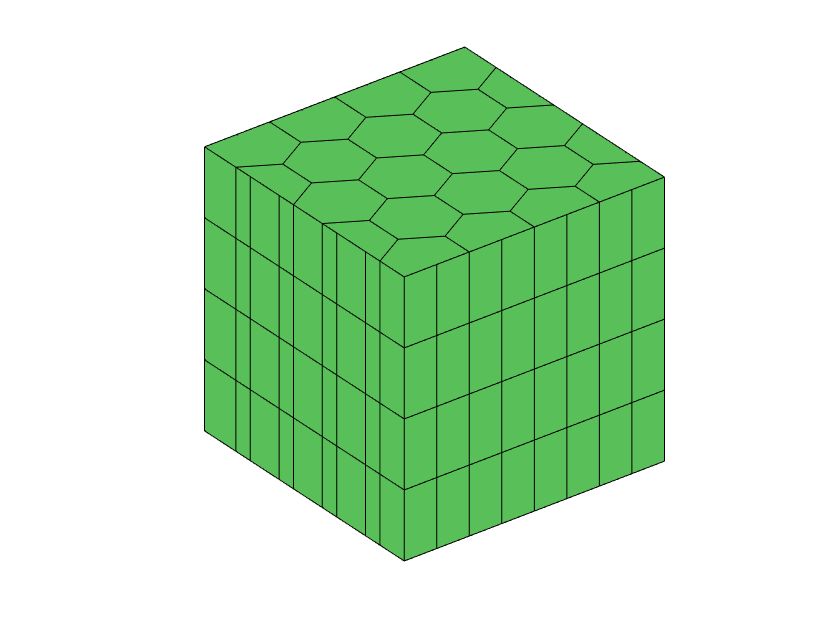}
             }
            \caption{Polyhedral partitions for $3D$ domains.}
            \label{fig:testcase3D_polymesh}
\end{figure}

As the numerical method was implemented in the MatLab platform, the computation for 3D problems is only performed on relatively coarse meshes.
For 3D problems, we note that tetrahedral partitions involve much more degree of freedoms than other polyhedral partitions (e.g., cubic partition) and are thus computationally more expensive than others.

\subsubsection{Test Case 3}\label{testcase3D1}
The computational domain in this 3D test is given by $\Omega=(0,1)\times (0,1)\times (0,1)$, which is uniformly partitioned into polyhedral elements of size $h$. The exact solution of the model problem is given by
\begin{equation*}
\begin{array}{lll}
\bm{u} =\left(
\begin{array}{lll}
\sin(x)\cos(y)\\
\cos(x)\sin(y)\\
(z-0.5)^2
\end{array}
\right),\\
\end{array}
\end{equation*}
The right-hand side function in the linear elasticity equation and the Dirichlet boundary data are chosen to match the exact solution.
For simplicity, the test problem assumed the following parameters: $E=1$ and $\nu =0.1$.

\begin{figure}[!h]
             \centering
             \subfigure[Cubic mesh]{\label{Fig.sub3.4.2}
             \includegraphics [width=0.3\textwidth]{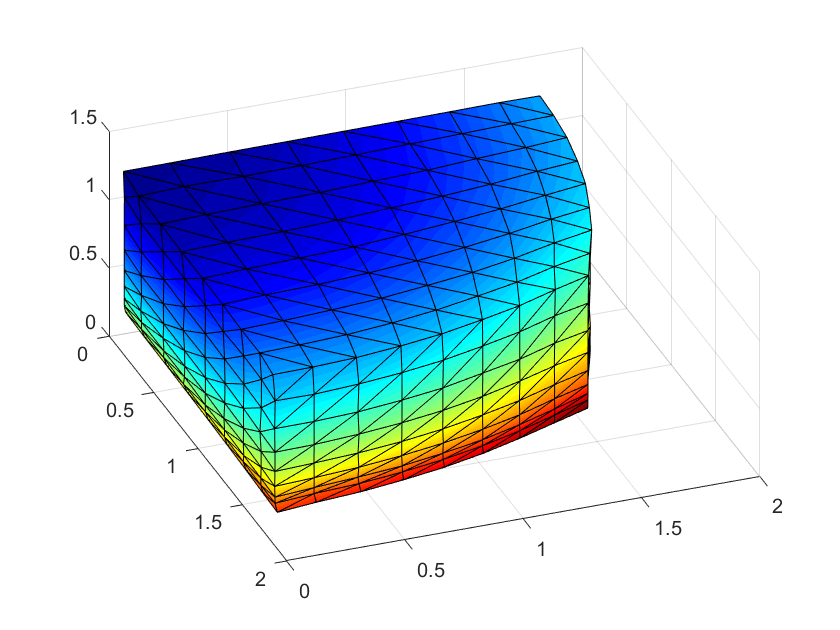}}
             \subfigure[Tetrahedral mesh]{\label{Fig.sub3.4.3}
             \includegraphics [width=0.3\textwidth]{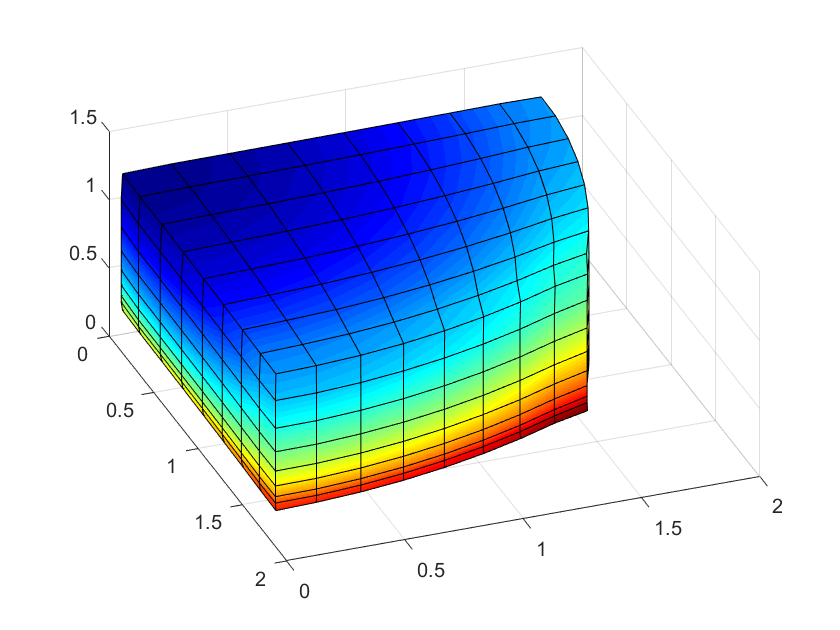}}
             \subfigure[Hexagon-based prismal mesh]{\label{Fig.sub3.4.4}
             \includegraphics [width=0.3\textwidth]{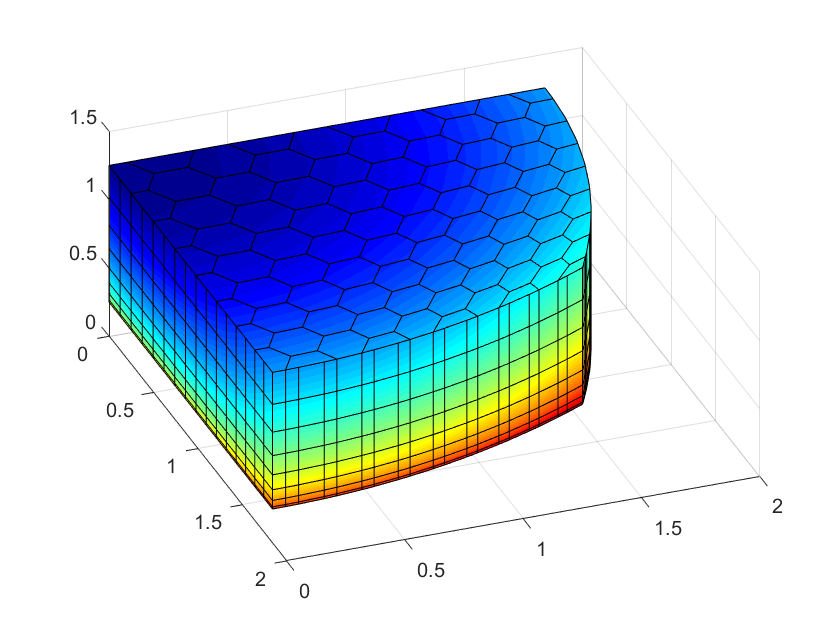}}
            \caption{Numerical results of test case 3 in 3D for the SWG scheme \eqref{equ.elasticity-Mix-SWG}, deformation of the elastic material with color provided by the magnitude of the numerical pseudo-pressure $p_h$: (a), (b), (c), polyhedral partitions with $h=1/8$.}
            \label{fig:testcase3_deformation_pressure}
\end{figure}

Tables \ref{table3D.11} and \ref{table3D.13} illustrate the numerical performance of the SWG scheme when tangential stability is included in the stabilizer $S(\cdot)$ (i.e., $\kappa =1$). Table \ref{table3D.11} is concerned with the addition of the tangential stability at all interior faces, while Table \ref{table3D.13} includes the tangential stability at all faces. It is seen that the numerical displacement converges at the order of $r \in [1.5,2]$ in the $L^2$ norm. For the numerical pressure, the convergence at the cell centers is seen to have the order of $r = 1$. For the numerical displacement in $H^1$ norm, the convergence is at the order of $r = 1$.

\begin{table}[!h]
\begin{center}
\caption{Error and convergence performance of the SWG scheme ($\kappa =1$, tangential stability at only interior faces) for the elasticity equation on polyhedral partitions of test case 3 in 3D. $r$ refers to the order of convergence in $O(h^r).$ }\label{table3D.11}
{\tiny
\begin{tabular}{||c|cc|cc|cc|cc|cc|cc||}
\hline
\multicolumn{7}{|>{\columncolor{mypink}}c|}{ Tetrahedral elements }&\multicolumn{6}{|>{\columncolor{green!15}}c|}{ Cubic elements }\\
\hline
$n$ & $\|e_{(u,v)} \|_{0}$ & $r=$ & $\|e_{(u,v)}\|_{1}$ & $r=$  &$\|e_{p} \|_{0}$& $r=$& $\|e_{(u,v)} \|_{0}$ & $r=$ & $\|e_{(u,v)}\|_{1}$ & $r=$  &$\|e_{p} \|_{0}$& $r=$ \\
\hline
2 & 1.06e-01 &   -     &  9.20e-01  &   -   &   3.61e-02  &  -  &   1.17e-02 &   -    &   3.47e-02 &   -    &   9.24e-04 &   -   \\
4 & 3.81e-02 &  1.47   &  6.24e-01  & 0.56  &   2.02e-02  & 0.84&   2.91e-03 &  2.01  &   1.15e-02 &  1.60  &   8.03e-04 &  0.20\\
8 & 1.15e-02 &  1.73   &  3.62e-01  & 0.78  &   1.01e-02  & 0.99&   6.43e-04 &  2.18  &   3.10e-03 &  1.89  &   2.48e-04 &  1.69\\
\hline
\multicolumn{7}{|>{\columncolor{yellow!20}}c|}{ Hexagon-based prismal elements }\\
\cline{1-7}
$n$ & $\|e_{(u,v)} \|_{0}$ & $r=$ & $\|e_{(u,v)}\|_{1}$ & $r=$  &$\|e_{p} \|_{0}$& $r=$ \\
\cline{1-7}
2 &    3.64e-02 &    -   &   1.04e-01 &    -  &    6.01e-03 &    - \\
4 &    6.82e-03 &  2.42  &   3.18e-02 &  1.71 &    1.76e-03 &  1.77\\
8 &    1.55e-03 &  2.13  &   1.29e-02 &  1.30 &    7.20e-04 &  1.29\\
\cline{1-7}
\end{tabular}}
\end{center}
\end{table}

\begin{table}[!h]
\begin{center}
\caption{Error and convergence performance of the SWG scheme ($\kappa =1$, tangential stability at all faces) for the elasticity equation on polyhedral meshes of test case 3 in 3D. $r$ refers to the order of convergence in $O(h^r).$ }\label{table3D.13}
{\tiny
\begin{tabular}{||c|cc|cc|cc|cc|cc|cc||}
\hline
\multicolumn{7}{|>{\columncolor{mypink}}c|}{ Tetrahedral elements }&\multicolumn{6}{|>{\columncolor{green!15}}c|}{ Cubic elements }\\
\hline
$n$ & $\|e_{(u,v)} \|_{0}$ & $r=$ & $\|e_{(u,v)}\|_{1}$ & $r=$  &$\|e_{p} \|_{0}$& $r=$& $\|e_{(u,v)} \|_{0}$ & $r=$ & $\|e_{(u,v)}\|_{1}$ & $r=$  &$\|e_{p} \|_{0}$& $r=$ \\
\hline
2 & 1.14e-01 &    -   &   9.77e-01 &    -   &   4.14e-02 &    - &   2.76e-02 &    -   &   8.02e-02 &    -    &  5.57e-03 &    - \\
4 & 3.99e-02 &  1.51  &   6.46e-01 &  0.60  &   2.24e-02 &  0.88&   8.03e-03 &  1.78  &   2.84e-02 &  1.50   &  2.06e-03 &  1.43\\
8 & 1.18e-02 &  1.76  &   3.70e-01 &  0.81  &   1.08e-02 &  1.05&   1.64e-03 &  2.29  &   7.40e-03 &  1.94   &  5.22e-04 &  1.98\\
\hline
\multicolumn{7}{|>{\columncolor{yellow!20}}c|}{ Hexagon-based prismal elements }\\
\cline{1-7}
$n$ & $\|e_{(u,v)} \|_{0}$ & $r=$ & $\|e_{(u,v)}\|_{1}$ & $r=$  &$\|e_{p} \|_{0}$& $r=$ \\
\cline{1-7}
2 &    6.96e-02 &    -   &   1.90e-01  &   -   &   1.41e-02 &    -   \\
4 &    1.43e-02 &  2.29  &   5.55e-02  & 1.77  &   3.57e-03 &  1.98\\
8 &    3.07e-03 &  2.21  &   2.28e-02  & 1.28  &   1.44e-03 &  1.30\\
\cline{1-7}
\end{tabular}}
\end{center}
\end{table}

Table \ref{table3D.12} illustrates the numerical performance of the SWG scheme with $\kappa=0$ (i.e., no tangential stability). Our computation indicates that the stiffness matrices for the SWG scheme on tetrahedral partitions are all singular so that no numerical results are possible to report. For other polyhedral partitions, such as the cubic and hexagon-based prismal elements, the numerical displacement turns out to be quite accurate and furthermore convergent to the exact solution at the optimal order of $r = 2$ in the $L^2$ norm. For the numerical pressure and for the numerical displacement in $H^1$ norm, the convergence is at the order of $r \ge 1$, and some are at $r=1.5$.

\begin{table}[!h]
\begin{center}
\caption{Error and convergence performance of the SWG scheme ($\kappa =0$, no tangential stability) for the elasticity equation on polyhedral partitions of test case 3 in 3D. $r$ refers to the order of convergence in $O(h^r).$ }\label{table3D.12}
{\tiny
\begin{tabular}{||c|cc|cc|cc|cc|cc|cc||}
\hline
\multicolumn{7}{|>{\columncolor{mypink}}c|}{ Tetrahedral elements }&\multicolumn{6}{|>{\columncolor{green!15}}c|}{ Cubic elements }\\
\hline
$n$ & $\|e_{(u,v)} \|_{0}$ & $r=$ & $\|e_{(u,v)}\|_{1}$ & $r=$  &$\|e_{p} \|_{0}$& $r=$& $\|e_{(u,v)} \|_{0}$ & $r=$ & $\|e_{(u,v)}\|_{1}$ & $r=$  &$\|e_{p} \|_{0}$& $r=$ \\
\hline
2 & -  & -  & -  & -  & -  & -  &   5.74e-03  &  -     &  1.87e-02 &   -    &   4.27e-04 &   -  \\
4 & -  & -  & -  & -  & -  & -  &   1.34e-03  & 2.10   &  4.97e-03 &  1.91  &   1.89e-04 &  1.17\\
8 & -  & -  & -  & -  & -  & -  &   3.30e-04  & 2.02   &  1.36e-03 &  1.87  &   5.92e-05 &  1.68\\
\hline
\multicolumn{7}{|>{\columncolor{yellow!20}}c|}{ Hexagon-based prismal elements }\\
\cline{1-7}
$n$ & $\|e_{(u,v)} \|_{0}$ & $r=$ & $\|e_{(u,v)}\|_{1}$ & $r=$  &$\|e_{p} \|_{0}$& $r=$ \\
\cline{1-7}
2 &    2.24e-02 &  -     &   7.89e-02  & -     &   4.25e-03 &  -   \\
4 &    5.40e-03 &  2.05  &   2.80e-02  & 1.49  &   1.39e-03 &  1.61\\
6 &    1.29e-03 &  2.06  &   1.12e-02  & 1.32  &   5.81e-04 &  1.26\\
\cline{1-7}
\end{tabular}}
\end{center}
\end{table}
\subsubsection{Test Case 4}
The computational domain for test case 4 is given by $\Omega=(0,1)\times (0,1)\times (0,1)$. The exact solution for the model problem is given by
\begin{equation*}
\begin{array}{lll}
\bm{u} =\left(
\begin{array}{c}
-xyz\\
3\mu z(x^2-y^2) - z^3\\
3yz^2 +\mu y(y^2-3x^2)
\end{array}
\right).\\
\end{array}
\end{equation*}
The right-hand side function in the linear elasticity equation and the Dirichlet boundary data are computed to match the exact solution.
In addition, we have $E=1$ and $\nu =0.1$.

\begin{figure}[!h]
             \centering
             \subfigure[Tetrahedral mesh]{\label{Fig.sub4.4.2}
             \includegraphics [width=0.3\textwidth]{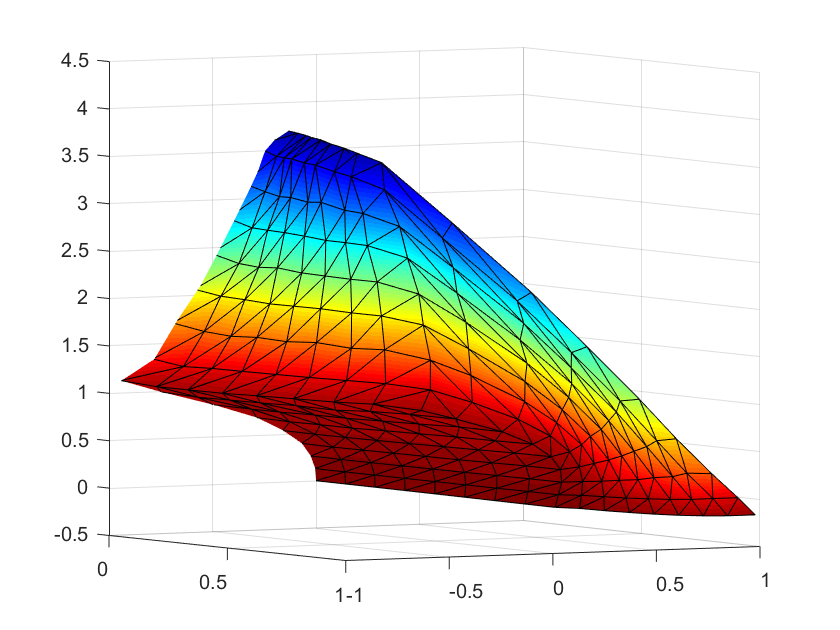}}
             \subfigure[Cubic mesh]{\label{Fig.sub4.4.3}
             \includegraphics [width=0.3\textwidth]{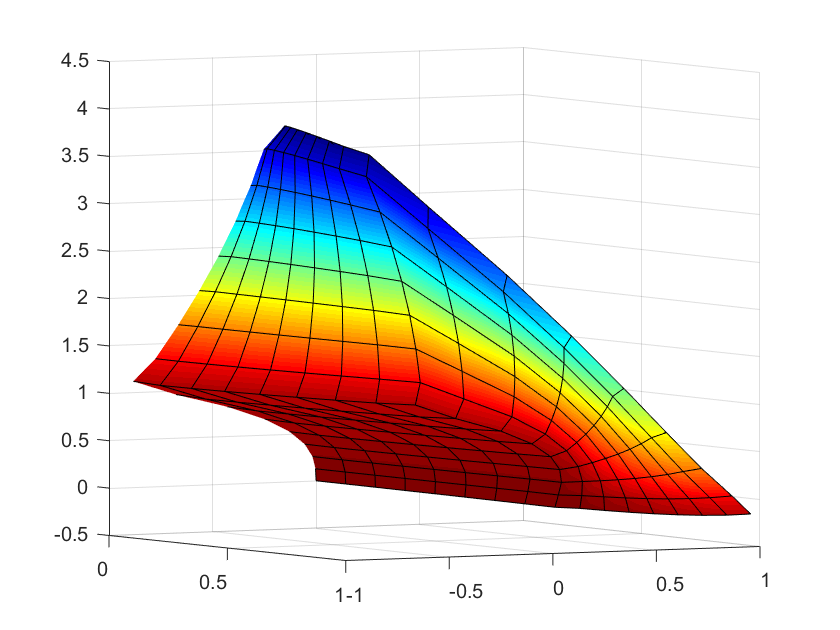}}
             \subfigure[Hexagon-based prismal mesh]{\label{Fig.sub4.4.4}
             \includegraphics [width=0.3\textwidth]{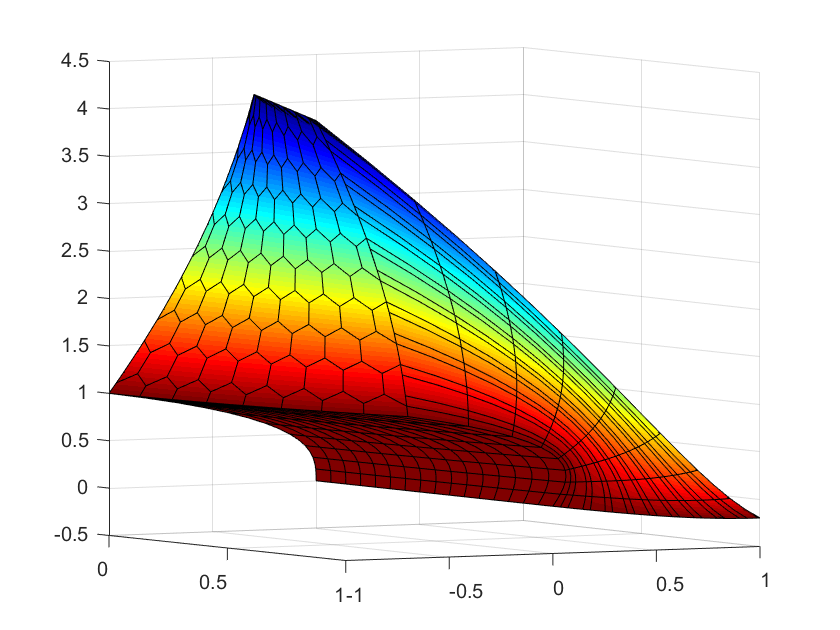}}
            \caption{Numerical results for the SWG scheme \eqref{equ.elasticity-Mix-SWG} for test case 4 in 3D, deformation plot of the elastic material with color provided by the magnitude of the numerical pseudo-pressure $p_h$: (a), (b), (c), polyhedral partition with $h=1/16$.}
            \label{fig:testcase4_deformation_pressure}
\end{figure}

Table \ref{table3D.21} illustrates the numerical performance of the SWG scheme when tangential stability is added on interior faces (i.e., $\kappa =1.0)$. It is suggested that the numerical displacement converges at the order of $r = 2$ in the $L^2$ norm. For the numerical pressure, the order of convergence at the cell centers is of $r\approx 1$ or better for cubical partitions and the hexagon-based prismal partiion. For the numerical displacement in $H^1$ norm, the convergence is at the order of $r \ge 1$. The convergence for the numerical displacement in $H^1$ norm is lower than the optimal order of $r=1$ for tetrahedral elements, and this might be caused by the missing of tangential stability on the boundary faces.

\begin{table}[!h]
\begin{center}
\caption{Error and convergence performance of the SWG scheme ($\kappa=1$, tangential stability on interior faces only) for the elasticity equation on polyhedral partitions of test case 4 in 3D. $r$ refers to the order of convergence in $O(h^r).$ }\label{table3D.21}
{\tiny
\begin{tabular}{||c|cc|cc|cc|cc|cc|cc||}
\hline
\multicolumn{7}{|>{\columncolor{mypink}}c|}{ Tetrahedral elements }&\multicolumn{6}{|>{\columncolor{green!15}}c|}{ Cubic elements }\\
\hline
$n$ & $\|e_{(u,v)} \|_{0}$ & $r=$ & $\|e_{(u,v)}\|_{1}$ & $r=$  &$\|e_{p} \|_{0}$& $r=$& $\|e_{(u,v)} \|_{0}$ & $r=$ & $\|e_{(u,v)}\|_{1}$ & $r=$  &$\|e_{p} \|_{0}$& $r=$ \\
\hline
2 &2.35e-01 &  0.00   &  2.01e+00  & 0.00   &  8.99e-02  & 0.00&   1.04e-01 &   -    &   3.02e-01  &  -     &  8.47e-03 &   -  \\
4 &9.47e-02 &  1.31   &  1.49e+00  & 0.43   &  5.34e-02  & 0.75&   2.70e-02 &  1.95  &   9.67e-02  & 1.64   &  3.06e-03 &  1.47\\
8 &3.03e-02 &  1.64   &  9.27e-01  & 0.69   &  2.77e-02  & 0.95&   6.79e-03 &  1.99  &   2.90e-02  & 1.74   &  9.56e-04 &  1.68\\
\hline
\multicolumn{7}{|>{\columncolor{yellow!20}}c|}{ Hexagon-based prismal elements }\\
\cline{1-7}
$n$ & $\|e_{(u,v)} \|_{0}$ & $r=$ & $\|e_{(u,v)}\|_{1}$ & $r=$  &$\|e_{p} \|_{0}$& $r=$  \\
\cline{1-7}
2  &    1.46e-01 &  0.00  &   4.28e-01 &  0.00  &   1.87e-02  & 0.00\\
4  &    3.48e-02 &  2.07  &   1.30e-01 &  1.72  &   5.24e-03  & 1.83\\
8  &    8.80e-03 &  1.98  &   4.25e-02 &  1.61  &   1.81e-03  & 1.54\\
\cline{1-7}
\end{tabular}}
\end{center}
\end{table}

Table \ref{table3D.22} illustrates the numerical performance of the SWG scheme when no tangential stabilities are included in the stabilizer $S(\cdot)$ (i.e., $\kappa=0$). The numerical tests indicate that the stiffness matrices for the SWG scheme on tetrahedral elements are singular so that no numerical results are possible. For other polyhedral partitions such as cubic partitions and hexagon-based prismal partitions, the numerical displacement is seen to converge at the optimal order of $r = 2$ in the $L^2$ norm. For the numerical pressure and for the numerical displacement in $H^1$ norm, the error is at the order of $r\ge 1.5$.

\begin{table}[!h]
\begin{center}
\caption{Test case 4: error and convergence performance of the SWG scheme ($\kappa =0$, no tangential stability) for the elasticity equation on polyhedral partitions. $r$ refers to the order of convergence in $O(h^r).$ }\label{table3D.22}
{\tiny
\begin{tabular}{||c|cc|cc|cc|cc|cc|cc||}
\hline
\multicolumn{7}{|>{\columncolor{mypink}}c|}{ Tetrahedral elements }&\multicolumn{6}{|>{\columncolor{green!15}}c|}{ Cubic elements }\\
\hline
$n$ & $\|e_{(u,v)} \|_{0}$ & $r=$ & $\|e_{(u,v)}\|_{1}$ & $r=$  &$\|e_{p} \|_{0}$& $r=$& $\|e_{(u,v)} \|_{0}$ & $r=$ & $\|e_{(u,v)}\|_{1}$ & $r=$  &$\|e_{p} \|_{0}$& $r=$ \\
\hline
2 & -  & -  & -  & -  & -  & -  &   1.04e-01 &  -     &   3.02e-01  & -     &   8.47e-03 &  -    \\
4 & -  & -  & -  & -  & -  & -  &   2.70e-02 &  1.95  &   9.67e-02  & 1.64  &   3.06e-03 &  1.47\\
8 & -  & -  & -  & -  & -  & -  &   6.79e-03 &  1.99  &   2.90e-02  & 1.74  &   9.56e-04 &  1.68\\
\hline
\multicolumn{7}{|>{\columncolor{yellow!20}}c|}{ Hexagon-based prismal elements }\\
\cline{1-7}
$n$ & $\|e_{(u,v)} \|_{0}$ & $r=$ & $\|e_{(u,v)}\|_{1}$ & $r=$  &$\|e_{p} \|_{0}$& $r=$ \\
\cline{1-7}
2 &    1.45e-01 &  -     &   4.25e-01 &  -     &   1.33e-02 &  -   \\
4 &    3.49e-02 &  2.06  &   1.30e-01 &  1.71  &   4.44e-03 &  1.59\\
6 &    8.77e-03 &  1.99  &   4.19e-02 &  1.63  &   1.53e-03 &  1.53\\
\cline{1-7}
\end{tabular}}
\end{center}
\end{table}

\subsubsection{Shear-loaded beam in 3D}
The goal of this test case is to study the performance of the SWG scheme for the cantilever beam loaded in shear. The domain $\Omega$ for this problem is $\Omega=(-1,1)\times (-1,1)\times(0,L)$, which is occupied by an isotropic material with Young's modulus $E$ and Poisson's ratio $\nu$, and is subject to constant traction, given by $\bm{\varrho} = [0, F, 0]$, on the face passing through the origin. In the present numerical study, the length of the beam is $L = 10$, the shear load is taken as F = $0.1$, and the material properties are selected as $E = 25$ and $\nu = 0.3$.
The Neumann boundary condition is applied on the face $(z=0)$ while the Dirichlet boundary condition is applied to other sides of the domain. The displacement fields $\bm{u}$ corresponding to these stresses, up to the addition of a rigid body motion, is given by
\begin{equation*}
\begin{array}{lll}
\bm{u} =\left(
\begin{array}{c}
-\displaystyle\frac{3F\nu}{4E}xyz\\
 \displaystyle\frac{F}{8E}(3\nu z (x^2 -y^2) -z^3)\\
 \displaystyle\frac{F}{8E}(3yz^2 + \nu y(y^2 - 3x^2)) + \displaystyle\frac{2(1+\nu)}{E}\psi(x,y,z)
\end{array}
\right),\\
\end{array}
\end{equation*}
where $\psi(x,y,z)$, is the anti-derivative of the stress component $\sigma_{23}$  \eqref{eq:stress23} with respect to y.
The expressions for stresses are available in Barber \cite{Barber_2010,Gain_2014} and are quoted here for completeness:
\begin{eqnarray}
\sigma_{11}&=&\sigma_{22}=\sigma_{12}=0,\; \sigma_{33} = \displaystyle\frac{3F}{4}yz,\; \nonumber\\
\sigma_{31}&=&\displaystyle\frac{3F\nu}{2\pi^2(1+\nu)}\sum_{n=1}^{\infty}\frac{(-1)^n}{n^2\cosh(n\pi)}\sin(n\pi x)\sinh(n \pi y),\nonumber\\
\sigma_{23}&=&\displaystyle\frac{3F(1-y^2)}{8} + \frac{F\nu (3x^2-1)}{8(1+\nu)} \label{eq:stress23}\\
           & & \displaystyle-\frac{3F\nu}{2\pi^2(1+\nu)} \sum_{n=1}^{\infty}\frac{(-1)^n}{n^2\cosh(n\pi)}\cos(n\pi x)\cosh(n \pi y). \nonumber
\end{eqnarray}

The numerical simulation is performed on three types of polyhedral partitions: tetrahedral, cubic, and hexagon-based prismal shown as in Figure \ref{fig:testcaseCL_polymesh}. Our numerical experiments (Table \ref{table3D.31}) suggest that the SWG approximations are convergent in $H^1$ and $L^2$ norms.

\begin{figure}[!h]
             \subfigure[Tetrahedral partition ]{\label{Fig.sub10.4.l}
             \includegraphics [width=0.305\textwidth]{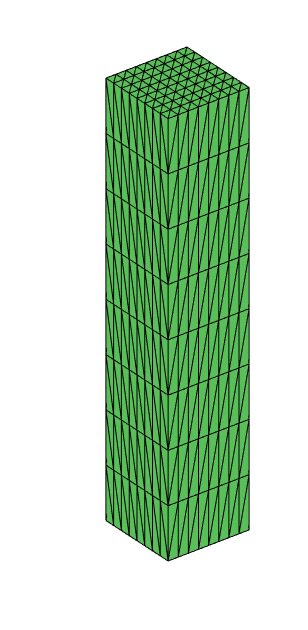}
             }
             \subfigure[Cubic partition]{\label{Fig.sub10.4.2}
             \includegraphics [width=0.325\textwidth]{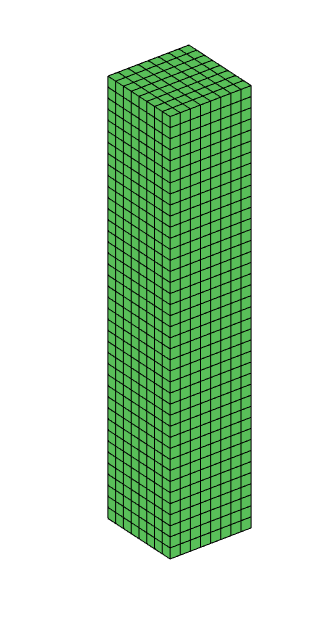}
             }\hskip-10pt
             \subfigure[Hexagon-based prismal mesh]{\label{Fig.sub10.4.3}
             \includegraphics [width=0.305\textwidth]{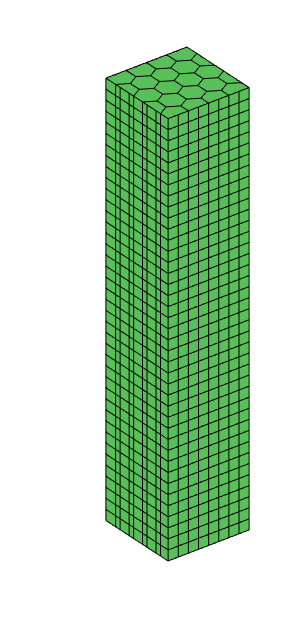}
             }
            \caption{Polyhedral meshes with $h=1/4$ used for the 3D Shear-loaded beam.}
            \label{fig:testcaseCL_polymesh}
\end{figure}

\begin{figure}[!h]
             \subfigure[]{\label{Fig.sub11.4.l}
             \includegraphics [width=0.305\textwidth]{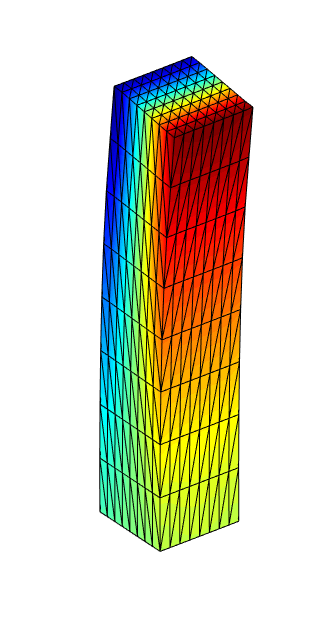}
             }
             \subfigure[]{\label{Fig.sub11.4.2}
             \includegraphics [width=0.305\textwidth]{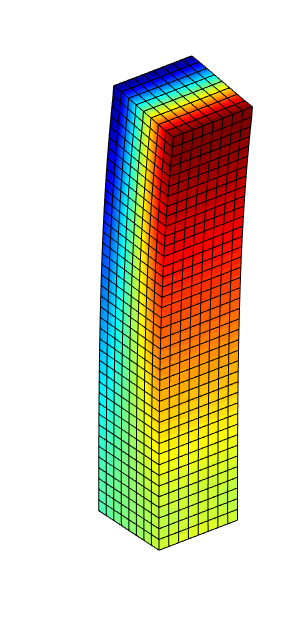}
             }
             \subfigure[]{\label{Fig.sub11.4.3}
             \includegraphics [width=0.305\textwidth]{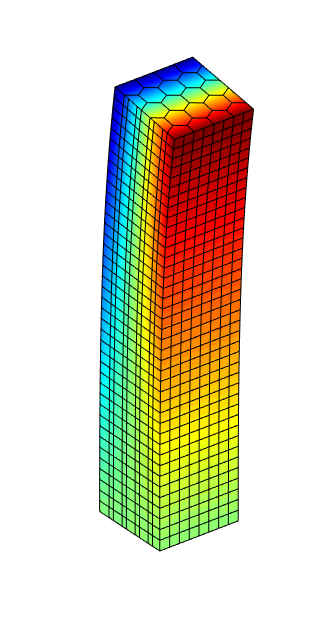}
             }
            \caption{3D Shear-loaded beam: numerical solution from the SWG scheme \eqref{elasticity_mixed1}-\eqref{elasticity_mixed2}, $\kappa=1$ and tangential stability for interior faces only, deformation plot of the elastic material with color provided by the magnitude of the numerical pseudo-pressure $p_h$ under different partitions.}
            \label{fig:testcaseCLdeform_polymesh}
\end{figure}

\begin{table}[!h]
\begin{center}
\caption{Error and convergence performance of the SWG scheme ($\kappa=1$, tangential stability on interior faces only) for the Elasticity equation on polygonal meshes of the shear-loaded beam test case. $r$ refers to the order of convergence in $O(h^r).$ }\label{table3D.31}
{\tiny
\begin{tabular}{||c|cc|cc|cc|cc|cc|cc||}
\hline
\multicolumn{7}{|>{\columncolor{mypink}}c|}{ Tetrahedron }&\multicolumn{6}{|>{\columncolor{green!15}}c|}{ Cube }\\
\hline
$n$ & $\|e_{(u,v)} \|_{0}$ & $r=$ & $\|e_{(u,v)}\|_{1}$ & $r=$  &$\|e_{p} \|_{0}$& $r=$& $\|e_{(u,v)} \|_{0}$ & $r=$ & $\|e_{(u,v)}\|_{1}$ & $r=$  &$\|e_{p} \|_{0}$& $r=$ \\
\hline
2 &5.19e-01 &  -     &   6.00e-01&   -     &   1.27e+00 &  -   &   8.97e-03 &  -     &   3.11e-03 &  -     &   2.36e-02 &  -    \\
4 &1.54e-01 &  1.76  &   2.84e-01&   1.08  &   7.49e-01 &  0.76&   5.48e-03 &  0.71  &   2.32e-03 &  0.43  &   1.49e-02 &  0.66\\
8 &4.38e-02 &  1.81  &   1.54e-01&   0.88  &   4.02e-01 &  0.90&   2.11e-03 &  1.38  &   1.34e-03 &  0.79  &   6.39e-03 &  1.22\\
\hline
\multicolumn{7}{|>{\columncolor{yellow!20}}c|}{ Hexagon-based prismal elements }\\
\cline{1-7}
$n$ & $\|e_{(u,v)} \|_{0}$ & $r=$ & $\|e_{(u,v)}\|_{1}$ & $r=$  &$\|e_{p} \|_{0}$& $r=$  \\
\cline{1-7}
2  &    6.66e-01 &  -     &   3.30e-01 &  -     &   1.25e+00 &  -   \\
4  &    3.46e-01 &  0.95  &   2.22e-01 &  0.57  &   7.36e-01 &  0.76\\
8  &    1.36e-01 &  1.34  &   1.03e-01 &  1.11  &   3.44e-01 &  1.10\\
\cline{1-7}
\end{tabular}}
\end{center}
\end{table}

\vfill\eject

\end{document}